\documentclass[11pt,letterpaper]{amsart}
\usepackage[top=2.5cm, bottom=2.5cm ,left=3cm, right=3cm]{geometry}
\usepackage[T1]{fontenc}
\usepackage{amsmath,amsfonts,amssymb,extarrows}
\usepackage{stmaryrd}
\SetSymbolFont{stmry}{bold}{U}{stmry}{m}{n}
\usepackage{dsfont}
\usepackage[shortlabels]{enumitem}
\usepackage{mathtools}

\usepackage{mathrsfs}
\usepackage{comment}
\usepackage{anyfontsize}

\newcommand{\id}{\mathbf{id}}

\usepackage{xcolor}
\usepackage{rotating}
\usepackage{mathdots}
\usepackage{appendix}

\usepackage{graphicx}
\definecolor{hypercolor}{rgb}{0.5,0,0.5}

\usepackage{tikz}
\usepackage{tikz-cd}
\usepackage{quiver}
\tikzcdset{scale cd/.style={every label/.append style={scale=#1},cells={nodes={scale=#1}}}}

\usepackage{mdframed}
\usepackage{tcolorbox}

\usepackage[style=alphabetic,maxnames=10,giveninits=true,backref=true,url=false]{biblatex}
\numberwithin{equation}{section}
\usepackage[colorlinks,linkcolor=hypercolor,citecolor=hypercolor,unicode]{hyperref}
\usepackage{xurl}

\setitemize[1]{itemsep=5pt}
\setenumerate[1]{itemsep=5pt}

\newtheorem{theorem}{Theorem}[section]
\newtheorem{lemma}[theorem]{Lemma}
\newtheorem{proposition}[theorem]{Proposition}
\newtheorem{corollary}[theorem]{Corollary}

\theoremstyle{definition}
\newtheorem{definition}[theorem]{Definition}
\newtheorem{remark}[theorem]{Remark}
\newtheorem*{question*}{Question}
\newtheorem{example}[theorem]{Example}

\newcommand\cD{\mathcal{D}}

\newcommand\cJ{\mathcal{J}}

\newcommand\cO{\mathcal{O}}

\newcommand\CC{\mathbb{C}}

\newcommand\GG{\mathbb{G}}

\newcommand\QQ{\mathbb{Q}}
\newcommand\RR{\mathbb{R}}

\newcommand\ZZ{\mathbb{Z}}

\newcommand\bD{\mathbf{D}}

\newcommand\bI{\mathbf{I}}
\newcommand\bJ{\mathbf{J}}

\newcommand\bR{\mathbf{R}}

\newcommand\rH{\mathrm{H}}

\newcommand\rT{\mathrm{T}}
\newcommand\rU{\mathrm{U}}

\newcommand\rZ{\mathrm{Z}}

\newcommand\NS{{\rm NS}}

\newcommand\GL{{\rm GL}}

\newcommand{\diag}{ \mathtt{diag}}

\newcommand{\scA}{{\mathbf{A}}}
\newcommand{\scF}{{\mathscr{F}}}
\newcommand{\scE}{{\mathscr{E}}}
\newcommand{\scP}{{\mathscr{P}}}

\newcommand{\scL}{{\mathscr{L}}}
\newcommand{\scH}{{\mathscr{H}}}
\newcommand{\scG}{{\mathscr{G}}}

\newcommand{\Aut}{{\mathtt{Aut}}}

\newcommand{\Br}{{\mathrm{Br}}}

\newcommand{\bXi}{\boldsymbol{\Xi}}
\newcommand{\bUp}{\boldsymbol{\Upsilon}}
\newcommand{\bAd}{\mathbf{Ad}}
\newcommand{\bad}{\mathbf{ad}}
\newcommand{\bvr}{\boldsymbol{\varrho}}

\DeclareMathOperator{\Coh}{Coh}
\DeclareMathOperator{\bdc}{D^b}
\DeclareMathOperator{\Hom}{Hom}
\DeclareMathOperator{\Pic}{Pic}
\DeclareMathOperator{\Ad}{Ad}

\DeclareMathOperator{\chara}{char}
\DeclareMathOperator{\ord}{ord}
\DeclareMathOperator{\image}{im}
\DeclareMathOperator{\Spec}{Spec}
\DeclareMathOperator{\pr}{pr}
\DeclareMathOperator{\Inf}{\mathbf{Inf}}

\DeclarePairedDelimiterX\set[1]\lbrace\rbrace{#1}

\newcommand{\cf}{\textit{cf.~}}

\makeatletter
\newcommand{\msc}[1]{\gdef\@msc{#1}}

\msc{14F08(primary), 14K05(secondary)}

\def\@settitle{\begin{center}%
  \baselineskip14\p@\relax
  \bfseries
  \@title
  \end{center}%
  \ifx\@msc\@empty\else
    \footnotetext[0]{\textbf{2020 Mathematics Subject Classification.} \@msc}%
  \fi
}
\makeatother

\title[Derived isogenies]{Derived isogenies between abelian varieties}

\author{Zhiyuan Li}
\address{Shanghai Center for Mathematical Sciences, Fudan University, 2005 Songhu Road, Shanghai, 200438, China}
\email{\url{zhiyuan\_li@fudan.edu.cn}}

\author{Ziwei Lu}
\address{Shanghai Center for Mathematical Sciences, Fudan University, 2005 Songhu Road, Shanghai, 200438, China}
\email{\url{zwlu21@gmail.com}}

\author{Zhichao Tang}
\address{Shanghai Center for Mathematical Sciences, Fudan University, 2005 Songhu Road, Shanghai, 200438, China}
\email{\url{zctang19@fudan.edu.cn}}

\begin{document}

\begin{abstract}
    The notion of derived isogeny provides a natural framework for comparing smooth projective varieties up to rational cohomological equivalence. In this paper, we establish a derived Torelli Theorem for twisted abelian varieties. Starting from this, we explore the relation between derived isogenies and classical isogenies.  We show that two abelian varieties of dimension $\geq 2$ are derived isogenous if and only if they are principally isogenous over fields of characteristic zero. This generalizes the result for abelian surfaces and completely solves the question raised in \cite{LZ25}.
\end{abstract}

\maketitle

\tableofcontents

\section{Introduction}

\subsection{Derived Torelli theorem}
Let $X$ be an algebraic variety over an algebraically closed field $k$. 
The bounded derived category $\bdc(X)$ of coherent sheaves on $X$ is studied as an important invariant attached to $X$.
For $X$ a smooth projective variety with ample canonical or anticanonical sheaf,
Bondal--Orlov \cite[Theorem 2.5]{BO01} has shown that if \[\bdc(X_1)\simeq \bdc(X_2)\,,\]  then $X_1\cong X_2$.
Therefore, additional, non-geometric derived equivalences can only occur in other cases.
Important examples include hyper-K{\"a}hler varieties, such as K3 surfaces, and abelian varieties.
For abelian varieties, a fundamental result proved by Orlov \cite[Theorem 2.19]{Orl02} and Polishchuk \cite[Theorem 4.3]{Pol96} is that two abelian varieties $X_1$ and $X_2$ are derived equivalent
if and only if there is a symplectic isomorphism (see Definition \ref{def:biextension} below and \cite[Definition 2.17]{Orl02})
\begin{equation*}
    X_1 \times \widehat{X}_1\cong X_2\times \widehat{X}_2\,.
\end{equation*}
This is known as \emph{the derived Torelli theorem} for abelian varieties.

A natural generalization involves the bounded derived category of twisted abelian varieties.
It has been conjectured in \cite[Remark 2.8]{KO03} in the case of $k=\CC$ that there is an analogous result for twisted abelian varieties.
To be precise, let $(X,\alpha)$ be a twisted abelian variety,
we denote by $\bdc(X,\alpha)$ the bounded derived category of $\alpha$-twisted coherent sheaves on $X$.
For such a pair, we can associate a symplectic abelian variety $A_{(X,\alpha)}$ (see Definition \ref{def:polishchuk-orlov})
that plays a key role in characterizing twisted derived equivalences and determining twisted Fourier--Mukai partners.
Notably, there is a natural isogeny
\[\pi\colon X \times \widehat{X}\longrightarrow A_{(X,\alpha)}\]
between $X \times \widehat{X}$ and $A_{(X,\alpha)}$,
which becomes an isomorphism when $\alpha$ is trivial.
In analogy with Orlov's celebrated result, we establish the following twisted version:
\begin{theorem}
\label{thm:main}
Let $(X_1,\alpha_1)$ and $(X_2,\alpha_2)$ be twisted abelian varieties over an algebraically closed field $k$.
Assume that $\chara(k)\neq2$ and that $\ord(\alpha_i)$ is invertible in $k$ for $i=1,2$.
If there exists a derived equivalence
\[
\bdc(X_1,\alpha_1) \simeq \bdc(X_2,\alpha_2)\,,
\]
then there is a symplectic isomorphism
\[
\psi\colon A_{(X_1,\alpha_1)} \xlongrightarrow{\sim} A_{(X_2,\alpha_2)}\,.
\]
Moreover, when $\chara(k) = 0$, the converse holds and thus these two conditions are equivalent.
\end{theorem}

This result constitutes \emph{the twisted derived Torelli theorem} for abelian varieties in characteristic zero.

\begin{remark}
We contextualize our main theorem within the existing literature:
\begin{itemize}
    \item \textbf{Polishchuk's work}~\cite[Theorem 4.3]{Pol96} stated the converse implication (i.e., symplectic isomorphism implies twisted derived equivalence) under the weaker hypothesis $\chara(k) \neq 2$.
    The characteristic assumption needed in the proof was later clarified in~\cite{Pol12}.
    See Theorem~\ref{Thm:Polishchuk} and Remark~\ref{rmk: assumptions-on-Pol-theorem} for more discussion of the additional assumptions in positive characteristic.

    \item \textbf{Other related work} includes \cite[Theorem D]{La24}, which addresses the case where $\alpha_1$ is trivial; a recent preprint \cite{la26} provides a characterization of twisted derived equivalences in terms of isogenies. As explicitly acknowledged in \cite{la26}, obtaining this characterization fundamentally relies on our Theorem \ref{thm:main} to deduce the symplectic isomorphism from the twisted derived equivalence.
    See also \cite{Li25} for some other related results.

    \item By Orlov's result \cite{Orl02}, derived equivalence is preserved under specialization.
    Theorem \ref{thm:main} shows that this also holds for twisted derived equivalence provided that the orders of the Brauer classes are invertible in the base field.
\end{itemize}
\end{remark}

When $k=\CC$, we obtain a Hodge-theoretic criterion for Theorem \ref{thm:main}, confirming the conjecture in \cite[Remark 2.8]{KO03} (see also \cite[\S 13.4]{Huy06}). This conjecture has drawn considerable attention, and a number of subsequent works have been devoted to this problem (see, e.g., \cite{La24,la26,Li25,LZ25}).
Recall that for a complex abelian variety $X$, the lattice $\rH_1(X\times\widehat{X},\mathbb{Z})=\rH_1(X,\mathbb{Z})\oplus\rH_1(\widehat{X},\mathbb{Z})$ comes with a natural bilinear form by the dual pairing:
\[
q_X\bigl((v_1,\theta_1),(v_2,\theta_2)\bigr)=\theta_2(v_1)+\theta_1(v_2)\,.
\]
\vspace{-1.3\baselineskip}

\begin{theorem}
\label{thm:TDE}
Assume $k = \mathbb{C}$. Then $\bdc(X_1,\alpha_1) \simeq \bdc(X_2,\alpha_2)$ if and only if there is an isometry of integral Hodge structures
\[
\mathrm{H}_1(X_1,\mathbb{Z}) \oplus \mathrm{H}_1(X_1,\mathbb{Z})^* \cong \mathrm{H}_1(X_2,\mathbb{Z}) \oplus \mathrm{H}_1(X_2,\mathbb{Z})^* \,,
\]
where the Hodge structures are endowed with complex structure operators:
\[
\cJ_{\alpha_i} \coloneq \begin{pmatrix}
    J_i & 0 \\
    B_{\alpha_i}J_i + J_i^\top B_{\alpha_i} & -J_i^\top
\end{pmatrix}, \quad i=1,2 \text{ respectively.}
\]
Here, $J_i$ denotes the complex structure on $\mathrm{H}_1(X_i,\mathbb{Z})$ and $B_{\alpha_i}$ is the \textbf{B}-field associated to $\alpha_i$.
\end{theorem}

As an application, we also answer a question posed by Huybrechts. Indeed, we prove in Theorem~\ref{thm:no-dual-equiv-general} that for any dimension $g \ge 2$, general twisted abelian varieties are not twisted derived equivalent to their duals, regardless of the Brauer class chosen on the dual.

\subsection{Derived Isogenies}

The notion of isogeny between abelian varieties is a natural equivalence relation in the study of abelian varieties. In recent years, this classical concept has found a natural generalization in the derived categorical setting. Building on the work of Huybrechts \cite{Huy19} and \cite{LZ25}, we introduce the following definition of derived isogeny via twisted derived equivalences:

\begin{definition}
    Two abelian varieties $X$ and $Y$ are \emph{derived isogenous} if they can be connected by derived equivalences between twisted abelian varieties,
    i.e., there exist $\alpha\in\Br(X)$, $\beta\in\Br(Y)$, and twisted abelian varieties $(X_i,\alpha_i)$ and $(X_i,\beta_i)$ for $i = 1,\dots,n$,
    such that there is a sequence of derived equivalences as illustrated in the following diagram.
    \begin{equation}
    \label{eq:zigzagtwisted}
    \begin{tikzcd}[row sep=tiny,column sep=small]
    \bdc(X,\alpha) \ar[r, "\simeq"]
        & \bdc(X_1,\beta_1) \\
        & \bdc(X_1,\alpha_1) \ar[r, "\simeq"]
            & \bdc(X_2,\beta_2) \\
        &   & \vdots &  \\
        &   & \bdc(X_{n-1},\alpha_{n-1}) \ar[r, "\simeq"]
            & \bdc(X_n,\beta_n) \\
        &   &   & \bdc(X_n,\alpha_{n}) \ar[r, "\simeq"]
            & \bdc(Y,\beta)\mathrlap{\,.}
    \end{tikzcd}
    \end{equation}
\end{definition}

In fact, the notion of derived isogeny can be defined for arbitrary smooth projective varieties. However, when the dualizing sheaf $\omega_X$ is ample or antiample, the theorems of Calabrese and of Bondal--Orlov show that derived isogenous varieties are isomorphic (\cf\cite[Theorem B]{Cal18}). The situation becomes much more complicated when $\omega_X$ is trivial. 

This framework is deeply inspired by analogous developments for K3 surfaces and abelian surfaces. It suggests that derived isogeny provides a more natural relation between varieties as it compares varieties up to rational cohomological equivalence. In particular, Huybrechts \cite[Theorem 0.1]{Huy19} established that two K3 surfaces are isogenous (in the sense of rational Hodge isometries) if and only if they are connected by a sequence of twisted derived equivalences. The definition above adapts this philosophy to abelian varieties, where the role of the lattice is naturally played by the symplectic structure on $X\times \widehat{X}$.

A natural question arises: How does the notion of \emph{derived isogeny} relate to the classical notion of \emph{isogeny}?
Although defined in quite different terms, these concepts are not independent. Indeed, using the twisted derived Torelli theorem, one can show that twisted Fourier--Mukai equivalences naturally induce isogenies between the underlying abelian varieties. This yields the following implication:
\begin{center}
\tcbox[colback=white]{Derived isogenous $\implies$ isogenous}
\end{center}
However, the converse does not hold in general. Counterexamples already exist among elliptic curves and abelian surfaces.

In this article, we provide a complete characterization of derived isogeny classes of abelian varieties. The key tool is the notion of \emph{principal isogenies}, introduced in \cite[\S 1.2]{LZ25}, which allows us to precisely describe when two isogenous abelian varieties are further derived isogenous.

\begin{definition}
Two abelian varieties $X$ and $Y$ are \emph{principally isogenous} if there exists an isogeny $f \colon X \to Y$ whose degree is a \emph{square}.
\end{definition}

The following theorem extends the main result of \cite[Theorem 1.2.1]{LZ25} for abelian surfaces to abelian varieties of arbitrary dimension.

\begin{theorem}
\label{thm:main-2}
Let $k$ be an algebraically closed field of characteristic zero,
and $X$, $Y$ be abelian varieties over $k$ of dimension $g$.
Then $X$ and $Y$ are derived isogenous if and only if
\begin{enumerate}
   \item  $X$ and $Y$ are principally isogenous when $g > 1$;
   \item $X\cong Y$ when $g = 1$.
\end{enumerate}
\end{theorem}

We remark that the second statement follows from two well-established facts: for curves, derived equivalences imply isomorphism, and Brauer groups are trivial.
See Huybrechts \cite[Corollary 5.46]{Huy06} for the details of the first fact.

\medskip

When $\chara(k)=p>2$, our strategy is still valid for \emph{prime-to-$p$ principal isogenies}.
\begin{theorem}\label{thm3}
Let $k$ be an algebraically closed field of $\chara(k)=p>2$, and $X$, $Y$ be abelian varieties over $k$ of dimension $g\geq2$. If there is a prime-to-$p$ principal isogeny between $X$ and $Y$, then $X$ and $Y$ are derived isogenous.
\end{theorem}

As an immediate application, we obtain that derived isogeny is also preserved under specialization:

\begin{corollary}\label{cor:speciliazation}
Let $R$ be a discrete valuation ring (DVR) with residue field $k$ and fraction field $K$.
Assume that $\mathrm{char}(k)=0$.
Let $\mathcal{X} \to \Spec R$ and $\mathcal{Y} \to \Spec R$ be two families of abelian varieties.
If the generic fibers $\mathcal{X}_K$ and $\mathcal{Y}_K$ are derived isogenous, then so are the special fibers $\mathcal{X}_k$ and $\mathcal{Y}_k$.
\end{corollary}

\begin{remark}{(Derived isogeny over positive characteristic fields)}
It is very interesting to know whether this holds when $\chara(k)>0$, especially for supersingular abelian varieties.
This needs a new approach as one has to deal with $p$-divisible derived isogenies.
The case of  abelian surfaces has been settled by Li and Zou \cite[Theorem 1.4.1]{LZ25}. 
We also remark that Bragg and Yang have studied isogenies between K3 surfaces over positive characteristic fields (\cf\cite[Theorem 1.2]{BY23}).
\end{remark}
A second application concerns derived isogenies between K3 surfaces and their associated Kuga--Satake varieties. It is shown in \cite{Huy19} that two complex projective K3 surfaces are derived isogenous if and only if there exists a rational Hodge isometry between their second rational cohomology groups. Combined with Theorem \ref{thm:main-2}, this implies the following.

\begin{corollary}\label{cor:KS-isogeous}
If two complex projective K3 surfaces are derived isogenous, then their associated Kuga--Satake abelian varieties are also derived isogenous, provided their dimension is at least $2$.
\end{corollary}

Let us briefly recall the construction. For a complex projective K3 surface $S$ let $V$ denote the rational transcendental lattice, i.e., the orthogonal complement in $\mathrm{H}^2(S, \mathbb{Z})$ of the Néron--Severi group, equipped with its natural weight‑$2$ Hodge structure. Consider the even Clifford algebra $C^+(V)$ attached to the quadratic form on $V$. The Hodge structure on $V$ induces a Hodge structure of weight $1$ on $C^+(V_{\mathbb{R}})$. Explicitly, choose an oriented orthonormal basis $(e_1, e_2)$ of the real subspace $(V^{-1,1} \oplus V^{1,-1})$, set $e_+ = e_1 e_2 \in C^+(V_{\mathbb{R}})$, and define a complex structure on $C^+(V_{\mathbb{R}})$ by left multiplication by $e_+$. The associated Kuga--Satake variety of $S$ is the complex torus
\[
\operatorname{KS}(S) \coloneq C^+(V_{\mathbb{R}})\,/\,C^+(V)\,.
\]
For full details we refer to the original works of Satake \cite{Satake_1966} and Kuga--Satake \cite{zbMATH03350035}, and to the exposition by van Geemen \cite{zbMATH01445141}.

\begin{remark}
Corollary \ref{cor:KS-isogeous} remains valid over arbitrary algebraically closed fields, provided that the associated Kuga--Satake variety exists. For example, this is always the case when the Picard number of the K3 surface is at most $18$. See Remark \ref{rmk:KSexists} for more details. 
\end{remark}

\subsection{Idea of the Proof}

The difficult direction of Theorem~\ref{thm:main} is to deduce the existence of a symplectic isomorphism $A_{(X_1,\alpha_1)} \cong_{\text{symplectic}} A_{(X_2,\alpha_2)}$ from the given derived equivalence $\bdc(X_1,\alpha_1) \simeq \bdc(X_2,\alpha_2)$. Our approach generalizes Orlov's method from the untwisted case, with the key challenge being the identification of the appropriate twist data for the symplectic abelian variety. The proof proceeds through the following steps:

\subsubsection*{1. Construction of Twisted Orlov Functor}

In the untwisted case, for an abelian variety $X$, there is a natural derived equivalence between $A_X = X \times \widehat{X}$ and $X \times X$ induced by the Poincar\'e line bundle (\cf\cite[Theorem 2.10]{Orl02}).

For a twisted abelian variety $(X,\alpha)$, we establish a natural twisted derived equivalence:
\[
\bXi_{X,\alpha}\colon\bdc(A_{(X,\alpha)}, q^*_X\alpha) \simeq \bdc(X \times X, \alpha^{-1} \boxtimes \alpha)\,.
\]
This construction is based on the theory of the equivariant category and provides the fundamental bridge between the twisted derived category and the equivariant derived category.

\subsubsection*{2. Action of Autoequivalence}

For any $z\in A_{(X,\alpha)}$, one can associate a natural autoequivalence
\[
\bdc(X,\alpha)\longrightarrow\bdc(X,\alpha)
\]
obtained by translation and tensor product with line bundles,
and there is an associated adjoint equivalence
\[ \bad_z\colon \bdc(X\times X, \alpha^{-1}\boxtimes \alpha)\longrightarrow \bdc(X\times X, \alpha^{-1}\boxtimes \alpha)\,.\]
A crucial property between $\bad_z$ and the equivalence $\bXi_{X,\alpha}$ in \emph{Step 1} is that the induced adjoint autoequivalence
\[
\bXi^{-1}_{X,\alpha} \circ \bad_z \circ \bXi_{X,\alpha}\colon\bdc(A_{(X,\alpha)}, q^*_X\alpha) \longrightarrow \bdc(A_{(X,\alpha)}, q^*_X\alpha)
\]
is always given by the tensor product with a line bundle.

\subsubsection*{3. Construction of Symplectic Equivalences}

More generally, given any derived equivalence $\Phi\colon \bdc(X_1,\alpha_1) \simeq \bdc(X_2,\alpha_2)$, there is an induced adjoint equivalence:
\[
\bAd_{\Phi}\colon \bdc(X_1\times X_1, \alpha^{-1}_1 \boxtimes \alpha^{\vphantom{1}}_1) \longrightarrow \bdc(X_2\times X_2, \alpha_2^{-1} \boxtimes \alpha^{\vphantom{2}}_2)\,.
\]
Using the computation in \emph{Step 2}, we can show that the composition
\[
\bXi^{-1}_{X_2,\alpha_2}\circ\bAd_{\Phi}\circ\bXi_{X_1,\alpha_1}\colon
\bdc(A_{(X_1,\alpha_1)}, q_{X_1}^\ast \alpha_1 )\longrightarrow\bdc(A_{(X_2,\alpha_2)}, q_{X_2}^\ast \alpha_2)
\]
is induced by a symplectic isomorphism. This establishes the desired symplectic isomorphisms of their associated symplectic abelian varieties.

For the proof of Theorem \ref{thm:main-2} and Theorem \ref{thm3}, we adopt a strategy analogous to the surface case. Specifically, we decompose the given principal isogeny into a sequence of simplest isogenies, each induced by a twisted derived equivalence. In the context of abelian surfaces, such a decomposition, coming from the Cartan--Dieudonné decomposition for orthogonal transformations, can be realized via the second cohomology, where the simplest isogenies correspond to symmetries. This approach is rooted in Shioda's trick, which is only valid for abelian surfaces (cf.~\cite{LZ25}). In the present work, however, we provide a decomposition from the perspective of the first cohomology.

The key idea is that any principal isogeny $X \to Y$ can be factored into a composition of what we call \emph{spectrally paired isogenies}:
\[
X = X_0 \to X_1 \to X_2 \to \cdots \to X_{n-1} \to X_n \to X_{n+1} = Y\,,
\]
such that each step $X_i \to X_{i+1}$ is induced by a twisted derived equivalence between $X_i$ and $X_{i+1}$.
Details will be given in \S \ref{sec:derived-isogeny}.

\subsection{Conventions}

We adopt the following conventions throughout this article:

\begin{itemize}
    \item For schemes $X_1, \dots, X_n$ and a subset $I \subseteq \{1, \dots, n\}$, let
    \[
    \pr_I: \prod_{i=1}^n X_i \to \prod_{j \in I} X_j
    \]
    denote the natural projection morphism.

    \item  To simplify notation in the computations, we do not distinguish between derived functors and their underived counterparts. Unless otherwise specified, functors between categories of coherent sheaves (such as pullback and pushforward) use the same notation for both derived and underived versions.

    \item For an abelian variety $X$, let $\widehat{X}$ denote its dual abelian variety and $\mathscr{P}_X$ the Poincaré bundle on $X \times \widehat{X}$. For any $\xi \in \widehat{X}$, let $\mathscr{P}_{X,\xi}$ denote the line bundle on $X$ corresponding to $\xi$. When clear from context, we write $\mathscr{P}$ (resp. $\mathscr{P}_\xi$) instead of $\mathscr{P}_X$ (resp. $\mathscr{P}_{X,\xi}$) for brevity.
\end{itemize}

\subsection*{Acknowledgments}
The authors would like to thank Ruxuan Zhang, Haitao Zou, Zaiyuan Chen and Julian Holstein for helpful discussions and valuable comments on an earlier draft of this work. 
Special thanks go to Daniel Huybrechts for posing the question of whether a  derived equivalence exists between a twisted abelian surface and its dual.
The authors also thank the referees for a careful reading of the manuscript and for helpful comments and suggestions.
Z.~Li is supported by the NSFC grants (No.~12171090 and No.~12425105) and the Shanghai Pilot Program for Basic Research (No.~21TQ00).
Z.~Li is also a member of LMNS.
Z.~Lu and Z.~Tang are supported by the NSFC grant (No.~12121001).

\section{Preliminary on twisted and equivariant derived categories}
\label{sec:basics}

In this section, we will review the theory of equivariant categories and discuss its relation with the twisted derived categories. The main point is to recall how twisted derived categories, particularly for abelian varieties, can be described in terms of equivariant categories—a perspective that will greatly simplify many of the constructions to come.

\subsection{Twisted Sheaves}
\label{subsec:twisted_sheaves}

Let $(X, \alpha)$ be a twisted smooth projective variety over an algebraically closed field $k = \overline{k}$ with $\alpha \in \Br(X)$ of order $n$.

\begin{definition}
    Let \(\{\alpha_{ijk}\}_{i,j,k\in I}\) be a {\v C}ech \(2\)-cocycle representing $\alpha\in \mathrm{H}^2(X, \mathbb{G}_m)$ with respect to an \'etale covering \(\{U_i\}_{i \in I}\) of \(X\).
    An \emph{\(\alpha\)-twisted coherent sheaf} on \(X\) consists of the following data:
    \begin{itemize}
      \item a coherent sheaf \(\mathscr{F}_i\) on \(U_i\),
      \item an isomorphism \(\theta_{ij} \colon \mathscr{F}_j|_{U_{ij}} \xrightarrow{\,\sim\,} \mathscr{F}_i|_{U_{ij}}\),
    \end{itemize}
    satisfying the following conditions for all indices:
    \begin{enumerate}[label=(\roman*)]
      \item \(\theta_{ii} = \id_{\mathscr{F}_i}\),
      \item \(\theta_{ji} = \theta_{ij}^{-1}\),
      \item \(
            \theta_{ki} \circ \theta_{ij} \circ \theta_{jk} = \alpha_{ijk} \cdot \id_{\mathscr{F}_k|_{U_{ijk}}}
            \) on \(U_{ijk}\).
    \end{enumerate}
\end{definition}

The category of $\alpha$-twisted coherent sheaves is denoted $\Coh(X,\alpha)$,
with bounded derived category $\bdc(X,\alpha)$,
and up to equivalence they are independent of the choice of the {\v C}ech cocycle \(\{\alpha_{ijk}\}_{i,j,k\in I}\) that represents $\alpha$.

In many situations, there is a natural connection between twisted categories and equivariant categories. While twisted categories can be complicated to work with directly, their description via equivariant categories often provides a more accessible framework for computations.

\subsection{Equivariant Categories}
\label{subsection:equivariant_categories}

We now review the formal framework for group actions on categories and their equivariant counterparts.
Let $G$ be a finite group acting on a $k$-linear additive category $\cD$ through data $(\rho, \sigma)$.
Following Beckmann--Oberdieck \cite[\S 2.1]{BO23}, such a \emph{group action} consists of two collections of data:

\begin{enumerate}[label=(\roman*)]
    \item For each \(g \in G\), an \emph{autoequivalence} \(\rho_g \colon \cD \to \cD\),
    \item For each pair $g$, $h \in G$, a \emph{natural isomorphism} \(\sigma_{g,h} \colon \rho_g \circ \rho_h \to \rho_{gh}\) satisfying the \emph{cocycle condition}:
    for each triple $g$, $h$, $k \in G$, the diagram commutes:
    \[\begin{tikzcd}[sep=large]
        \rho_g \circ \rho_h \circ \rho_k \ar[r, "\rho_g(\sigma_{h,k})"] \ar[d, "\sigma_{g,h} \circ \rho_k"'] &
        \rho_g \circ \rho_{hk} \ar[d, "\sigma_{g,hk}"] \\
        \rho_{gh} \circ \rho_k \ar[r, "\sigma_{gh,k}"'] &
        \rho_{ghk}\mathrlap{\,.}
    \end{tikzcd}\]
\end{enumerate}
These data define how the group $G$ acts on both objects and compositions of autoequivalences. And a $G$-functor
\[
  (\Phi,\Sigma)\colon (\cD,\rho,\sigma) \to (\cD',\rho',\sigma')
\]
between categories with $G$-actions consists of the following data:

\begin{enumerate}[label=(\roman*)]
    \item A functor $\Phi\colon\cD\to\cD'$,
    \item For each $g\in G$, a natural isomorphism $\Sigma_g\colon \Phi\circ\rho_g\to\rho'_g\circ\Phi$ making the following diagram commute:
    \[
    \begin{tikzcd}
	{\Phi\circ\rho_g\circ\rho_h} & {\rho'_g\circ\Phi\circ\rho_h} & {\rho'_g\circ\rho'_h\circ\Phi} \\
	{\Phi\circ\rho_{gh}} && {\rho'_{gh}\circ\Phi}\mathrlap{\,.}
	\arrow["{\Sigma_g\rho_h}", from=1-1, to=1-2]
	\arrow["{\Phi\sigma_{g,h}}"', from=1-1, to=2-1]
	\arrow["{\rho'_g\Sigma_h}", from=1-2, to=1-3]
	\arrow["{\sigma'_{g,h}\Phi}", from=1-3, to=2-3]
	\arrow["{\Sigma_{gh}}", from=2-1, to=2-3]
    \end{tikzcd}
    \]
\end{enumerate}

To study invariant objects under this action, we introduce:

\begin{definition}[Equivariant Category]
\label{def:equivariant_category}
The \emph{equivariant category} \(\cD_{G,\rho,\sigma}\) consists of
\begin{enumerate}[label={\bfseries(\roman*)}]
    \item \textbf{Objects:} An object of  \(\cD_{G,\rho,\sigma}\) is a pair \((E, \theta)\) where \(E\) is an object of $\cD$ and \(\theta = \{\theta_g \colon \rho_g(E) \to E\}_{g \in G}\) is a collection of isomorphisms, satisfying the compatibility condition: for all \(g, h \in G\), the following diagram commutes:
    \[\begin{tikzcd}[sep=large]
        \rho_g(\rho_h (E)) \ar[r, "\rho_g(\theta_h)"] \ar[d, "\sigma_{g,h}^E"'] &
        \rho_g(E) \ar[d, "\theta_g"] \\
        \rho_{gh}(E) \ar[r, "\theta_{gh}"'] & E \mathrlap{\,.}
    \end{tikzcd}\]
    Such \(\theta\) is called a \emph{\((\rho,\sigma)\)-linearization} of \(E\).

    \item \textbf{Morphisms:} A morphism \(\phi\colon (E,\theta) \to (E',\theta')\) is a morphism \(\phi\colon E \to E'\) in \(\cD\) satisfying:
    \[
    \phi \circ \theta_g = \theta'_g \circ \rho_g(\phi) \quad \forall g \in G\,.
    \]
    Equivalently, morphisms in the equivariant category are invariant under the $G$-action:
    \[
    g \cdot \phi \coloneq \theta'_g \circ \rho_g(\phi) \circ \theta_g^{-1}\,.
    \]
\end{enumerate}

Given a $G$-functor $(\Phi, \Sigma)\colon (\cD, \rho, \sigma) \to (\cD', \rho', \sigma')$, there is an induced functor between the corresponding equivariant categories.
If $\Phi$ is an equivalence, then the induced functor is also an equivalence.
See~\cite[Section 3]{Shin18} for more details.

\paragraph{\bfseries Geometric Convention:} When \(\cD\) is a category arising from geometry (e.g., \(\cD = \Coh(Y)\) or \(\bdc(Y)\) for a scheme \(Y\)) and the group action is induced by an untwisted geometric action on $Y$ (i.e., $\sigma$ is trivial), we simplify the notation to \(\cD_G\).

\paragraph{\bfseries Inflation and Forgetful Functors:} A standard way to construct equivariant objects is provided by the \emph{Inflation functor} (\cf\cite[\S 3.1]{Plo05}\cite[\S 1.4]{Plo07}):
\[
\Inf_{\rho,\sigma} \colon \cD \longrightarrow \cD_{G,\rho,\sigma},\quad E \longmapsto \Bigl( \bigoplus_{g\in G} \rho_g(E),  \Theta \Bigr)\,,
\]
where the linearization \(\Theta = \{\Theta_h\}_{h\in G}\) is defined component-wise for each \(h \in G\):
\[
\Theta_h \colon \rho_h \Bigl( \bigoplus_{g\in G} \rho_g (E) \Bigr) \xlongrightarrow{\ \cong\ } \bigoplus_{g\in G} \rho_h(\rho_g (E)) \xlongrightarrow{\ \oplus_g\, \sigma_{h,g}^E\ } \bigoplus_{g\in G} \rho_{hg} (E)\,.
\]
Conversely, one also has the \emph{forgetful functor}
\[
\bUp\colon \mathcal{D}_{G,\rho,\sigma}\longrightarrow\mathcal{D},\quad(F,\Theta)\longmapsto F\,,
\]
which omits the linearization.
\end{definition}

The equivariant category behaves well under (normal) subgroup restrictions:

\begin{proposition}[{\cite[Proposition 3.3]{BO23}}]
\label{prop:successive_equivariant}
Let \(H \trianglelefteq G\) be a normal subgroup.
A \(G\)-action \((\rho,\sigma)\) on \(\cD\) naturally induces a \(G\)-action on \(\cD_{H,\rho|_H,\sigma|_H}\), making the forgetful functor \[ \cD_{H,\rho|_H,\sigma|_H} \longrightarrow \cD\] \(G\)-equivariant.
Moreover, there is a canonical equivalence:
\[
\cD_{G,\rho,\sigma} \simeq \bigl( \cD_{H,\rho|_H,\sigma|_H} \bigr)_{G/H,\,\overline{\rho},\,\overline{\sigma}}\,,
\]
where \(\overline{\rho}\), \(\overline{\sigma}\) denote the induced action and the composition isomorphisms of the quotient group \(G/H\).
\end{proposition}

\noindent
This hierarchical structure allows one to analyze group actions through successive quotients, which is particularly useful for Galois coverings.

\subsection{Equivariant Categories for Galois Covers}
\label{subsection:galois_equivariant}

 We now specialize to Galois coverings between varieties.
 Let $Y$ be a smooth projective variety over $k$ with a finite Galois covering $\pi \colon  Y \to X = Y/G$, where $G$ is the Galois group of $\pi$.
Consider the \emph{right} $G$-action on $Y$,
\[
R_g \colon  Y \xlongrightarrow{\sim} Y, \quad y \longmapsto y \cdot g, \quad \forall g \in G\,,
\]
this geometric action induces a \emph{left} $G$-action $(\boldsymbol{\varrho}, \id)$ on $\Coh(Y)$ via pullback:
\[
\bvr_g \coloneq R_g^* \colon \Coh(Y) \xlongrightarrow{\sim} \Coh(Y), \quad \bvr_g \circ \bvr_h = \bvr_{gh} \quad \forall g, h \in G\,,
\]
where the composition isomorphisms $\sigma_{g,h} = \id$ are trivial.
As in Definition \ref{def:equivariant_category}, the equivariant category $\Coh(Y)_{G}$ admits a canonical geometric interpretation (\cf\cite[Example 3.4]{BO23}):
\[
\Coh(Y)_{G} \simeq \Coh(Y/G) = \Coh(X)\,.
\]

The action extends naturally to the derived category $\bdc(Y)$, yielding the equivariant category $\bdc(Y)_{G}$.
When $|G|$ is invertible in $k$, Elagin \cite[Theorem 1.1]{Ela15} establishes the equivalences:
\[
\bdc(Y)_{G} \simeq \bdc\bigl(\Coh(Y)_{G}\bigr) \simeq \bdc(X)\,.
\]
\medskip
We now generalize to the case of \emph{twisted actions} where the composition isomorphisms are no longer assumed to be trivial.
Since we have
\[
\Hom\bigl(\id_{\Coh(Y)}, \id_{\Coh(Y)}\bigr) = k\,,
\]
the composition isomorphisms are parameterized by normalized $2$-cocycles $\sigma \in \rZ^2(G, \mathbb{G}_m)$ (\cf\cite[Theorem 2.1]{BO23}):

\begin{definition}[Twisted Galois Action]
For any $\sigma \in \rZ^2(G, \mathbb{G}_m)$, the \emph{$\sigma$-twisted $G$-action} $(\bvr, \sigma)$ on $\Coh(Y)$ consists of the following data:
\begin{itemize}
    \item Autoequivalences $\bvr_g = R_g^* \colon \Coh(Y) \to \Coh(Y)$ (same as untwisted case) for any $g\in G$,
    \item Natural isomorphisms for every pair $g, h \in G$:
    \[
    \sigma_{g,h} \colon \bvr_g \circ \bvr_h \xlongrightarrow{\, \sigma_{g,h} \cdot \id_{\Coh(Y)} \,} \bvr_{gh}\,,
    \]
    satisfying $\sigma_{g,hk} \cdot \sigma_{h,k} = \sigma_{g,h} \cdot \sigma_{gh,k}$.
\end{itemize}
The resulting equivariant category (see Definition \ref{def:equivariant_category}) is denoted $\Coh(Y)_{G,\bvr,\sigma}$.
\end{definition}

This construction generalizes the untwisted case, which corresponds to $\sigma_{g,h} \equiv 1$. A natural question arises:
\begin{question*}
     Do $\Coh(Y)_{G,\rho,\sigma}$ and $\bdc(Y)_{G,\rho,\sigma}$ admit a geometric interpretation via the quotient variety $X = Y/G$?
\end{question*}

The following result provides an affirmative answer by establishing a concrete connection to the geometry of the quotient map.

\begin{proposition}[Twisted Equivariant-Quotient Equivalence]
\label{prop:twisted-equivariant-equivalence}
There exists a canonical homomorphism
\[\mathrm{H}^2(G,\mathbb{G}_m) \longrightarrow \Br(X)\,.\]
Moreover, for any cohomology class $[\sigma] \in \mathrm{H}^2(G,\mathbb{G}_m)$ with image $\alpha \in \Br(X)$, there are canonical equivalences:
    \[
    \Coh(X,\alpha) \simeq \Coh(Y)_{G,\bvr,\sigma}\,.
    \]
When $|G|$ is invertible in $k$, one also obtains
$\bdc(X,\alpha) \simeq \bdc(Y)_{G,\bvr,\sigma}.$
\end{proposition}

\begin{proof}
Since $Y\to X$ is a finite \'etale covering,
the group cohomology $\rH^2(G,\mathbb{G}_m)$ is canonically isomorphic to the {\v C}ech cohomology $\check{\rH}^2({Y\to X},\mathbb{G}_m)$ (\cf\cite[Example III.2.6]{Mil80}).
After composing with the map from the {\v C}ech cohomology to the \'etale cohomology $\rH^2(X,\mathbb{G}_m)$,
we get the canonical map
\[
\mathrm{H}^2(G,\mathbb{G}_m) \cong \check{\mathrm{H}}^2(\{Y \to X\},\mathbb{G}_m) \longrightarrow \mathrm{H}^2(X,\mathbb{G}_m)\,.
\]
Note that there are natural isomorphisms:
  \begin{align*}
    &Y \times G \cong Y \times_X Y,& &(y, g)\longmapsto (y, yg)\,, \\
    &Y \times G \times G \cong Y \times_{X} Y \times_{X} Y,& &(y, g, h) \longmapsto (y, yg, ygh)\,.
  \end{align*}
These identifications allow us to interpret:
    \begin{itemize}
        \item The cocycle $\sigma$ as descent data for the covering;
        \item The $\sigma$-twisted $G$-equivariant sheaves as $\alpha$-twisted sheaves on $X$.
    \end{itemize}
These interpretations yield a canonical equivalence
\[
 \Coh(X,\alpha)\simeq \Coh(Y)_{G,\bvr,\sigma}\,,
\]
and the corresponding equivalence between the derived categories follows from \cite[Theorem 7.1]{Ela15}.
\end{proof}

\subsection{Twisted Sheaves on Abelian Varieties}
\label{subsec:twisted_abelian}

Let $(X,\alpha)$ be a twisted abelian variety with $\alpha \in \Br(X)[n]$. When $n$ is invertible in $k$, Proposition \ref{prop:twisted-equivariant-equivalence} interprets $\bdc(X,\alpha)$ as an equivariant category.
The following construction is from~\cite{Ber72}.

Let $[n]_X \colon X\to X$ denote the multiplication-by-$n$ homomorphism. The Kummer exact sequence induces:
\begin{equation}
\label{eq:lifting-Brauer}
0 \longrightarrow \mathrm{NS}(X)/n\mathrm{NS}(X) \xlongrightarrow{\iota} \mathrm{Hom}(\wedge^2 X[n], \mu_n) \xlongrightarrow{\delta} \text{Br}(X)[n] \longrightarrow 0.
\end{equation}
Thus $\alpha$ is represented by an alternating bilinear pairing
\[
e_\alpha \colon X[n] \times X[n] \longrightarrow \mu_n,
\]
which is unique modulo commutator forms $e^{L^{\otimes n}}|_{X[n]}$ for $L \in \Pic(X)$.
See~\cite[\S 23]{Mum70} for the explicit definition of the commutator form.

Via the Weil pairing, $e_\alpha$ corresponds to a skew-symmetric homomorphism
\begin{equation}
\label{eq:twist-hom}
\phi_\alpha \colon X[n] \longrightarrow \widehat{X}[n]\,,
\end{equation}
satisfying
\[
e_\alpha(\sigma_1,\sigma_2) = \langle \sigma_1, \phi_\alpha(\sigma_2) \rangle\,,
\]
where $\langle -, - \rangle\colon X[n] \times \widehat{X}[n] \to \mu_n$ is the restriction of the canonical Weil pairing.
And a homomorphism $\phi\colon X[n]\to\widehat{X}[n]$ is called \emph{skew-symmetric} if
\[
\langle \sigma_1, \phi(\sigma_2) \rangle = \langle \sigma_2, \phi(\sigma_1) \rangle^{-1}\quad\forall\sigma_1, \sigma_2\in X[n]\,.
\]
By the surjectivity of the map
\[
\rH^2(X[n], \mathbb{G}_m) \longrightarrow \Hom(\wedge^2 X[n], \mu_n)\,,
\]
which sends a $2$-cocycle to its commutator form,
there exists a normalized $2$-cocycle $a \in \rZ^2(X[n], \mathbb{G}_m)$ such that
\begin{equation}
\label{eq:pairing-cocycle}
e_\alpha(\sigma_1,\sigma_2) = \frac{a_{\sigma_1,\sigma_2}}{a_{\sigma_2,\sigma_1}}\,;
\end{equation}
see~\cite[Proposition 2]{Ber72}.
We then say $\alpha$ is represented by the $2$-cocycle $a$, and twisted sheaves can be described in terms of equivariant sheaves.  By Proposition \ref{prop:twisted-equivariant-equivalence}, there is a natural twisted $X[n]$-action $(\bvr, a)$  on $\Coh(X)$ and we have

\begin{corollary}
\label{cor:av-twisted-equivariant-equivalence}
Assume that $n$ is invertible in $k$. For any $2$-cocycle $a\in\rZ^2(X[n],\mathbb{G}_m)$ representing the Brauer class $\alpha$, the finite \'etale covering $[n]_X \colon X\to X$ induces equivalences
\[
\Coh(X,\alpha)\simeq\Coh(X)_{X[n],\bvr,a}
\]
and \(
\bdc(X,\alpha)\simeq\bdc(X)_{X[n],\bvr,a}\).
\end{corollary}

\begin{proof}
    It suffices to check that the Brauer class defined by the $2$-cocycle $a$ via Proposition~\ref{prop:twisted-equivariant-equivalence} coincides with $\alpha$.
    This is exactly \cite[Proposition 3]{Ber72}.
\end{proof}

\section{Symplectic abelian varieties and Polishchuk's correspondence}
\label{sec:derived-torelli}

In this section, we recall Polishchuk's correspondence \cite{Pol96} between symplectic structures and twisted derived equivalences for abelian varieties. It includes the construction of symplectic biextensions and the notion of Lagrangian subvarieties; these will play a central role in the proof of Theorem~1.1. The key object for us is the symplectic abelian variety \(A_{(X,\alpha)}\) associated to a twisted abelian variety \((X,\alpha)\).

\subsection{Symplectic Biextensions and Lagrangian Subvarieties}
\label{subsec:polishchuk}

\begin{definition}[Biextension]
\label{def:biextension}
Let $A$ be an abelian variety and let \(\sigma \in \Aut(A^2)\) be the involution \(\sigma(z_1,z_2) = (z_2,z_1)\). Then
\begin{itemize}
    \item A \emph{biextension} of \(A^2\) is a line bundle \(L\) with isomorphisms:
          \[
          (\pr_1 + \pr_2, \pr_3)^*L \cong \pr_{13}^*L \otimes \pr_{23}^*L\,, \quad
          (\pr_1, \pr_2 + \pr_3)^*L \cong \pr_{12}^*L \otimes \pr_{13}^*L\,,
          \]
          satisfying cocycle conditions on \(A^3\) (see \cite[\S 10.3]{Pol03}).

    \item A biextension \(L\) is \emph{skew-symmetric} if equipped with an isomorphism of biextensions \(\phi\colon \sigma^*L \cong L^{-1}\) and trivialization \(\Delta^*L \cong \mathcal{O}_A\) over the diagonal \(\Delta\colon A \hookrightarrow A^2\), compatible with \(\phi\).
    It is \emph{symplectic} if additionally the induced map \(\psi_L\colon A \to \widehat{A}\) (\(a \mapsto L|_{\{a\} \times A}\)) is an isomorphism.
    Note that for a skew-symmetric biextension $L$ we have $\psi_L=-\widehat{\psi_L}$.
    A \emph{symplectic isomorphism} between two symplectic abelian varieties $(A_1,L_1)$ and $(A_2,L_2)$ is an isomorphism $f\colon A_1\to A_2$ such that $(f\times f)^*L_2\cong L_1$, i.e.,
    the following diagram
    \[\begin{tikzcd}[sep=large]
        {A_1} & {A_2} \\
        {\widehat{A}_1} & {\widehat{A}_2}
        \arrow["f", from=1-1, to=1-2]
        \arrow["{\psi_{L_1}}"', from=1-1, to=2-1]
        \arrow["{\psi_{L_2}}", from=1-2, to=2-2]
        \arrow["{\widehat{f}}"', from=2-2, to=2-1]
    \end{tikzcd}\]
    is commutative.

    \item A subvariety \(X \subset A\) is \emph{isotropic} if \(L|_{X \times X} \cong \mathcal{O}_{X \times X}\) as skew-symmetric biextensions.
    Equivalently,
    \[
    X\xrightarrow{i} A\xrightarrow{\psi_L}\widehat{A}\xrightarrow{\widehat{i}}\widehat{X}
    \]
    is zero.
    Moreover, it is \emph{Lagrangian} if it is isotropic and the morphism $X\to\widehat{A/X}$ induced by $\psi_L$ is an isomorphism.
\end{itemize}
\end{definition}

Let \(A\) be an abelian variety equipped with a symplectic biextension of the form \(P \otimes \sigma^* P^{-1}\) where $P$ itself is a biextension of $A^2$,
and \(X\) and \(Y\) be Lagrangian subvarieties in \(A\) with the intersection \(X \cap Y\) finite.
Polishchuk \cite[\S 3]{Pol96} constructs a \emph{central extension}
\[
1 \longrightarrow \mathbb{G}_m \longrightarrow G \longrightarrow X\cap Y \longrightarrow 0\,,
\]
where \(G\) is the group scheme that parameterizes the isomorphisms between the symmetric structures on \(P|_{(X\cap Y)^2}\).
This defines a \emph{canonical} element
\begin{equation}\label{eq:Pol-Brauer}
    e_X \in \mathrm{H}^2(A/X, \mathbb{G}_m)\,,
\end{equation}
which is independent of the choice of \(Y\) (\cf\cite[Proposition 3.1]{Pol96}). Furthermore, under the Lagrangian identifications
\[
\widehat{X} \cong A/Y \quad \text{and} \quad \widehat{Y} \cong A/X\,,
\]
the kernels satisfy
\[
\ker(Y \to \widehat{X}) = Y \cap X \quad \text{and} \quad \ker(X \to \widehat{Y}) = X \cap Y\,,
\]
and there exists a \emph{non-degenerate} alternating form
\[
e\colon (Y \cap X) \times (X \cap Y) \longrightarrow \mathbb{G}_m\,,
\]
which realizes the canonical duality between these kernels.
This implies that $X\cap Y$ admits a Lagrangian decomposition as a finite abelian group,
and hence the cardinality of \(X \cap Y\) is a \emph{square} (\cf\cite[Lemma 5.2]{Da10}).

\begin{remark}
    In general, we cannot assume that an arbitrary symplectic biextension on a particular abelian variety is of the form $P\otimes\sigma^*P^{-1}$, see \S \ref{subsec:vanishing-Brauer} for further discussions.
    However, according to the proof of \cite[Theorem 4.2]{Pol96} and the remark after it,
    Polishchuk points out that once one can define the Brauer class and the twisted categories,
    these assumptions can be removed except in the case where the characteristic of the base field is $2$.
\end{remark}

\subsection{Symplectic Isomorphisms and Derived Equivalences}
\label{sec:sp-var-and-biextension}

\begin{definition}
\label{def:polishchuk-orlov}
For an abelian variety \(X\) and \(\alpha \in \Br(X)[n]\) with $n$ invertible in $k$, we define the quotient abelian variety
\[
A_{(X,\alpha)} \coloneq \bigl( X \times \widehat{X} \bigr) / K
\]
where \(K \subset X[n] \times \widehat{X}[n]\) is the graph of \(\phi_\alpha\colon X[n] \to \widehat{X}[n]\) defined in \eqref{eq:twist-hom},
and the quotient map is denoted by $\pi_X\colon X\times\widehat{X}\to A_{(X,\alpha)}$.
This variety fits into a short exact sequence:
\begin{equation}
\label{eq:exact-seq-A}
0 \longrightarrow \widehat{X} \xlongrightarrow{i} A_{(X,\alpha)} \xlongrightarrow{q_X} X \longrightarrow 0\,.
\end{equation}
Here $i(\xi)=\pi_X(0,\xi)$ for any $\xi\in\widehat{X}$ and $q_X$ is the morphism induced by $[n]_X\circ\pr_X$ such that
\[
q_X\bigl(\pi_X(y,\xi)\bigr)=ny
\]
for any $(y,\xi)\in X\times\widehat{X}$.
\end{definition}
By \cite[Theorem 1.2]{Pol96}, \(A_{(X,\alpha)}\) is independent of the choice of \(\phi_\alpha\) up to isomorphism.

Let $\scP$ be the Poincar\'e line bundle on $X\times \widehat{X}$.
The abelian variety \(A_{(X,\alpha)}\) admits a symplectic biextension descending from \(L_X^{\otimes n}\),
where 
\[
L_X = p^*_{23}\scP \otimes p^*_{14}\scP^{-1}
\]
is the symplectic biextension of \((X 
\times \widehat{X})^2\); see \cite[Theorem 1.2]{Pol96} and the argument before the theorem.

\begin{remark}\label{rmk: diff-choice-symplectic-biextension}
The canonical symplectic biextension on $X\times\widehat{X}\times X\times\widehat{X}$ introduced by Polishchuk in \cite[\S 1]{Pol96} is
\[
L'=p_{14}^*\mathscr{P}\otimes p_{23}^*\mathscr{P}^{-1}
\]
and the corresponding isomorphism $\psi_{L'}\colon X\times\widehat{X}\to\widehat{X}\times X$ is
\[
\psi_{L'}=
\begin{pmatrix}
    0 & -1_{\widehat{X}} \\
    1_X & 0
\end{pmatrix}\,.
\]
Our construction generalizes Orlov's construction,
and the symplectic biextension used on $X\times\widehat{X}\times X\times\widehat{X}$ in \cite[Proposition 2.14]{Orl02} is instead $L_X$.
\end{remark}

We will denote the induced isomorphism by
\begin{equation}
\label{eq:symplectic-isomorphism}
\psi_{\alpha}\colon A_{(X,\alpha)}\longrightarrow\widehat{A}_{(X,\alpha)}\,,
\end{equation}
which is the unique morphism satisfying $\widehat{\pi_X}\circ\psi_\alpha\circ\pi_X=n\psi_{L_X}$.
Via the inclusion \(i\), we identify \(\widehat{X}\) as a Lagrangian subvariety of \(A_{(X,\alpha)}\), with the associated canonical Brauer class
\[
e_{\widehat{X}}\in \mathrm{H}^2(X,\mathbb{G}_m)\,.
\]

\begin{lemma}
    The canonical class $e_{\widehat{X}}$ defined in \eqref{eq:Pol-Brauer} coincides with the Brauer class $\alpha$, up to a uniform sign depending on the convention for twisted sheaves.
\end{lemma}

\begin{proof}
    Denote $\pi_X=(s,i)\colon X\times\widehat{X}\to A_{(X,\alpha)}$.
    Then $i(\widehat{X})$ is a Lagrangian subvariety of $A_{(X,\alpha)}$ and $s(X)$ is another Lagrangian subvariety of $A_{(X,\alpha)}$ such that
    \[
    H\coloneqq i(\widehat{X})\cap s(X)\cong \phi_{\alpha}(X[n])
    \]
    is finite.
    According to the proof of~\cite[Proposition 3.2]{Pol96}, the commutator form $e$ of the associated central extension of $H$ is
    \[
    e\bigl(\phi_{\alpha}(h_1),\phi_{\alpha}(h_2)\bigr)=\langle h_1, \phi_{\alpha}(h_2)\rangle\quad\forall h_1,h_2\in X[n]
    \]
    up to sign.
    After pulling back along $s|_{X[n]}\colon X[n]\to H$, we obtain a commutator form on $X[n]$ which is exactly $e_\alpha$ defined in \eqref{eq:twist-hom}.
    Thus according to~\cite[\S 3]{Pol96}, $e_{\widehat{X}}$ coincides with $\alpha$ up to sign.
\end{proof}

However, this sign issue does not affect the derived Torelli statement, since the derived duality gives $\bdc(X,\alpha)^{op}\simeq\bdc(X,\alpha^{-1})$.
Thus, we will identify the canonical Brauer class $e_{\widehat{X}}$ with $\alpha$.
See Remark~\ref{rmk:difference} for a related discussion.

The following foundational result is essentially established in~\cite{Pol96}.

\begin{theorem}
\label{Thm:Polishchuk}
For twisted abelian varieties \((X,\alpha)\) and \((Y,\beta)\), suppose that there exists a symplectic isomorphism
\[
\psi \colon A_{(X,\alpha)} \xlongrightarrow{\sim} A_{(Y,\beta)}\,.
\]
Then
\begin{enumerate}
\item If $\chara(k) \ne 2$ and \(\psi(\widehat{X})\cap\widehat{Y}\) is finite  with order invertible in \(k\), then there is an equivalence of triangulated categories
\[
\bdc(X, \alpha) \simeq \bdc(Y, \beta)\,.
\]
\item \(X\) and \(Y\) are principally isogenous.
\end{enumerate}
In particular, if $\chara(k)=0$, then we always have an equivalence of triangulated categories \(\bdc(X, \alpha) \simeq \bdc(Y, \beta)\).
\end{theorem}

\begin{proof}
For the first assertion,~\cite[Proposition 3.1]{Pol96} is valid under our assumptions.
Then~\cite[Theorem 4.2]{Pol96} gives an intertwining operator that induces a derived equivalence
\[
\bdc(X,\alpha) \simeq \bdc(Y,\beta)\,,
\]
whenever $\psi(\widehat{X})\cap \widehat{Y}$ is finite with its order invertible in $k$.
See Remark~\ref{rmk: assumptions-on-Pol-theorem} for a more detailed discussion.

For the second assertion, it is implicitly stated in \cite[Remark on \pno~10]{Pol96} and we provide the details here for the completeness of the proof.
Given a symplectic isomorphism $\psi\colon A_{(X,\alpha)} \xrightarrow{\sim} A_{(Y,\beta)}$,
 the image $\psi(\widehat{X})$ of $\widehat{X}$ is Lagrangian in $A_{(Y,\beta)}$ as $\widehat{X}$ is Lagrangian in $A_{(X,\alpha)}$ via the identification \eqref{eq:exact-seq-A}.
 Then it divides into two cases according to the cardinality of $\psi(\widehat{X}) \cap \widehat{Y}$.

\begin{enumerate}[align=left,leftmargin=0pt,labelindent=0pt,listparindent=\parindent,labelwidth=0pt,topsep=0.5em,itemsep=0.5em,itemindent=!,label={\underline{\bfseries Case \arabic*:}}]

\item
\textbf{$\psi(\widehat{X}) \cap \widehat{Y}$ is finite.} The composition
\[
\phi\colon \widehat{X} \xhookrightarrow{\;\;i\;\;} A_{(X,\alpha)} \xlongrightarrow{\psi} A_{(Y,\beta)} \xlongrightarrow{q_Y} Y
\]
is a principal isogeny because
\[
\ker(\phi) \cong \psi(\widehat{X}) \cap \widehat{Y}\,,
\]
whose cardinality is finite and thus is a square as explained in \S \ref{subsec:polishchuk}.
Composing with a polarization $X \to \widehat{X}$ yields a principal isogeny $X \to Y$.

\item
\textbf{$\psi(\widehat{X}) \cap \widehat{Y}$ is not finite.} In this case, by \cite[Lemma 4.1]{Pol96}, there exists a Lagrangian subvariety $\widehat{Z} \subset A_{(Y,\beta)}$ such that both $\widehat{Z} \cap \widehat{Y}$ and $\widehat{Z} \cap \psi(\widehat{X})$ are finite.
Applying the arguments in {\bfseries Case 1}, we obtain two principal isogenies $X \to Z$ and $Z \to Y$, and their composition $X \to Y$ is also a principal isogeny.
\end{enumerate}

Finally, if $\chara(k)=0$, then~\cite[Theorem 4.3]{Pol96} gives the desired equivalence of triangulated categories $\bdc(X,\alpha)\simeq\bdc(Y,\beta)$.
\end{proof}

\begin{remark}\label{rmk: assumptions-on-Pol-theorem}
    In~\cite{Pol96}, the assertion that symplectic isomorphisms imply derived equivalences is obtained from the combination of~\cite[Proposition 3.1, Lemma 4.1, Theorem 4.2]{Pol96}.
    Just as mentioned in part (2) of~\cite[Remark 2.3.1]{Pol12},~\cite[Proposition 3.1]{Pol96} requires the assumption that $\chara(k)=0$.
    And~\cite[Theorem 1.1]{Ela15} helps us to ensure that~\cite[Proposition 3.1]{Pol96} holds in positive characteristic under our additional assumptions.
    Moreover, even if $\psi(\widehat{X})\cap \widehat{Y}$ is not finite, all the statements remain valid under the additional assumption that there exists a Lagrangian subvariety \(\widehat{Z} \subseteq A_{(Y,\beta)}\) such that both \(\widehat{Z} \cap \widehat{Y}\) and \(\widehat{Z} \cap \psi(\widehat{X})\) are finite and of order invertible in $k$.
    Such a $\widehat{Z}$ exists if $Y$ admits a prime-to-$p$ polarization.
\end{remark}

\subsection{Enhanced Symplectic Abelian Varieties}
\label{subsec:vanishing-Brauer}

In the context of the twisted derived Torelli theorem, the twisted symplectic abelian variety $(A_{(X,\alpha)}, q_X^\ast \alpha)$ plays a fundamental role.
As we shall see, the order of the Brauer class $q_X^\ast \alpha$ is governed by the enhancement data on $A_{(X,\alpha)}$.

Recall from \cite{Pol12} that a symplectic abelian variety $(A,L)$ is called \emph{enhanced} if $L \cong P \otimes \sigma^* P^{-1}$ as biextensions for some biextension $P \in \Pic(A^2)$ or equivalently, $\psi_L=f-\widehat{f}$ for some $f\colon A\to\widehat{A}$ (see \cite[Definition 2.1.2]{Pol12}). Then we have

\begin{proposition}
Let $(X,\alpha)$ be a twisted abelian variety with $\ord(\alpha)$ invertible in $k$.
\begin{enumerate}
    \item The class $q^*_X\alpha$ is $2$-torsion, i.e., $q^*_X\alpha \in \Br\bigl(A_{(X,\alpha)}\bigr)[2]$.

    \item  $q_X^*\alpha \in \Br\bigl(A_{(X,\alpha)}\bigr)$ is trivial if and only if $A_{(X,\alpha)}$ is enhanced.

    \item For any twisted abelian variety $(Y,\beta)$ with $\ord(\beta)$ invertible in $k$, any symplectic isomorphism
    \[
    \psi\colon A_{(X,\alpha)} \longrightarrow A_{(Y,\beta)}
    \]
    satisfies $\psi^*q_Y^*\beta = q^*_X\alpha$.
\end{enumerate}
\end{proposition}

\begin{proof}
Recall the following exact sequence:
\[
0 \longrightarrow \mathrm{NS}(A_{(X,\alpha)})/n\mathrm{NS}(A_{(X,\alpha)}) \xlongrightarrow{\iota} \mathrm{Hom}(\wedge^2 A_{(X,\alpha)}[n], \mu_n) \xlongrightarrow{\delta} \Br\bigl(A_{(X,\alpha)}\bigr)[n] \longrightarrow 0.
\]
In particular, (\cf\cite[Proposition 3]{Ber72}),
\[
\ker \delta=\image \iota=\Bigl\{e^{L^{\otimes n}}\colon A_{(X,\alpha)}[n]\times A_{(X,\alpha)}[n]\longrightarrow\mu_n\ \Big|\ L\in\Pic\bigl(A_{(X,\alpha)}\bigr)\Bigr\}\,.
\]
Using the Weil pairing, it corresponds to
\begin{align*}
&\Bigl\{\phi_L\big|_{A_{(X,\alpha)}[n]}\colon A_{(X,\alpha)}[n]\longrightarrow \widehat{A}_{(X,\alpha)}[n]\ \Big|\ L\in\Pic\bigl(A_{(X,\alpha)}\bigr)\Bigr\}\\
=&\Bigl\{\phi\big|_{A_{(X,\alpha)}[n]}\colon A_{(X,\alpha)}[n]\longrightarrow \widehat{A}_{(X,\alpha)}[n]\ \Big|\ \text{symmetric}\ \phi\colon A_{(X,\alpha)}\longrightarrow\widehat{A}_{(X,\alpha)}\Bigr\}\,.
\end{align*}
And the alternating bilinear pairing
\[
q^*_Xe_\alpha\colon A_{(X,\alpha)}[n]\times A_{(X,\alpha)}[n]\longrightarrow\mu_n\,,
\]
defined by $q^*_Xe_\alpha(z,w)=e_\alpha\bigl(q_X(z),q_X(w)\bigr)$ with $z$, $w\in A_{(X,\alpha)}[n]$, is sent to $q^*_X\alpha$ via $\delta$.
Suppose that $q^*_Xe_\alpha$ corresponds to the homomorphism
\[
\widetilde{\phi}_\alpha\colon A_{(X,\alpha)}[n]\longrightarrow \widehat{A}_{(X,\alpha)}[n]\,,
\]
then for any $z$, $w\in A_{(X,\alpha)}[n]$,
\begin{align*}
q^*_Xe_\alpha(z,w)&=\langle q_X(z),\phi_\alpha\circ q_X(w)\rangle\\
&=\langle z,\widehat{q_X}\circ\phi_\alpha\circ q_X(w)\rangle\\
&=\langle z,\psi_\alpha\circ i\circ\phi_\alpha\circ q_X(w)\rangle\\
&=\langle z,\psi_\alpha\circ(-s)\circ q_X(w)\rangle\,,
\end{align*}
where $s$ is the restriction of the projection $X\times\widehat{X}\to A_{(X,\alpha)}$ to $X$.

Hence $-\widetilde{\phi}_\alpha$ is equal to the restriction of
\[
\varphi_{\alpha}\colon A_{(X,\alpha)}\xlongrightarrow{q_X} X \xlongrightarrow{s}  A_{(X,\alpha)} \xlongrightarrow{\psi_\alpha} \widehat{A}_{(X,\alpha)}\,,
\]
to the $n$-torsion points of $A_{(X,\alpha)}$, where $\psi_\alpha$ is the isomorphism induced by the biextension.

Now we are ready to prove the proposition.
\begin{enumerate}
\item Consider the following commutative diagram
\[\begin{tikzcd}[sep=huge]
	{X\times\widehat{X}} & A_{(X,\alpha)} & {X\times\widehat{X}} \\
	{\widehat{X}\times X} & {\widehat{A}_{(X,\alpha)}} & {\widehat{X}\times X} \mathrlap{\,,}
	\arrow["{(s,i)}", from=1-1, to=1-2]
	\arrow["{n^2\psi_{L_X}}"', from=1-1, to=2-1]
	\arrow["{(q_X,\widehat{s}\psi_\alpha)}", from=1-2, to=1-3]
	\arrow["{n\psi_{\alpha}}"', from=1-2, to=2-2]
	\arrow["{\psi_{L_X}}"', from=1-3, to=2-3]
	\arrow["{(\widehat{s},\widehat{i})}"', from=2-2, to=2-1]
	\arrow["{(\widehat{q_X},-\psi_{\alpha} s)}"', from=2-3, to=2-2]
\end{tikzcd}\]
the commutativity of the right square means that
\begin{equation}
\label{eq:Brauer-enhanced}
n\psi_\alpha=\varphi_\alpha-\widehat{\varphi_\alpha}\,.
\end{equation}
This implies that
\[
\varphi_\alpha|_{A_{(X,\alpha)}[n]}=\widehat{\varphi_{\alpha}}|_{A_{(X,\alpha)}[n]}\,,
\]
so $2\widetilde{\phi}_\alpha$ is equal to the restriction of $-\varphi_\alpha-\widehat{\varphi_\alpha}$, and this implies $(q^*_X\alpha)^2=1$.

\item $A_{(X,\alpha)}$ is enhanced means that $\psi_\alpha=f-\widehat{f}$ for some $f\colon A_{(X,\alpha)}\to \widehat{A}_{(X,\alpha)}$ by definition.
Then $nf-\varphi_\alpha$ is symmetric by \eqref{eq:Brauer-enhanced} and its restriction is equal to $\widetilde{\phi}_\alpha$.
\newline
Conversely, if $q^*_X\alpha=0$, then there is a symmetric $\varphi\colon A_{(X,\alpha)}\to\widehat{A}_{(X,\alpha)}$ whose restriction to $A_{(X,\alpha)}[n]$ is $\widetilde{\phi}_\alpha$.
Thus $\varphi_\alpha+\varphi=nf$ for some $f$ and then $\psi_\alpha=f-\widehat{f}$ by \eqref{eq:Brauer-enhanced}.

\item For any $z$, $w\in A_{(X,\alpha)}[n]$,
\[
\psi^*q^*_Ye_\beta(z,w)=\langle\psi(z),-\varphi_\beta\circ\psi(w)\rangle=\langle z,-\widehat{\psi}\circ\varphi_\beta\circ\psi(w)\rangle\,,
\]
where $\varphi_\beta$ is defined similarly to $\varphi_\alpha$.
Hence the difference between $\psi^*q^*_Ye_\beta$ and $q^*_Xe_\alpha$ is the restriction of $\varphi_\alpha-(\widehat{\psi}\circ\varphi_\beta\circ\psi)$,
and note that
\[
\varphi_\alpha-(\widehat{\psi}\circ\varphi_\beta\circ\psi)-\widehat{\varphi_\alpha}+(\widehat{\psi}\circ\widehat{\varphi_\beta}\circ\psi)=n\psi_\alpha-(\widehat{\psi}\circ n\psi_\beta\circ\psi)=0\,,
\]
which implies $\psi^*q^*_Y\beta=q^*_X\alpha$.
\end{enumerate}
\end{proof}

\section{Twisted autoequivalences via equivariant categories}
\label{sec:twisted-functors}

Orlov showed that for an abelian variety \(X\), the group \(X\times \widehat{X}\) embeds naturally into the autoequivalence group of \(\mathrm{D}^{\mathrm{b}}(X)\). In this section, we extend this picture to the twisted setting. Working equivariantly and following Ploog's formalism, we construct for each \(z\in A_{(X,\alpha)}\) an autoequivalence \(\Phi_z\) of \(\mathrm{D}^{\mathrm{b}}(X,\alpha)\), and show that this gives an injective homomorphism from \(A_{(X,\alpha)}\) into the autoequivalence group. This will be used later in the construction of the twisted Orlov functors.

\subsection{Setup and Examples}

We begin by reviewing some foundational results of autoequivalences of twisted derived categories through equivariant categories.

Fix a twisted abelian variety $(X,\alpha)$ with $\alpha \in \Br(X)[n]$, where $n$ is invertible in $k$.
Choose a normalized $2$-cocycle $a \in \rZ^2(X[n], \mathbb{G}_m)$ representing $\alpha$. 
Let $G \coloneq X[n]$ be equipped with its natural action
\[\bvr \colon G \longrightarrow \Aut(X)\longrightarrow \Aut(\bdc(X))\,.\]
By Corollary \ref{cor:av-twisted-equivariant-equivalence}, we have an equivalence
\[
\bdc(X,\alpha) \simeq \bdc(X)_{G,\bvr, a}\,.
\]
We will mainly work on this equivariant category.

\begin{example}\label{ex:twisted-skyscraper}
    Under the above identification, a (twisted) skyscraper sheaf $k(x)\in\bD^b(X,\alpha)$ for $x\in X$ corresponds to the inflation 
    \[
    \Inf_{\bvr, a}\bigl(k(y)\bigr)=\Bigl(\bigoplus_{g\in G}t^*_g k(y),\{a_\sigma\}_{\sigma\in G}\Bigr)
    \]
    where the linearization $\{a_{\sigma}\}_{\sigma\in G}$ is
    \[
    a_\sigma \colon t^*_\sigma \Bigl( \bigoplus_{g\in G} t^*_g \bigl(k(y)\bigr) \Bigr) \xlongrightarrow{\ \cong\ } \bigoplus_{g\in G} t^*_{\sigma} t^*_g \bigl(k(y)\bigr) \xlongrightarrow{\ \oplus_g\, a_{\sigma, g}\ } \bigoplus_{g\in G} t^*_{\sigma+g} \bigl(k(y)\bigr)
    \]
    and $y\in X$ satisfies $[n]_X (y)=x$.
\end{example}

We start with some autoequivalences of $\bdc(X)_{G,\bvr, a}$.
To construct such autoequivalences, consider the finite \'etale covering
\[[n]_{X \times X} \colon X\times X\longrightarrow X\times X \,,\]
and define the $2$-cocycle $a^{-1} \boxtimes a \in \rZ^2(G \times G, \mathbb{G}_m)$ by
\begin{equation}
(a^{-1} \boxtimes a){(\tau_1,\tau_2)} \coloneq \frac{a_{\sigma'_1,\sigma'_2}}{a_{\sigma_1,\sigma_2}}, \quad \text{where } \tau_i = (\sigma_i,\sigma'_i) \in G \times G\,.
\end{equation}
Then any object $(E,\eta) \in \bdc(X \times X)_{G \times G, \bvr, a^{-1} \boxtimes a}$ gives an equivariant Fourier--Mukai functor
\[
\Phi_{(E,\eta)}\colon \bdc(X)_{G,\bvr, a}\longrightarrow \bdc(X)_{G,\bvr, a}\,.
\]
In particular, one can construct a $(G\times G,a^{-1}\boxtimes a)$-linearized object from any $G_{\Delta}$-linearized object via the \emph{(left) inflation}, where $G_\Delta$ is the diagonal subgroup of $G\times G$, see \cite[\S 3]{Plo05} for the complete procedures.

\begin{example}
\label{ex:identity-functor}
Consider the structure sheaf of the diagonal $\mathcal{O}_\Delta\in\bdc(X\times X)$ equipped with the canonical $G_{\Delta}$-linearization $t^*_{(g,g)}\mathcal{O}_\Delta=\mathcal{O}_\Delta$, $\forall g\in G$,
its (left) inflation is
\[
\bigoplus_{g\in G}t^*_{(g,0)}\mathcal{O}_\Delta\in\bdc(X\times X)
\]
together with isomorphisms
\begin{align*}
\zeta_{\sigma,\sigma'} \colon
& t^*_{(\sigma,\sigma')}\Bigl(\bigoplus_{g\in G}t^*_{(g,0)}\mathcal{O}_\Delta\Bigr)\xrightarrow{\bigoplus_{g\in G}a^{-1}_{\sigma,g}}\bigoplus_{g\in G}t^*_{(\sigma+g,\sigma')}\mathcal{O}_\Delta \\
& \quad \xrightarrow{\bigoplus_{g\in G} a_{\sigma+g-\sigma',\sigma'}}
\bigoplus_{g\in G}t^*_{(\sigma+g-\sigma',0)}t^*_{(\sigma',\sigma')}\mathcal{O}_\Delta= \bigoplus_{g\in G}t^*_{(\sigma+g-\sigma',0)}\mathcal{O}_\Delta\,.
\end{align*}
Since $a\in\rZ^2(G,\mathbb{G}_m)$, the following diagram is commutative
\begin{equation*}
\begin{tikzcd}[row sep=4em,column sep=8em]
	{t^*_{(\sigma_1,\sigma'_1)}t^*_{(\sigma_2,\sigma'_2)}\Bigl(\bigoplus\limits_{g\in G}t^*_{(g,0)}\mathcal{O}_\Delta\Bigr)} & {t^*_{(\sigma_1,\sigma'_1)}\Bigl(\bigoplus\limits_{g\in G}t^*_{(\sigma_2+g-\sigma'_2,0)}\mathcal{O}_\Delta\Bigr)} \\
	{t^*_{(\sigma_1+\sigma_2,\sigma'_1+\sigma'_2)}\Bigl(\bigoplus\limits_{g\in G}t^*_{(g,0)}\mathcal{O}_\Delta\Bigr)} & {\bigoplus\limits_{g\in G}t^*_{(\sigma_1+\sigma_2+g-\sigma'_2-\sigma'_1,0)}\mathcal{O}_\Delta}\mathrlap{\,,}
	\arrow["{t^*_{(\sigma_1,\sigma'_1)}\bigl(\zeta_{\sigma_2,\sigma'_2}\bigr)}", from=1-1, to=1-2]
	\arrow["{\bigoplus\limits_{g\in G}\frac{a_{\sigma'_1,\sigma'_2}}{a_{\sigma_1,\sigma_2}}}"', from=1-1, to=2-1]
	\arrow["{\zeta_{\sigma_1,\sigma'_1}}", from=1-2, to=2-2]
	\arrow["{\zeta_{\sigma_1+\sigma_2,\sigma'_1+\sigma'_2}}"', from=2-1, to=2-2]
\end{tikzcd}
\end{equation*}
where the arrows are given by
\begin{align*}
    t^*_{(\sigma_1,\sigma'_1)}\Bigl(\zeta_{\sigma_2,\sigma'_2}\Bigr)
    &=\bigoplus_{g\in G}\frac{a_{\sigma_2+g-\sigma'_2,\sigma'_2}}{a_{\sigma_2,g}}\,,\\
    \zeta_{\sigma_1+\sigma_2,\sigma'_1+\sigma'_2}
    &=\bigoplus_{g\in G}\frac{a_{\sigma_1+\sigma_2+g-\sigma'_2-\sigma'_1,\sigma'_1+\sigma'_2}}{a_{\sigma_1+\sigma_2,g}}\,,\\
    \zeta_{\sigma_1,\sigma'_1}
    &=\bigoplus_{g\in G}\frac{a_{\sigma_1+\sigma_2+g-\sigma'_2-\sigma'_1,\sigma'_1}}{a_{\sigma_1,\sigma_2+g-\sigma'_2}}\,.\\
\end{align*}
Hence $\bigl(\bigoplus_{g\in G}t^*_{(g,0)}\mathcal{O}_\Delta,\zeta\bigr)\in\bdc(X\times X)_{G\times G, \bvr, a^{-1}\boxtimes a}$,
and the equivariant Fourier--Mukai functor it defines is isomorphic to the identity; see \cite[Example 3.14]{Plo05}.
\end{example}

Besides the left inflation,
one can also construct the right inflation from a $G_\Delta$-linearized object.
The construction is given by simply choosing $\{(0,g)\}_{g\in G}$ as the set of representatives of $(G\times G)/G_\Delta$, instead of $\{(g,0)\}_{g\in G}$ used above for the left inflation.
For example, the right inflation of $\mathcal{O}_\Delta$ is
\[
\bigoplus_{g\in G}t^*_{(0,g)}\mathcal{O}_\Delta\in\bdc(X\times X)
\]
together with isomorphisms
\begin{align*}
\zeta'_{\sigma,\sigma'} \colon
& t^*_{(\sigma,\sigma')}\Bigl(\bigoplus_{g\in G}t^*_{(0,g)}\mathcal{O}_\Delta\Bigr)\xrightarrow{\bigoplus_{g\in G}a_{\sigma',g}}\bigoplus_{g\in G}t^*_{(\sigma,\sigma'+g)}\mathcal{O}_\Delta \\
& \quad \xrightarrow{\bigoplus_{g\in G} a^{-1}_{\sigma'+g-\sigma,\sigma}}
\bigoplus_{g\in G}t^*_{(0,\sigma'+g-\sigma)}t^*_{(\sigma,\sigma)}\mathcal{O}_\Delta= \bigoplus_{g\in G}t^*_{(0,\sigma'+g-\sigma)}\mathcal{O}_\Delta\,.
\end{align*}

Since $a\in\rZ^2(G,\mathbb{G}_m)$,
one can check that
\[
\bigoplus_{g\in G}a^{-1}_{-g,g}\colon\bigoplus_{g\in G}t^*_{(g,0)}\mathcal{O}_\Delta\xrightarrow{\oplus_{g\in G}a^{-1}_{-g,g}}\bigoplus_{g\in G}t^*_{(0,-g)}t^*_{(g,g)}\mathcal{O}_\Delta=\bigoplus_{g\in G}t^*_{(0,-g)}\mathcal{O}_\Delta
\]
gives an isomorphism between the left inflation and the right inflation,
where the latter equality is the canonical $G_\Delta$-linearization.
Generally, the left inflation and the right inflation of any $G_\Delta$-linearized object are isomorphic via similar arguments.

\subsection{Generalized Functors}

Following the ideas in Example \ref{ex:identity-functor}, we define two fundamental families of functors that generalize the identity:

\begin{enumerate}[label={\bfseries(\arabic*)}]
    \item \textbf{Translation functors}: For any $y \in X$, the compatibility ${t_y}_*t_\tau^* \cong t_\tau^*{t_y}_*$ (canonical for all $\tau \in G$) induces an autoequivalence:
    \[
    {t_y}_* \colon \bdc(X)_{G,\bvr, a} \longrightarrow \bdc(X)_{G,\bvr, a}\,.
    \]

    \item \textbf{Tensor functors}: For any $\xi \in \widehat{X}$, under $\Coh(X) \simeq  \Coh(X)_G$, the line bundle $\mathscr{P}_\xi$ corresponds to $(\mathscr{P}_\nu, \chi_\xi) \in  \Coh(X)_G$ where:
    \begin{itemize}
        \item $\mathscr{P}_\nu \cong [n]_X^*\mathscr{P}_\xi$\,,
        \item $\chi_\xi(g)\colon t_g^*\mathscr{P}_\nu \to \mathscr{P}_\nu$ are descent isomorphisms satisfying $\chi_\xi(g) \circ t_g^*(\chi_\xi(h)) = \chi_\xi(gh)$\,.
    \end{itemize}
    This yields the tensor functor:
    \[
    (-) \otimes (\mathscr{P}_\nu, \chi_\xi) \colon \bdc(X)_{G,\bvr, a} \longrightarrow \bdc(X)_{G,\bvr, a}\,.
    \]
\end{enumerate}

\begin{remark}
    In the language of twisted sheaves, one can regard the functors defined above as the pushforward along a translation and tensoring with a line bundle.
    However, the notation "$\Coh(X,\alpha)$" just means an equivalence class of categories defined via different {\v C}ech cocycles.
    Thus we need a further equivalence to make the pushforward an autoequivalence.
    And there is no canonical choice of such an equivalence, so we work in the equivariant category to avoid this ambiguity.
\end{remark}

Note that we can also define the translation functors for the classical (i.e., untwisted) equivariant category $\bdc(X)_{G,\bvr}$ by the compatibility ${t_y}_*t_\tau^* \cong t_\tau^*{t_y}_*$, so
\[
    {t_y}_* \colon \bdc(X)_{G,\bvr} \longrightarrow \bdc(X)_{G,\bvr}\,.
\]
Be aware that we abuse notation by using the same symbol for translation functors of different categories when there is no confusion.
In particular, ${t_y}_*(\mathscr{P_\nu},\chi_\xi)\cong(\mathscr{P}_{\nu},\chi_\xi)\in\bdc(X)_{G,\bvr}$ and thus we have the following isomorphisms between the generalized functors:
\begin{align*}
{t_y}_*(-)\otimes(\mathscr{P}_\nu,\chi_\xi)\cong {t_y}_*\bigl((-)\otimes(\mathscr{P}_\nu,\chi_\xi)\bigr)\colon\bdc(X)_{G,\bvr,a}\longrightarrow\bdc(X)_{G,\bvr,a}\,,
\end{align*}
which is similar to the commutativity of translation and tensor product with a line bundle in $\Pic^0$ for abelian varieties in classical cases.

\begin{remark}
    One should be aware that the above isomorphism of functors is not canonical; see \cite[\S 15.2]{Pol03} for the discussion.
\end{remark}

Now we present the equivariant Fourier--Mukai kernels of generalized functors.
For any $(y,\xi) \in X \times \widehat{X}$, denote the graph morphism of the translation by $y$ as
\begin{equation*}
\begin{aligned}
    \Gamma_y\colon X &\longrightarrow X \times X\\
    z &\longmapsto (z,z+y)\,.
\end{aligned}
\end{equation*}
The object ${\Gamma_y}_\ast\mathscr{P}_\nu$ carries a $G_\Delta$-linearization:

\begin{equation*}
    \lambda_g \colon t^*_{(g,g)}\bigl({\Gamma_y}_*\mathscr{P}_\nu\bigr) \xlongrightarrow{\sim} {\Gamma_y}_*\bigl(t_g^*\mathscr{P}_\nu\bigr) \xlongrightarrow{{\Gamma_y}_*(\chi_\xi(g))} {\Gamma_y}_*\mathscr{P}_\nu\,,
\end{equation*}
where
\begin{itemize}
    \item The first isomorphism is the canonical base change;
    \item The linearizations $\lambda_g$ satisfy the trivial cocycle condition $\lambda_g \circ t^*_{(g,g)}(\lambda_h) = \lambda_{gh}$.
\end{itemize}
Then the linearized kernel corresponding to the (left) inflation is
\[
\Bigl(\bigoplus_{\tau \in G} t^*_{(\tau,0)}{\Gamma_y}_*\mathscr{P}_\nu, \zeta^{(y,\xi)}\Bigr) \in \bdc(X \times X)_{G \times G, \bvr, a^{-1} \boxtimes a}
\]
with isomorphisms
\begin{equation*}
\zeta^{(y,\xi)}_{\tau,\tau'} = \bigoplus_{g \in G} \left(\frac{a_{\tau+g-\tau',\tau'}}{a_{\tau,g}} \cdot t^*_{(\tau+g-\tau',0)}(\lambda_{\tau'})\right)\,.
\end{equation*}
These data define the equivariant Fourier--Mukai transform:
\begin{equation}
\label{eq:equivariant-y-xi}
\Phi_{(y,\xi)} \coloneq \Phi_{\left(\bigoplus t^*_{(\tau,0)}{\Gamma_y}_*\mathscr{P}_\nu, \zeta^{(y,\xi)}\right)} \cong {t_y}_*\bigl((-) \otimes (\mathscr{P}_\nu, \chi_\xi)\bigr)\cong{t_y}_*(-) \otimes (\mathscr{P}_\nu, \chi_\xi)\,,
\end{equation}
(\cf\cite[Lemma 3.16]{Plo05}).
And one can check that these functors satisfy
\begin{align}
\label{eq:commutativity-actions}
\Phi_{(y_1,\xi_1)}\circ\Phi_{(y_2,\xi_2)}\cong\Phi_{(y_1+y_2,\xi_1+\xi_2)}
\end{align}
for any two $(y_i,\xi_i)\in X\times\widehat{X}$, $i=1,2$.

\begin{lemma}
\label{lem:trivial-functors}
For $(y,\xi)\in X\times \widehat{X}$, when $y \in G$ and $\xi = \phi_\alpha(y)$, the equivariant Fourier--Mukai transform $\Phi_{(y,\xi)}$ is isomorphic to the identity functor.
\end{lemma}

\begin{proof}
Using the Weil pairing $\langle -, - \rangle$, we identify $\mathscr{P}_\nu \cong \mathcal{O}_X$ with descent data:
\begin{equation}
\label{eq:character}
\chi_\xi(g) = \langle g, -\xi \rangle = \frac{a_{g,-y}}{a_{-y,g}} \colon t^*_g\mathcal{O}_X\longrightarrow\mathcal{O}_X,\quad \forall g\in G\,,
\end{equation}
so the isomorphisms $\zeta^{(y,\xi)}$ are given by
\[
\zeta^{(y,\xi)}_{\sigma,\sigma'}=\bigoplus_{g\in G}\Bigl(\frac{a_{\sigma+g-\sigma',\sigma'}}{a_{\sigma,g}}\cdot\frac{a_{\sigma',-y}}{a_{-y,\sigma'}}\Bigr) \colon t^*_{(\sigma,\sigma')}\Bigl(\bigoplus_{g\in G}t^*_{(g,0)}{\Gamma_y}_*\mathcal{O}_X\Bigr)\longrightarrow\bigoplus_{g\in G}t^*_{(\sigma+g-\sigma',0)}{\Gamma_y}_*\mathcal{O}_X\,.
\]
Consider the isomorphism
\[
\psi\colon \bigoplus_{g\in G}t^*_{(g,0)}{\Gamma_y}_*\mathcal{O}_X=\bigoplus_{g\in G}t^*_{(g,-y)}\mathcal{O}_\Delta\xlongrightarrow{\oplus_{g\in G}a_{g+y,-y}}\bigoplus_{g\in G}t^*_{(g+y,0)}t^*_{(-y,-y)}\mathcal{O}_\Delta=\bigoplus_{g\in G}t^*_{(g+y,0)}\mathcal{O}_\Delta
\]
where the equalities are all canonical isomorphisms, we need to show that this is an isomorphism in $\bdc(X\times X)_{G\times G, \bvr, a^{-1}\boxtimes a}$, i.e., the following diagram is commutative
\begin{equation*}
\begin{tikzcd}[row sep=large,column sep=5em]
	{t^*_{(\sigma,\sigma')}\Bigl(\bigoplus\limits_{g\in G}t^*_{(g,0)}{\Gamma_y}_*\mathcal{O}_X\Bigr)} & {\bigoplus\limits_{g\in G}t^*_{(\sigma+g-\sigma',0)}{\Gamma_y}_*\mathcal{O}_X} \\
	{t^*_{(\sigma,\sigma')}\Bigl(\bigoplus\limits_{g\in G}t^*_{(g+y,0)}\mathcal{O}_\Delta\Bigr)} & {\bigoplus\limits_{g\in G}t^*_{(\sigma+g-\sigma'+y,0)}\mathcal{O}_\Delta}
	\arrow["{\zeta^{(y,\xi)}_{\sigma,\sigma'}}", from=1-1, to=1-2]
	\arrow["{t^*_{(\sigma,\sigma')}\psi}"', from=1-1, to=2-1]
	\arrow["\psi", from=1-2, to=2-2]
	\arrow["{\zeta_{\sigma,\sigma'}}", from=2-1, to=2-2]
\end{tikzcd}
\end{equation*}
for any $(\sigma,\sigma')\in G\times G$,
which follows from the equalities
\begin{align*}
&\frac{a_{\sigma+g-\sigma'+y,-y}\cdot a_{\sigma+g-\sigma',\sigma'}\cdot a_{\sigma',-y}}{a_{\sigma,g}\cdot a_{-y,\sigma'}} =\frac{a_{\sigma+g-\sigma'+y,\sigma'-y}\cdot a_{\sigma',-y}}{a_{\sigma,g}}\\
=&\frac{a_{\sigma+g-\sigma'+y,\sigma'}\cdot a_{\sigma+g+y,-y}}{a_{\sigma,g}}
=\frac{a_{\sigma+g+y-\sigma',\sigma'}\cdot a_{g+y,-y}}{a_{\sigma,g+y}}
\end{align*}
since $a\in\rZ^2(G,\mathbb{G}_m)$.
\end{proof}

\subsection{$A_{(X,\alpha)}$ as a group of autoequivalences}

After explaining some particular autoequivalences in the previous subsection, we now describe a whole group of autoequivalences.
Recall that for the untwisted derived category $\bdc(X)$, there is a natural embedding
\[
X\times\widehat{X}\xhookrightarrow{\quad}\Aut\bigl(\bdc(X)\bigr)
\]
sending every $(x,\xi)\in X\times\widehat{X}$ to the functor ${t_x}_*(-)\otimes\mathscr{P}_\xi$.
When considering twisted derived category $\bdc(X,\alpha)$, we have a similar construction.

\begin{definition}
For any $z=[(y,\xi)]\in A_{(X,\alpha)}$, define
\[
\Phi_z\in\Aut(\bdc(X,\alpha))
\]
to be the autoequivalence that fits into the following commutative diagram
\[\begin{tikzcd}[sep=large]
	{\bdc(X,\alpha)} & {\bdc(X,\alpha)} \\
	{\bdc(X)_{G,\bvr,a}} & {\bdc(X)_{G,\bvr,a}}\mathrlap{\,,}
	\arrow["{\Phi_z}", from=1-1, to=1-2]
	\arrow["\simeq"', from=1-1, to=2-1]
	\arrow["\simeq", from=1-2, to=2-2]
	\arrow["{\Phi_{(y,\xi)}}", from=2-1, to=2-2]
\end{tikzcd}\]
where $\Phi_{(y,\xi)}$ is defined in \eqref{eq:equivariant-y-xi} and the vertical equivalence is given by Corollary \ref{cor:av-twisted-equivariant-equivalence}.
\end{definition}
Note that according to Lemma \ref{lem:trivial-functors}, $\Phi_z$ is well defined up to isomorphism.

\begin{proposition}
\label{prop:embedding-A-to-Auto}
The map $z\longmapsto\Phi_z$ defines an injective homomorphism
\[
A_{(X,\alpha)}(k)\xhookrightarrow{\quad}\Aut\bigl(\bdc(X,\alpha)\bigr)\,,
\]
where $A_{(X,\alpha)}(k)$ is the group of $k$-points of $A_{(X,\alpha)}$.
\end{proposition}

\begin{proof}
This assignment is a homomorphism by \eqref{eq:commutativity-actions}.
We now show the injectivity,
suppose that $(y_0,\xi_0)$ belongs to the kernel of this homomorphism,
which means that there is an isomorphism
\[
\psi_0\colon \bigoplus_{g\in G}t^*_{(g,0)}{\Gamma_{y_0}}_*\mathscr{P}_{\nu_0}\longrightarrow\bigoplus_{g\in G}t^*_{(g,0)}\mathcal{O}_\Delta
\]
in $\bdc(X\times X)$,
where $\nu_0=[n]_{\widehat{X}}(\xi_0)$.
By comparison on both sides, it is immediate that we have $y_0\in X[n]$ and $\xi_0\in \widehat{X}[n]$.
First note that $\psi_0$ splits as a sum of maps.
Then since each direct summand $t^*_{(g,0)}\mathcal{O}_{\Delta}$ is simple, $\psi_0$ is equivalent to a map $\omega \colon G\to\mathbb{G}_m$.
For any $g\in G$,
\[
\omega(g)\colon t^*_{(g,0)}{\Gamma_{y_0}}_*\mathcal{O}_X\longrightarrow t^*_{(g+y_0,0)}\mathcal{O}_\Delta
\]
is an isomorphism and $\psi_0=\bigoplus_{g\in G}\omega(g)$.

Similar to the proof of Lemma \ref{lem:trivial-functors}, since $\psi_0$ is an isomorphism in $\bdc(X\times X)_{G\times G, \bvr, a^{-1}\boxtimes a}$,
the map $\omega$ must satisfy
\[
\omega(\sigma+g-\sigma')\cdot\frac{a_{\sigma+g-\sigma',\sigma'}}{a_{\sigma,g}}\cdot\langle\sigma',-\xi_0\rangle=\frac{a_{\sigma+g+y_0-\sigma',\sigma'}}{a_{\sigma,g+y_0}}\cdot\omega(g), \quad \forall\ \sigma,\sigma'\in G\,.
\]
Taking $\sigma'=\sigma$, we obtain that
\[
\langle\sigma,-\phi_\alpha(g)\rangle\cdot\langle\sigma,-\xi_0\rangle=\langle\sigma,-\phi_\alpha(g+y_0)\rangle\,,
\]
which holds for any $\sigma\in G$
and the non-degeneracy of the Weil pairing forces that
\[
\xi_0=\phi_\alpha(y_0)\,.
\]
Then the proposition follows from Lemma \ref{lem:trivial-functors}.
\end{proof}

\section{Twisted Orlov functors and adjoint equivalences}

We now turn to the main technical construction of the paper: the twisted analogue of Orlov's Fourier--Mukai equivalence.
In the untwisted case, Orlov gave an equivalence between \(\mathrm{D}^{\mathrm{b}}(X\times \widehat{X})\) and \(\mathrm{D}^{\mathrm{b}}(X\times X)\):
\[
\bXi_X \coloneq \bR\mu_* \circ (\id \times \Phi_{\scP}) \colon \bdc(X \times \widehat{X}) \longrightarrow \bdc(X \times X)\,,
\]
where
$\id \times \Phi_{\scP} \colon \bdc(X \times \widehat{X}) \longrightarrow \bdc(X \times X)$ is the Fourier--Mukai transformation induced by the kernel $\pr_{13}^*\cO_{\Delta_X} \otimes \pr_{24}^*\scP$,
and \(\mu(x_1, x_2) = (x_2, x_1 + x_2)\). 
(Our convention differs slightly from Orlov's where a different $\mu$ is used in \cite{Orl02} as it uses a different order and pairing of the four abelian varieties.)
The functor $\bXi_X$ actually sends a skyscraper sheaf $k(x,\xi)$ to the kernel of $t_{x*}(-)\otimes \scP_{\xi}$, i.e.
\begin{equation}\label{eq:Orlov-point-to-auto}
\bXi_X(k(x,\xi))\cong\Gamma_{x*}\cO_X\otimes\pr^*_2\scP_{\xi}\,.
\end{equation}

Our goal here is to show that the same construction goes through when \(X\times \widehat{X}\) is replaced by the symplectic abelian variety \(A_{(X,\alpha)}\), at the cost of twisting the product category by \(\alpha^{-1}\boxtimes \alpha\). The resulting functor $\bXi_{X,\alpha}$, which we call the twisted Orlov functor, will satisfy the analogue of~\eqref{eq:Orlov-point-to-auto} stated in Proposition~\ref{prop:point-to-translation} and serve as the main bridge between the two sides of the Torelli argument.

\subsection{Setup}
\label{subsec:rep-functor}

Set \(A \coloneq A_{(X,\alpha)}\).
We first fix the cocycles and equivariant identifications needed to define the twisted Orlov functor.
\begin{enumerate}[label={(\arabic*)}]
    \item Consider the following Cartesian diagram,
        \[
        \begin{tikzcd}[sep=large]
        X \times \widehat{X} \ar[r, "\pr_1"] \ar[d, "\pi"'] & X \ar[d, "{[n]}"] \\
          A \ar[r, "q"'] & X \mathrlap{\,,}
        \end{tikzcd}
        \]
        which induces the commutative diagram:
        \[\begin{tikzcd}
        	{\rZ^2(X[n],\mathbb{G}_m)} & {\check{\rH}^2(\{{[n]_X}\colon X\to X\},\mathbb{G}_m)} & {\rH^2(X,\mathbb{G}_m)} \\
        	{\rZ^2(K,\mathbb{G}_m)} & {\check{\rH}^2(\pi\colon \{X\times\widehat{X}\to  A\},\mathbb{G}_m)} & {\rH^2(  A,\mathbb{G}_m)}\mathrlap{\,.}
        	\arrow[from=1-1, to=1-2]
        	\arrow["\cong"', from=1-1, to=2-1]
        	\arrow[from=1-2, to=1-3]
        	\arrow["{\pr_1^\ast}"', from=1-2, to=2-2]
        	\arrow["{q^*}"', from=1-3, to=2-3]
        	\arrow[from=2-1, to=2-2]
        	\arrow[from=2-2, to=2-3]
        \end{tikzcd}\]
        Hence the normalized $2$-cocycle
        \begin{equation*}
        (\pr_1^* a)_{k_1, k_2} \coloneq a_{\sigma_1, \sigma_2}, \quad \text{where} \;\; k_i = (\sigma_i, \phi_\alpha(\sigma_i)) \in K\,,
        \end{equation*}
        represents the class \(q^* \alpha\) on $A$.

    \item Under the automorphism $\mu$, the $2$-cocycle $\mu^*(a^{-1}\boxtimes a)\in \rZ^2(G\times G,\mathbb{G}_m)$ is defined by
        \begin{equation*}
        \bigl(\mu^*(a^{-1}\boxtimes a)\bigr)_{\tau_1,\tau_2}\coloneq \frac{a_{\sigma_1+{\sigma'}_1,\sigma_2+{\sigma'}_2}}{a_{{\sigma'}_1,{\sigma'}_2}},\quad\ \tau_i=(\sigma_i,\sigma'_i)\in G\times G,\; i =1,2\,.
        \end{equation*}

    \item According to Proposition \ref{prop:twisted-equivariant-equivalence}, we have canonical equivalences
    \begin{subequations}
    \label{eq:twisted-equi}
        \begin{align}
           \label{eq:twisted-equi-A}
           \bdc(A,q^*\alpha)&\simeq \bdc(X\times\widehat{X})_{K, \bvr, \pr_1^* a}\,,\\
           \label{eq:twisted-equi-auto}
           \bdc(X\times X,\alpha^{-1}\boxtimes\alpha)&\simeq\bdc(X\times X)_{G\times G, \bvr, a^{-1}\boxtimes a}\,,\\
           \label{eq:twisted-equi-middle}
           \bdc(X\times X,\mu^*(\alpha^{-1}\boxtimes\alpha)) &\simeq\bdc(X\times X)_{G\times G, \bvr, \mu^*(a^{-1}\boxtimes a)}\,.
        \end{align}
    \end{subequations}
\end{enumerate}

\subsection{Identifications of Equivariant Categories}

Our next goal is to construct a natural equivalence between twisted derived categories of $A$ and $X\times X$.
To establish equivalences between these derived categories, we first identify certain autoequivalences.
 In the classical case,  the derived equivalence \[\id\times \Phi_\scP\colon \bdc(X\times \widehat{X})\longrightarrow \bdc(X\times X)\] identifies $\Aut\bigl(\bdc(X\times \widehat{X})\bigr)\cong \Aut\bigl(\bdc(X\times X)\bigr)$.
For actions given by translations on $X\times \widehat{X}$,
this identification is realized as follows.

\begin{proposition}
\label{prop:poincare-exchange-transaltion}
There are canonical isomorphisms of autoequivalences
\[
(\id\times\boldsymbol{\Phi}_\mathscr{P})\circ t^*_{(y,\xi)}(-)\cong \bigl(t^*_{(y,0)}(-)\otimes\pr_2^*\mathscr{P}_{-\xi}\bigr)\circ(\id\times\Phi_\mathscr{P})
\]
for any $(y,\xi)\in X\times \widehat{X}$.
\end{proposition}

\begin{proof}
Note that we have  canonical isomorphisms between the following functors:
    \begin{align*}
        t^*_{(y,\xi)}(-)&\cong t^*_y(-)\times t^*_{\xi}(-)\,, \\
        t^*_{(y,0)}(-)\otimes \pr^*_2\mathscr{P}_{-\xi}&\cong t^*_y(-)\times (-)\otimes\mathscr{P}_{-\xi}\,, \\
        \Phi_\mathscr{P}\circ t^*_{\xi}(-)&\cong\Phi_\mathscr{P}(-)\otimes\mathscr{P}_{-\xi}\,.
    \end{align*}
Hence one can directly verify that
\begin{align*}
    (\id\times\Phi_\mathscr{P})\circ t^*_{(y,\xi)}(-)&\cong(\id\times\Phi_\mathscr{P})\circ\bigl(t^*_y(-)\times t^*_{\xi}(-)\bigr)\\
    &\cong\bigl(\id\circ t^*_y(-)\bigr)\times\bigl(\Phi_{\mathscr{P}}\circ t^*_{\xi}(-)\bigr)\\
    &\cong\bigl(t^*_y(-)\circ\id\bigr)\times\bigl(\Phi_\mathscr{P}(-)\otimes\mathscr{P}_{-\xi}\bigr)\\
    &\cong\bigl(t^*_y(-)\times (-)\otimes\mathscr{P}_{-\xi}\bigr)\circ(\id\times\Phi_\mathscr{P})\\
    &\cong\bigl(t^*_{(y,0)}(-)\otimes\pr_2^*\mathscr{P}_{-\xi}\bigr)\circ(\id\times\Phi_\mathscr{P})
\end{align*}
are canonically isomorphic.
\end{proof}

With the above identification, the action of $(K,\pr_1^\ast a)$ on $\bdc(X\times \widehat{X})$ can be identified as the action of $(G,\eta,a)$ on $\bdc(X\times X)$ given by the following data:
\begin{enumerate}[align=left,leftmargin=\parindent,labelindent=\parindent, listparindent=\parindent,labelwidth=0pt,itemindent=!,label=(\roman*)]
\item For each $\sigma\in G$, an autoequivalence $\eta_\sigma \colon \bdc(X\times X)\to \bdc(X\times X)$ defined by
\[
\eta_\sigma=t^*_{(\sigma,0)}(-)\otimes\pr^*_2\bigl(\mathscr{P}_{-\phi_{\alpha}(\sigma)}\bigr)\,;
\]
\item For $\sigma_1,\sigma_2\in G$, an isomorphism
\begin{align*}
      t^*_{(\sigma_1,0)}\Bigl(t^*_{(\sigma_2,0)}(-)\otimes\pr_2^*\bigl(\mathscr{P}_{-\phi_{\alpha}(\sigma_2)}\bigr)\Bigr)\otimes\pr^*_2\bigl(\mathscr{P}_{-\phi_{\alpha}(\sigma_1)}\bigr) \xlongrightarrow{\cong}   t^*_{(\sigma_1+\sigma_2,0)}(-)\otimes\pr^*_2\bigl(\mathscr{P}_{-\phi_{\alpha}(\sigma_1+\sigma_2)}\bigr)
\end{align*}
which is the canonical isomorphism twisted by $a_{\sigma_1,\sigma_2}$.
\end{enumerate}
Therefore, we obtain an equivalence of the corresponding equivariant categories:
\begin{equation}
\bdc(X\times \widehat{X})_{K, \bvr, \pr_1^*a} \simeq \bdc(X\times X)_{G,\eta,a}.
\label{eq:equivariant-FM}
\end{equation}

The next step is to relate the equivariant category $\bdc(X\times X)_{G,\eta,a}$ to the twisted derived category $\bdc\bigl(X\times X,\mu^*(\alpha^{-1}\boxtimes \alpha)\bigr)$,
established by the following key result.

\begin{proposition}
\label{prop:success-quotient}
There exists a canonical equivalence
\begin{equation}
\label{eq:success-quotient}
\bdc(X\times X)_{G,\eta,a} \simeq \bdc(X\times X)_{G\times G, \bvr, \mu^*(a^{-1}\boxtimes a)}\,,
\end{equation}
which induces an equivalence between $\bdc(X\times X)_{G,\eta,a}$ and $\bdc\bigl(X\times X,\mu^*(\alpha^{-1}\boxtimes \alpha)\bigr)$.

Moreover, we can define an equivalence
\[
(\id\times\Phi_{\mathscr{P}})_{G} \colon \bdc(A, q^*\alpha) \longrightarrow \bdc\bigl(X \times X, \mu^*(\alpha^{-1} \boxtimes \alpha)\bigr)
\]
that fits into the following commutative diagram of equivalences:
\begin{equation}
\begin{tikzcd}[row sep=2.5em,column sep=2.5em]
	{\bdc(  A,q^*\alpha)} && {\bdc\bigl(X\times X,\mu^*(\alpha^{-1}\boxtimes\alpha)\bigr)} \\
	{\bdc(X\times\widehat{X})_{K, \bvr, \pr_1^*a}} & {\bdc(X\times X)_{G, \eta,a}} & {\bdc(X\times X)_{G\times G, \bvr, \mu^*(a^{-1}\boxtimes a)}} \mathrlap{\,.}
	\arrow["(\id\times\Phi_{\mathscr{P}})_{G}", from=1-1, to=1-3]
	\arrow["(\ref{eq:twisted-equi-A})"', from=1-1, to=2-1]
	\arrow["(\ref{eq:equivariant-FM})", from=2-1, to=2-2]
	\arrow["(\ref{eq:success-quotient})", from=2-2, to=2-3]
	\arrow["(\ref{eq:twisted-equi-middle})"', from=2-3, to=1-3]
\end{tikzcd}
\end{equation}
\end{proposition}

\begin{proof}
We begin with the subgroup $0 \times G \subseteq G \times G$, and note that the quotient satisfies \[(G \times G)/(0 \times G) \cong G \times 0\,.\]
A key observation is that the cocycle becomes trivial when restricted to this subgroup, i.e.,
\[
\mu^\ast(a^{-1} \boxtimes a)|_{0 \times G} = 1\,.
\]
By Proposition \ref{prop:successive_equivariant}, this induces a natural action of $(G\times 0, \bar{\bvr}, \bar{a})$ on the equivariant category
\[
\bdc(X\times X)_{0\times G, \bvr|_{0\times G},  \mu^\ast(a^{-1}\boxtimes a)|_{0\times G}} = \bdc(X\times X)_{0\times G, \bvr}\,,
\]
yielding the canonical equivalence
\begin{equation}
\label{eq:induced-action}
\bdc(X \times X)_{G \times G, \bvr, \mu^\ast(a^{-1} \boxtimes a)} \simeq \bigl( \bdc(X \times X)_{0 \times G,\bvr} \bigr)_{G \times 0, \bar{\bvr}, \bar{a}}\,.
\end{equation}

The next objective is to identify this induced action with the original $(G, \eta, a)$-action on $\bdc(X\times X)$ via the natural equivalence
\begin{equation}
\label{eq:untwisted-quotient}
\Psi\colon \bdc(X\times X)_{0\times G, \bvr} \simeq \bdc(X\times (X/G)) \simeq \bdc(X\times X)\,.
\end{equation}
Following \cite[\S~3.4]{BO23}, the $(G \times 0, \bar{\bvr}, \bar{a})$-action can be specified by
\begin{itemize}
    \item \textbf{Autoequivalences}: For any $(\sigma, 0) \in G \times 0$, define:
        \[
        \begin{aligned}
            \bar{\bvr}_{(\sigma,0)} \colon \bdc(X\times X)_{0\times G, \bvr}&\longrightarrow \bdc(X\times X)_{0\times G, \bvr}\\
            \Bigl( E, \bigl\{ \theta_{(0,\sigma')} \bigr\} \Bigr) & \longmapsto \Bigl( t_{(\sigma,0)}^* E, \bigl\{ \langle\sigma',\phi_\alpha(\sigma)\rangle \cdot t_{(\sigma,0)}^*\bigl(\theta_{(0,\sigma')}\bigr) \bigr\} \Bigr)\,,
        \end{aligned}
        \]
        where the isomorphism  $\langle\sigma',\phi_\alpha(\sigma)\rangle \cdot t_{(\sigma,0)}^*\bigl(\theta_{(0,\sigma')}\bigr)$ is given by
        \[
        t_{(0,\sigma')}^* t_{(\sigma,0)}^* E \xlongrightarrow{a_{\sigma',\sigma}} t_{(\sigma,\sigma')}^* E \xlongrightarrow{a_{\sigma,\sigma'}^{-1}} t_{(\sigma,0)}^* t_{(0,\sigma')}^* E \xlongrightarrow{t_{(\sigma,0)}^*\bigl(\theta_{(0,\sigma')}\bigr)} t_{(\sigma,0)}^* E\,.
        \]
        Here,  the compatibility condition is ensured by the identity $\frac{a_{\sigma',\sigma}}{a_{\sigma,\sigma'}} = \langle\sigma',\phi_\alpha(\sigma)\rangle$.

    \item \textbf{Cocycle and Isomorphisms}: For $\sigma_1, \sigma_2 \in G$, we have:
        \[
        \bar{a}_{(\sigma_1,0),(\sigma_2,0)} = \mu^\ast(a^{-1}\boxtimes a)_{(\sigma_1,0),(\sigma_2,0)} = a_{\sigma_1,\sigma_2}\,.
        \]
        The isomorphism $\bar{\bvr}_{(\sigma_1,0)} \circ \bar{\bvr}_{(\sigma_2,0)} \to \bar{\bvr}_{(\sigma_1+\sigma_2,0)}$ is given by the natural identification $t_{(\sigma_1,0)}^* t_{(\sigma_2,0)}^* \cong t_{(\sigma_1+\sigma_2,0)}^*$, twisted by multiplication with $a_{\sigma_1,\sigma_2}$.
\end{itemize}

We now try to establish an identification between the above two actions with the equivalence $\Psi$, through the following commutative diagram:
\begin{equation}
\label{diag:identify-actions}
\begin{tikzcd}[column sep=8em,row sep=4em]
\bdc(X \times X)_{0\times G, \bvr} \ar[r, "\bar{\bvr}_{(\sigma,0)}"] \ar[d, "\Psi"'] & \bdc(X \times X)_{0\times G, \bvr} \ar[d, "{\Psi}"] \\
\bdc(X\times X) \ar[r, "t_{(\sigma,0)}^\ast(-)\otimes \pr_2^\ast\scP_{-\phi_\alpha(\sigma)}"] & \bdc(X\times X)\mathrlap{\,.}
\end{tikzcd}
\end{equation}

To establish such identification, first observe that the action of $\bar{\bvr}_{(\sigma,0)}$ on the equivariant category coincides with the functor
\[
t^*_{(\sigma,0)}(-)\otimes\bigl(\mathcal{O}_{X\times X},\chi_{-\phi_\alpha(\sigma)}\bigr)\colon \bdc(X \times X)_{0\times G, \bvr}\longrightarrow \bdc(X \times X)_{0\times G, \bvr}\,,
\]
where
\begin{itemize}
    \item $t^*_{(\sigma,0)}(-)$ is an autoequivalence induced by the canonical isomorphisms $t^*_{(\sigma,0)}t^*_{(0,g)}\cong t^*_{(0,g)}t^*_{(\sigma,0)}$;
    \item $\chi_{-\phi_\alpha(\sigma)}$ denotes the isomorphisms associated to the characters $\langle-,\phi_\alpha(\sigma)\rangle$ on $0\times G$, explicitly given by
        \[
        \chi_{-\phi_\alpha(\sigma)}(g)=\langle g,\phi_\alpha(\sigma)\rangle\colon t^*_{(0,g)}\mathcal{O}_{X\times X}\longrightarrow\mathcal{O}_{X\times X},\quad \forall (0,g)\in 0\times G\,.
        \]
\end{itemize}

Then since the equivalence $\Psi$ is naturally induced by the \'etale covering $\id\times[n]_X$, we obtain canonical correspondences between the functors $\bar{\bvr}_{(g,0)}$ and $t_{(\sigma,0)}^\ast(-)\otimes \pr_2^\ast\scP_{-\phi_\alpha(\sigma)}$ (cf. \cite[Example 3.4]{BO23}), establishing the commutativity of diagram \eqref{diag:identify-actions}.

Crucially, $\Psi$ identifies the $(G \times 0, \bar{\bvr}, \bar{a})$-action on $\bdc(X\times X)_{0\times G, \bvr}$ with the $(G, \eta, a)$-action on $\bdc(X\times X)$, as both of them are canonically twisted by the same $2$-cocycle $a\in\rZ^2(G,\mathbb{G}_m)$, which yields the equivalence
\begin{equation}
\label{eq:success-quotient-step1}
\bigl( \bdc(X \times X)_{0 \times G, \bvr} \bigr)_{G \times 0, \bar{\bvr}, \bar{a}} \simeq \bdc(X\times X)_{G,\eta,a}\,.
\end{equation}

Finally, with the equivalence \eqref{eq:induced-action} (and up to quasi-inverse), we arrive at the desired equivalence \eqref{eq:success-quotient}:
\[
\bdc(X\times X)_{G,\eta,a} \simeq \bdc(X\times X)_{G\times G, \bvr, \mu^*(a^{-1}\boxtimes a)}\,.
\]
\end{proof}

\subsection{The Twisted Orlov Functor}

In the untwisted case, as we have seen,
the Orlov functor $\bXi_X$ exhibits any point of $X\times \widehat{X}$ as an autoequivalence of $\bD^b(X)$.
We now define the twisted Orlov functor in the same spirit.

\begin{definition}
\label{def:rep_functor}
The \emph{twisted Orlov functor} is the composition:
\[
 \bXi_{X,\alpha}\colon \bdc(  A, q^*\alpha) \xlongrightarrow{(\id\times\Phi_{\mathscr{P}})_{G}} \bdc(X \times X, \mu^*(\alpha^{-1} \boxtimes \alpha)) \xlongrightarrow{\;\mu_*\;} \bdc(X \times X, \alpha^{-1} \boxtimes \alpha)\,.
\]
\end{definition}

The following result justifies that the functor we constructed is a generalization of the Orlov functor.
\begin{proposition}
\label{prop:point-to-translation}
For any $z\in A$, the Fourier--Mukai transform $\Psi:\bdc(X,\alpha)\to \bdc(X,\alpha)$ with kernel
\[
\bXi_{X,\alpha}\bigl(k(z)\bigr)\in\bdc(X\times X,\alpha^{-1}\boxtimes\alpha)
\]
is isomorphic to the autoequivalence
\[
\Phi_z\colon\bdc(X,\alpha)\longrightarrow\bdc(X,\alpha)
\]
defined in Proposition \ref{prop:embedding-A-to-Auto}.
\end{proposition}
\begin{proof}
Recall that $\bXi_{X,\alpha}$ fits into the following commutative diagram by construction:
\begin{equation*}
\begin{tikzcd}[column sep=5em,row sep=2.5em]
	{\bdc(A,q^*\alpha)} & {\bdc(X\times X,\alpha^{-1}\boxtimes\alpha)} \\
    {\bdc(X\times\widehat{X})_{K, \bvr, \pr_1^*a}} & {\bdc(X\times X)_{G\times G, \bvr, a^{-1}\boxtimes a}} \\
	  {\bdc(X\times X)_{G,\eta,a}} & {\bdc(X\times X)_{G\times G, \bvr, \mu^*(a^{-1}\boxtimes a)}}
	\arrow["{\bXi_{X,\alpha}}", from=1-1, to=1-2]
	\arrow["{\eqref{eq:twisted-equi-A}}"', from=1-1, to=2-1]
	\arrow["{\mu_*}", from=3-2, to=2-2]
	\arrow["{\eqref{eq:twisted-equi-auto}}", from=2-2, to=1-2]
	\arrow["{\eqref{eq:equivariant-FM}}"', from=2-1, to=3-1]
	\arrow["{\eqref{eq:success-quotient}}", from=3-1, to=3-2]
\end{tikzcd}
\end{equation*}
By chasing this diagram, we will derive an explicit description of $\bXi_{X,\alpha}(k(z))$ in the equivariant setting via the equivalence \eqref{eq:twisted-equi-auto}.

\begin{enumerate}[align=left,leftmargin=0pt,labelindent=\parindent,listparindent=\parindent,labelwidth=0pt,itemindent=!,topsep=0.5\parindent,label={\bfseries \arabic*).}]

\item (Downward vertical arrows)
For any skyscraper sheaf $k(z) \in \bdc(A,q^*\alpha)$ with $z \in A$,
write $z = [(y,\xi)]$ for some $(y,\xi) \in X \times \widehat{X}$,
under the equivalence \eqref{eq:twisted-equi-A} the skyscraper sheaf $k(z)$ corresponds to the linearized object
\[
\Inf_{K, \pr_1^*a}\bigl(k(y,\xi)\bigr) = \Bigl(\bigoplus_{g\in G} t^*_{(g,\phi_\alpha(g))}k(y,\xi), \bigl\{\bigoplus_{g\in G} a_{\sigma,g}\bigr\}_{\sigma\in G}\Bigr) \in \bdc(X\times \widehat{X})_{K, \bvr, \pr_1^*a}\,.
\]

Using the canonical isomorphism $(\id\times\Phi_{\mathscr{P}})(k(y,\xi)) \cong k(y) \boxtimes \mathscr{P}_{\xi}$, the equivariant Fourier--Mukai transform \eqref{eq:equivariant-FM} maps this object to
\[
\Inf_{\eta,a}(k(y)\boxtimes\mathscr{P}_{\xi}) = \Bigl(\bigoplus_{g\in G} t^*_{(g,0)}(k(y)\boxtimes\mathscr{P}_\xi) \otimes \pr_2^*\mathscr{P}_{-\phi_\alpha(g)}, \bigl\{\bigoplus_{g\in G} a_{\sigma,g}\bigr\}_{\sigma\in G}\Bigr) \in \bdc(X\times X)_{G,\eta,a}\,.
\]

\item (Bottom horizontal arrow)
We now identify the above object $\Inf_{\eta,a}(k(y)\boxtimes\mathscr{P}_{\xi})$ to an object in $\bdc(X\times X)_{G\times G, \bvr, \mu^*(a^{-1}\boxtimes a)}$ via the equivalence \eqref{eq:success-quotient}.
Recall that in the proof of Proposition \ref{prop:success-quotient}, the equivalence is constructed by the following composition
\[
\bdc(X\times X)_{G,\eta,a}\xlongrightarrow{\eqref{eq:success-quotient-step1}}\bigl( \bdc(X \times X)_{0 \times G, \bvr} \bigr)_{G \times 0, \bar{\bvr}, \bar{a}}\xlongrightarrow{\eqref{eq:induced-action}}\bdc(X\times X)_{G\times G, \bvr, \mu^*(a^{-1}\boxtimes a)}\,.
\]
Since $\Inf_{\eta,a}(k(y)\boxtimes\mathscr{P}_{\xi})$ consists of an object in $\bdc(X\times X)$, simply denoted by $\mathcal{G}$, and the $(G,a)$-linearization on it,
under the equivalence \eqref{eq:success-quotient-step1},
$\mathcal{G}$ is sent to
\begin{align*}
(\id\times[n]_X)^*\Bigl(\bigoplus_{g\in G} t^*_{(g,0)}(k(y)\boxtimes\mathscr{P}_{\xi}) \otimes \pr_2^*\mathscr{P}_{-\phi_\alpha(g)}\Bigr)
&\cong \bigoplus_{g\in G} t^*_{(g,0)}(k(y)\boxtimes\mathscr{P}_\nu) \otimes \pr_2^*\mathcal{O}_X \\
&\cong \bigoplus_{g\in G} t^*_{(g,0)}(k(y)\boxtimes\mathscr{P}_\nu) \\
&\cong \bigoplus_{g\in G} t^*_{(g,0)}{i_y}_*\mathscr{P}_\nu\,,
\end{align*}
carrying the $(0\times G)$-linearization
\[
\theta_{0,\sigma'}\coloneq \bigoplus_{g\in G} \Bigl(t^*_{(g,0)}{i_y}_*\chi_\xi(\sigma') \cdot \langle\sigma',\phi_\alpha(g)\rangle\Bigr)\colon t^*_{(0,\sigma')}\Bigl(\bigoplus_{g\in G} t^*_{(g,0)}{i_y}_*\mathscr{P}_\nu\Bigr) \longrightarrow \bigoplus_{g\in G} t^*_{(g,0)}{i_y}_*\mathscr{P}_\nu
\]
for each $\sigma'\in G$,
where $\nu = [n]_{\widehat{X}}(\xi)$ and $i_y: X \to X\times X$ is the map  $i_y(x) = (y,x)$,
and the $(G,a)$-linearization of $\mathcal{G}$ is sent to a $(G\times0,\bar{a})$-linearization of
\[
\Bigl(\bigoplus_{g\in G} t^*_{(g,0)}{i_y}_*\mathscr{P}_\nu, \{\theta_{0,\sigma'}\}_{\sigma'\in G} \Bigr)\in \bdc(X\times X)_{0\times G, \bvr}\,.
\]
Then by the construction of the $(G\times0,\bar{\bvr},\bar{a})$-action, the equivalence \eqref{eq:induced-action} sends these data to an object in $\bdc(X\times X)_{G\times G, \bvr, \mu^*(a^{-1}\boxtimes a)}$ given by
\begin{equation}\label{obj:X*X,0*G}
\Bigl(\bigoplus_{g\in G} t^*_{(g,0)}{i_y}_*\mathscr{P}_\nu, \{\theta_{\sigma,\sigma'}\}_{(\sigma,\sigma')\in G\times G}\Bigr)\,,
\end{equation}
where the linearization maps are given by the composition
\begin{align*}
\theta_{\sigma,\sigma'} \colon & t^*_{(\sigma,\sigma')}\Bigl(\bigoplus_{g\in G} t^*_{(g,0)}{i_y}_*\mathscr{P}_\nu\Bigr) \xrightarrow{\oplus_{g\in G} a_{\sigma,\sigma'}^{-1}} t^*_{(\sigma,0)}t^*_{(0,\sigma')}\Bigl(\bigoplus_{g\in G} t^*_{(g,0)}{i_y}_*\mathscr{P}_\nu\Bigr) \\
& \quad \xrightarrow{\oplus_{g\in G} t^*_{(\sigma,0)}(\theta_{0,\sigma'})} t^*_{(\sigma,0)}\Bigl(\bigoplus_{g\in G} t^*_{(g,0)}{i_y}_*\mathscr{P}_\nu\Bigr) \xrightarrow{\oplus_{g\in G} a_{\sigma,g}} \bigoplus_{g\in G} t^*_{(g+\sigma,0)}{i_y}_*\mathscr{P}_\nu\,.
\end{align*}

\item (Upward vertical arrows)
Applying $\mu_*$ to \eqref{obj:X*X,0*G}, we obtain the canonical isomorphism
\[
\mu_*\Bigl(\bigoplus_{g\in G} t^*_{(g,0)}{i_y}_*\mathscr{P}_\nu\Bigr) \cong \bigoplus_{g\in G} t_{(0,g)}^*{\Gamma_y}_*\mathscr{P}_\nu\,,
\]
with linearization
\begin{equation*}
\bigoplus_{g\in G} \frac{a_{\sigma,g} \cdot a_{\sigma',g}}{a_{\sigma,\sigma'} \cdot a_{g,\sigma'}} \cdot t^*_{(0,g+\sigma)}{\Gamma_y}_*\bigl(\chi_\xi(\sigma')\bigr)\colon t^*_{(\sigma',\sigma+\sigma')}\Bigl(\bigoplus_{g\in G} t_{(0,g)}^*{\Gamma_y}_*\mathscr{P}_\nu\Bigr) \longrightarrow \bigoplus_{g\in G} t_{(0,g+\sigma)}^*{\Gamma_y}_*\mathscr{P}_\nu\,.
\end{equation*}
Since $a \in \rZ^2(G,\mathbb{G}_m)$ is a 2-cocycle, we have the identity:
\[
\frac{a_{\sigma,g} \cdot a_{\sigma',g}}{a_{\sigma,\sigma'} \cdot a_{g,\sigma'}} = \frac{a_{\sigma+\sigma',g}}{a_{\sigma+g,\sigma'}}\,, \quad \text{for all } \sigma,\sigma',g \in G\,.
\]
Therefore, under the equivalence \eqref{eq:twisted-equi-auto}, the object $\bXi_{X,\alpha}(k(z))$ is isomorphic to
\[
\biggl(\bigoplus_{g\in G} t^*_{(0,g)}{\Gamma_y}_*\mathscr{P}_\nu, \Bigl\{\bigoplus_{g\in G} \frac{a_{\sigma',g}}{a_{\sigma'+g-\sigma,\sigma}} \cdot t^*_{(0,\sigma'+g-\sigma)}{\Gamma_y}_*\bigl(\chi_\xi(\sigma)\bigr)\Bigr\}_{(\sigma,\sigma')\in G\times G}\biggr)\,,
\]
which is precisely the right inflation of ${\Gamma_y}_*\mathscr{P}_\nu$, equipped with a $G_\Delta$-linearization induced by $\mathscr{P}_\xi$.
Similar to Example \ref{ex:identity-functor} concerning $\mathcal{O}_\Delta$, one can show that the left and right inflations of a $G_\Delta$-linearized object are isomorphic. This establishes the desired isomorphism between $\Psi$ and $\Phi_z$.
\end{enumerate}
\end{proof}

\subsection{Adjoint of Generalized Functors}

We recall the adjoint construction, which is an important step in converting any given derived equivalence to a geometric equivalence for abelian varieties in untwisted cases.
We now explain such constructions in general twisted settings.
For any Fourier--Mukai transform \[\Phi_\scE \colon \bdc(X_1,\alpha_1) \xlongrightarrow{\sim} \bdc(X_2,\alpha_2)\] with the kernel $\scE \in \bdc(X_1\times X_2,\alpha_1^{-1}\boxtimes \alpha^{\vphantom{2}}_2)$,

let $\scF\in \bdc(X_1\times X_2,\alpha^{\vphantom{1}}_1\boxtimes \alpha_2^{-1})$ be the kernel satisfying $\Psi_{\scF} \cong \Phi_{\scE}^{-1}$.
\begin{definition}
    Combining the two equivalences given by kernels $\scE$ and $\scF$, there is an equivalence on the products, which is called the \emph{adjoint} of $\Phi_{\scE}$,
    \begin{equation}
    \label{eq:adjoint-equiv}
    \mathbf{Ad}_{\Phi_{\scE}} \colon \bdc(X_1 \times X_1, \alpha_1^{-1} \boxtimes \alpha^{\vphantom{1}}_1) \xlongrightarrow{\sim} \bdc(X_2 \times X_2, \alpha_2^{-1} \boxtimes \alpha^{\vphantom{2}}_2)\,,
    \end{equation}
    where $\bAd_{\Phi_{\scE}}$ is the Fourier--Mukai transform with kernel $\scF\boxtimes \scE$.
\end{definition}

A simple but important property is
\begin{lemma}
    For any object $\scH\in\bdc(X_1 \times X_1, \alpha_1^{-1} \boxtimes \alpha^{\vphantom{1}}_1)$, set  \[ \scH'=\bAd_{\Phi_{\scE}}(\scH) \in\bdc(X_2 \times X_2, \alpha_2^{-1} \boxtimes \alpha^{\vphantom{2}}_2)\,.\]
    There is a commutative diagram
\[\begin{tikzcd}[sep=large]
	{\bdc(X_1,\alpha_1)} & {\bdc(X_2,\alpha_2)} \\
	{\bdc(X_1,\alpha_1)} & {\bdc(X_2,\alpha_2)} \mathrlap{\,,}
	\arrow["{\Phi_{\scE}}", from=1-1, to=1-2]
	\arrow["{\Phi_{\scH}}"', from=1-1, to=2-1]
	\arrow["{\Phi_{\scH'}}", from=1-2, to=2-2]
	\arrow["{\Phi_{\scE}}", from=2-1, to=2-2]
\end{tikzcd}\]
up to isomorphism.
\end{lemma}
\begin{proof}
The assertion follows from a direct computation of the convolution of Fourier--Mukai kernels, which can be established by adapting the same lines in \cite[§1.~(1.6)]{Orl02} for the untwisted case.
\end{proof}

Recall that for any $z=[(y,\xi)]\in A$, according to Proposition \ref{prop:embedding-A-to-Auto}, we can associate an autoequivalence
\[
\Phi_z \colon \bdc(X,\alpha)\longrightarrow\bdc(X,\alpha)\,.
\]
Then applying the above construction to the autoequivalence $\Phi_z$, we obtain the associated adjoint autoequivalence
\[
\bad_z \coloneq \bAd_{\Phi_z} \colon \bdc(X \times X, \alpha^{-1} \boxtimes \alpha) \longrightarrow \bdc(X \times X, \alpha^{-1} \boxtimes \alpha)\,.
\]

\begin{proposition}
\label{prop:quotient-derived-commute}
For any $z \in A$, the following diagram commutes up to isomorphism:
\begin{equation}
\label{diag:pointwise-line-bundle}
\begin{tikzcd}[sep=huge]
    \bdc(A, q^*\alpha)
    \ar[r, "{(-) \otimes \scP_{\psi_{\alpha}(z)}}"]
    \ar[d, "\bXi_{X,\alpha}"'] &
    \bdc(A, q^*\alpha)
    \ar[d, "\bXi_{X,\alpha}"] \\
    \bdc(X \times X, \alpha^{-1} \boxtimes \alpha)
    \ar[r, "\bad_z"] &
    \bdc(X \times X, \alpha^{-1} \boxtimes \alpha)\mathrlap{\,,}
\end{tikzcd}
\end{equation}
where $\psi_\alpha\colon A\to\widehat{A}$ is the isomorphism \eqref{eq:symplectic-isomorphism}.
\end{proposition}
\begin{proof}
The proof proceeds in the following two steps.
\begin{enumerate}[align=left,leftmargin=0pt,labelindent=0pt,listparindent=\parindent, labelwidth=0pt,itemindent=!,topsep=1em,itemsep=1em,label={\underline{\bfseries{Step \arabic*.}}}]
\item
We claim that for any $z\in A$,  the functor $\bad_z$ is isomorphic to the tensor product with a line bundle under conjugation by $\bXi_{X,\alpha}$, which will be denoted by $\widetilde{L}_z$.
Such an assignment will further define a morphism of abelian varieties,
in particular, a homomorphism of group schemes
\begin{equation}
\begin{aligned}
    \psi \colon   A &\longrightarrow  \widehat{  A}\subseteq \Aut\bigl(\bdc( A,q^\ast\alpha)\bigr) \\
    z & \longmapsto \widetilde{L}_z \quad (\mathcal{X} \longmapsto \mathcal{X}\otimes \widetilde{L}_z) \,.
\end{aligned}
\end{equation}

Let $\scE \in \bdc(  A\times X\times X, q^*\alpha^{-1}\boxtimes \alpha^{-1}\boxtimes\alpha)$ be the kernel of $\bXi_{X,\alpha}$, $\scF \in \bdc(  A\times X\times X, q^*\alpha \boxtimes \alpha\boxtimes\alpha^{-1})$ be the  kernel of $\bXi_{X,\alpha}^{-1}$, i.e. $\Psi_\scF \cong \bXi_{X,\alpha}^{-1}$,
and $\widetilde{\bAd}$ be the Fourier--Mukai transform with  kernel
\[\pr^*_{14}\mathcal{O}_{\Delta} \otimes \pr^*_{136}\scE \otimes \pr^*_{125}\scF[-\dim X] \] in
$\bdc(A\times X\times X \times A\times X\times X,1\boxtimes\alpha\boxtimes\alpha^{-1}\boxtimes 1\boxtimes \alpha^{-1}\boxtimes\alpha)$.  We can obtain an autoequivalence
$\Psi\in\Aut(\bdc(A\times A, 1\boxtimes q^*\alpha ))$ defined by
\begin{equation}
\label{diag:family-line-bundle}
        \begin{tikzcd}[sep=huge]
        \bdc(A\times A, 1\boxtimes q^*\alpha)
            \ar[r, " \Psi"]
            \ar[d, "\id \times \bXi_{X,\alpha}"'] &
        \bdc(A\times A, 1\boxtimes q^*\alpha)
            \ar[d, "\id \times \bXi_{X,\alpha}"] \\
        \bdc(A \times X \times X, 1\boxtimes \alpha^{-1} \boxtimes \alpha)
            \ar[r, "\widetilde{\bAd}"] &
        \bdc(A \times X \times X, 1\boxtimes \alpha^{-1} \boxtimes \alpha)\mathrlap{\,.}
    \end{tikzcd}
\end{equation}

Suppose that $z\in A$ and $\mathscr{G}\in\bdc(X\times X, \alpha^{-1}\boxtimes\alpha)$.
Denote $i_z\colon X\times X\to A\times X\times X$, $x\mapsto(z,x)$.
By Proposition~\ref{prop:point-to-translation} we have
\[
\scE_z\coloneqq i^*_z\scE\cong\bXi_{X,\alpha}\bigl(k(z)\bigr)\,,\quad\Phi_{\scE_z}\cong\Phi_z\,.
\]
And $\scF_z\coloneqq i^*_z\scF$ satisfies $\scF_z[-\dim X]\cong\scE^\vee_z[\dim X]$, $\Psi_{\scF_z[-\dim X]}\cong\Phi^{-1}_z$ by the adjoint kernel formula. 
Then applying $\widetilde{\bAd}$ to $k(z)\boxtimes\scG$ gives
\begin{align*}
    \widetilde{\bAd}(k(z)\boxtimes\scG) & \cong \widetilde{\bAd}(i_{z*}\scG)\\
    & \cong i_{z*}\pr_{34*}(\pr^*_{12}\scG\otimes\pr^*_{24}\scE_z\otimes\pr^*_{13}\scF_z[-\dim X])\\
    & \cong k(z)\boxtimes\bad_z(\scG)\,.
\end{align*}
Here the projections in the second row are from $X\times X\times X\times X$.
Thus, $\widetilde{\bAd}$ gives a family of adjoint actions parametrized by $A$.

For any twisted skyscraper sheaf $k(z,z') \in \bdc(A\times A,1\boxtimes q^*\alpha)$ with $z$, $z'\in A$, a direct computation using Proposition \ref{prop:point-to-translation} shows that the image of $k(z,z')$ under $\Psi$ is
\begin{align*}
      \Psi\bigl(k(z,z')\bigr) & \cong  (\id\times\bXi_{X,\alpha})^{-1} \circ \widetilde{\bAd}\bigl(k(z)\boxtimes \bXi_{X,\alpha}\bigl(k(z')\bigr)\bigr)\\
& \cong  (\id\times\bXi_{X,\alpha})^{-1} \bigl(k(z) \boxtimes \bad_z\bigl(\bXi_{X,\alpha}(k(z'))\bigr)\bigr)\\
& \cong k(z)\boxtimes k(z') \\
& \cong k(z,z')\,,
\end{align*}
in other words, each $k(z,z')$ is invariant under $\Psi$.
By \cite[Corollary 5.3]{CS07},  $\Psi$ is given by the tensor product with a line bundle $\widetilde{L}\in \Pic(A\times A)$.
Moreover, the restriction $\widetilde{L}|_{\{z\}\times A}$ is isomorphic to the line bundle associated to $\bXi_{X,\alpha}^{-1}\circ\bad_z\circ\bXi_{X,\alpha}$, that is
\[
\psi(z) = \widetilde{L}_z \cong \widetilde{L}|_{\{z\}\times A} \in \Pic(A)\,.
\]

Since $z = 0$ corresponds to $\bad_0 = \id$ and thus \[\widetilde{L}|_{\{0\}\times A} \cong \widetilde{L}_0= \psi(0) \cong \mathcal{O}_{A} \in \Pic^0(A) = \widehat{A}\,,\]
so we actually have that $\psi\colon  A \to \widehat{ A}$ is a morphism between algebraic varieties, due to $\psi(0) \cong \mathcal{O}_A = \widehat{0} \in \widehat{A}$ and rigidity.
Moreover, $\psi$ is indeed a homomorphism,
since any morphism between abelian varieties is a homomorphism up to translation.

\item
We now check that the morphism $\psi$ is exactly the isomorphism $\psi_{\alpha}$ induced by the symplectic biextension on $ A^2$.

Recall the isomorphism $\bXi_{X,\alpha}\bigl((-)\otimes\widetilde{L}_z\bigr)\cong\bad_z\circ\bXi_{X,\alpha}$,
both sides are viewed as functors between the equivariant categories under the equivalences \eqref{eq:twisted-equi}.
Applying them to the (left) inflation of the structure sheaf $\mathcal{O}_{X\times\widehat{X}}$, i.e.,
\[
\mathscr{G}\coloneq\Inf_{K, \pr_1^*a}(\mathcal{O}_{X\times \widehat{X}})=\Bigl(\bigoplus_{g\in G}t^*_{(g,\phi_\alpha(g))}\mathcal{O}_{X\times\widehat{X}}\,,\zeta\Bigr)\in\bdc(X\times\widehat{X})_{K, \bvr, \pr_1^*a}\,,
\]
this leads to an isomorphism between two equivariant objects in $\bdc(X\times X)_{G\times G, \bvr, a^{-1}\boxtimes a}$.
Moreover, there is an isomorphism between the images of $\bXi_{X,\alpha}(\mathscr{G}\otimes\widetilde{L}_z)$ and $\bad_z\circ\bXi_{X,\alpha}(\mathscr{G})$ under the forgetful functors, which are objects in $\bdc(X\times X)$.

To get further information of these two objects in $\bdc(X\times X)$,
we rely on the explicit construction of the equivariant categories and functors,
which implies the following commutative diagram (up to isomorphism):
\begin{equation}\label{diag:forgetful-commute}
\begin{tikzcd}[scale cd=0.92]
	{\bdc(X\times \widehat{X})} &&& {\bdc(X\times X)} \\
	& {\bdc(X\times \widehat{X})_{K, \bvr, \pr_1^*a}} & {\bdc(X\times X)_{G\times G, \bvr, a^{-1}\boxtimes a}} \\
	& {\bdc(X\times \widehat{X})_{K, \bvr, \pr_1^*a}} & {\bdc(X\times X)_{G\times G, \bvr, a^{-1}\boxtimes a}} \\
	{\bdc(X\times \widehat{X})} &&& {\bdc(X\times X)}\mathrlap{\,.}
	\arrow["{\mu_*(\id\times[n]_X)^*\circ(\id\times\Phi_\mathscr{P})}", from=1-1, to=1-4]
	\arrow["{(-)\otimes\mathscr{P}_{(\xi',y')}}", from=1-1, to=4-1]
	\arrow["{\bAd_{(y,\xi)}}", from=1-4, to=4-4]
	\arrow["{\bUp}"{description}, from=2-2, to=1-1]
	\arrow["{\bXi_{X,\alpha}}", from=2-2, to=2-3]
	\arrow["{(-)\otimes\widetilde{L}_z}"', from=2-2, to=3-2]
	\arrow["{\bUp}"{description}, from=2-3, to=1-4]
	\arrow["{\bad_z}", from=2-3, to=3-3]
	\arrow["{\bXi_{X,\alpha}}", from=3-2, to=3-3]
	\arrow["{\bUp}"{description}, from=3-2, to=4-1]
	\arrow["{\bUp}"{description}, from=3-3, to=4-4]
	\arrow["{\mu_*(\id\times[n]_X)^*\circ(\id\times\Phi_\mathscr{P})}", from=4-1, to=4-4]
\end{tikzcd}
\end{equation}
Recall that $\boldsymbol{\Upsilon}$ are the forgetful functors (see Definition \ref{def:equivariant_category}),
$\mathscr{P}_{(\xi',y')}\cong\pi^*\widetilde{L}_z\in\Pic^0(X\times\widehat{X})$, and $\bAd_{(y,\xi)}$ is the adjoint of
\[
{t_y}_*(-)\otimes\mathscr{P}_\nu\colon\bdc(X)\longrightarrow\bdc(X)\,,
\]
where $(y,\xi)\in X\times\widehat{X}$ is a representative of $z\in  A$ and $\nu=[n]_{\widehat{X}}(\xi)$.
Note that the diagram is well defined since different choices of the representative give naturally isomorphic functors.

Hence, according to the commutativity of \eqref{diag:forgetful-commute},
for the inflation $\mathscr{G}$,
there is an isomorphism between the object
\begin{align*}
\bUp\circ\bXi_{X,\alpha}(\mathscr{G}\otimes\widetilde{L}_z)&\cong\mu_*(\id\times[n]_X)^*\circ(\id\times\Phi_{\mathscr{P}})\Bigl(\bigoplus_{g\in G}\mathscr{P}_{(\xi',y')}\Bigr)\\
&\cong\mu_*(\id\times[n]_X)^*\Bigl(\bigoplus_{g\in G}\mathscr{P}_{\xi'}\boxtimes k(-y')[-\dim X]\Bigr)\\
&\cong\Bigl(\bigoplus_{g\in G}k(g+w)\Bigr)\boxtimes\Bigl(\bigoplus_{g\in G}\mathscr{P}_{\xi'}\Bigr)[-\dim X]
\end{align*}
where $[n]_X(w)=-y'$, and the object
\begin{align*}
\bUp\circ\bad_z\circ\bXi_{X,\alpha}(\mathscr{G})\cong&\bAd_{(y,\xi)}\biggl(\Bigl(\bigoplus_{g\in G}k(g)\Bigr)\boxtimes\Bigl(\bigoplus_{g\in G}\mathcal{O}_X\Bigr)[-\dim X]\biggr)\\
\cong&\ \Bigl(\bigoplus_{g\in G}k(y+g)\Bigr)\boxtimes\Bigl(\bigoplus_{g\in G}\mathscr{P}_\nu\Bigr)[-\dim X]\,.
\end{align*}
Comparing the supports in the first factor and the line bundles in the second factor, this implies that the morphism $\psi: A\to\widehat{ A}$ satisfies
\[
\widehat{\pi}\circ\psi\circ\pi=n\psi_{L_X}:X\times \widehat{X}\longrightarrow \widehat{X}\times X\,.
\]
Finally one obtains $\psi=\psi_{\alpha}$ due to the uniqueness of the descent of the biextension (\cf\cite[Proposition 10.4]{Pol03}).
\end{enumerate}
\end{proof}

\section{Twisted derived Torelli}
With the twisted Orlov functors, we are ready to prove Theorem~1.1. 
Given a twisted derived equivalence between \((X_1,\alpha_1)\) and \((X_2,\alpha_2)\), the key idea is to use \(\Xi_{X_i,\alpha_i}\) to transport its adjoint to an equivalence between twisted derived categories of the associated symplectic abelian varieties.
We then show that the transported equivalence is geometric---i.e., induced by an isomorphism of the underlying symplectic abelian varieties.
We also give an application of Theorem \ref{thm:main}: for \(g\ge 2\), there exist twisted abelian varieties that are not twisted derived equivalent to their duals.

\subsection{Proof of Theorem \ref{thm:main}}

For simplicity, we adopt the abbreviated notation:
\[
A_i \coloneqq A_{(X_i,\alpha_i)} \quad (i = 1,2)\,,
\]
where each \(A_i\) carries a symplectic biextension \(L_i \in \Pic(A_i \times A_i)\), with natural projections \(q_i \colon A_i \to X_i\) and \(\pi_i \colon X_i \times \widehat{X}_i \to A_i\).

Let \(\Phi \colon \bdc(X_1, \alpha_1) \xrightarrow{\sim} \bdc(X_2, \alpha_2)\) be a Fourier--Mukai equivalence. We now construct a symplectic isomorphism between \(A_1\) and \(A_2\) from \(\Phi\) through the following steps:

\begin{enumerate}[align=left,leftmargin=0pt,labelindent=0pt,listparindent=\parindent,labelwidth=0pt,itemindent=!,topsep=1em,itemsep=1em,label={\bfseries \underline{Step \arabic*.}}]
\item We first construct a derived equivalence $\Psi\colon \bdc( A_1,q_1^\ast\alpha_1)\simeq \bdc( A_2,q_2^\ast\alpha_2)$.
Set \[\bXi_i \colon \bdc(A_i, q_i^*\alpha_i) \xlongrightarrow{\sim} \bdc(X_i \times X_i, \alpha_i^{-1} \boxtimes \alpha^{\vphantom{i}}_i)\] to be the twisted Orlov functor.

For the Fourier--Mukai transform $\Phi \colon \bdc(X_1,\alpha_1) \xrightarrow{\sim} \bdc(X_2,\alpha_2)$
 as in \eqref{eq:adjoint-equiv}, we have an associated adjoint equivalence
 \[\bAd_{\Phi} \colon \bdc(X_1 \times X_1, \alpha_1^{-1} \boxtimes \alpha^{\vphantom{1}}_1) \xlongrightarrow{\sim} \bdc(X_2 \times X_2, \alpha_2^{-1} \boxtimes \alpha^{\vphantom{2}}_2)\,.\]
Then the composition of equivalences
\[
\Psi \coloneqq \bXi_2^{-1} \circ \bAd_{\Phi} \circ \bXi_1
\]
fits into a commutative diagram (up to natural equivalence):
\begin{equation}
\label{diag:adjoint-between-two}
\begin{tikzcd}[column sep=large,row sep=large]
\bdc(A_1, q_1^*\alpha_1)
    \ar[r, "\Psi"]
    \ar[d, "\bXi_1"'] &
\bdc(A_2, q_2^*\alpha_2)
    \ar[d, "\bXi_2"] \\
\bdc(X_1 \times X_1, \alpha_1^{-1} \boxtimes \alpha^{\vphantom{1}}_1)
    \ar[r, "\bAd_{\Phi}"] &
\bdc(X_2 \times X_2, \alpha_2^{-1} \boxtimes \alpha^{\vphantom{2}}_2)\mathrlap{\,.}
\end{tikzcd}
\end{equation}

\item
We show \(\Psi\) is induced by an isomorphism of twisted abelian varieties.

\begin{proposition}
\label{prop:orlov-twisted}
    There exist an isomorphism of twisted abelian varieties
    \[
    \psi\colon (A_1,q^*_1\alpha_1)\longrightarrow (A_2,q^*_2\alpha_2)\,,
    \]
    and a line bundle $\mathscr{N}\in \Pic(A_2)$ such that
    $\Psi(-) \simeq \psi_*(-)\otimes \mathscr{N}$.
\end{proposition}
\begin{proof}
The argument follows the same strategy as in the untwisted case in \cite[Proposition 9.39]{Huy06}.
We give a sketch here.

According to Proposition \ref{prop:point-to-translation},
the functor \(\bXi_i\) sends twisted skyscraper sheaves \(k(z_i)\) to kernels of autoequivalences \(\Phi_{z_i}\). In particular, $\bXi_i\bigl(k(0_{A_i})\bigr)$ is isomorphic to the kernel of $\id_{\bdc(X_i,\alpha_i)}$, where $0_{A_i}\in A_i$ is the neutral element.
Hence $\Psi\bigl(k(0_{A_1})\bigr)\cong k(0_{A_2})$.
It follows from \cite[Corollary 6.14]{Huy06} that there exists an open neighborhood $0_{A_1}\in U\subset A_1$ such that $\Psi$ sends any twisted skyscraper sheaf supported on $U$ to a twisted skyscraper sheaf on $A_2$.
Indeed, the proof of loc. cit. applies verbatim in the twisted setting.

Note that for any $z_1$, $w_1\in U$, $\Psi\bigl(k(z_1+w_1)\bigr)$ is also isomorphic to a twisted skyscraper sheaf due to the following commutative diagram
\[\begin{tikzcd}[row sep=4em,column sep=8em]
	{\bdc(X_1,\alpha_1)} & {\bdc(X_1,\alpha_1)} & {\bdc(X_1,\alpha_1)} \\
	{\bdc(X_2,\alpha_2)} & {\bdc(X_2,\alpha_2)} & {\bdc(X_2,\alpha_2)}\mathrlap{\,,}
	\arrow["{\Phi_{z_1}}"', from=1-1, to=1-2]
	\arrow["{\Phi_{(z_1+w_1)}}", curve={height=-24pt}, from=1-1, to=1-3]
	\arrow["\Phi", from=1-1, to=2-1]
	\arrow["{\Phi_{w_1}}"', from=1-2, to=1-3]
	\arrow["\Phi", from=1-2, to=2-2]
	\arrow["\Phi", from=1-3, to=2-3]
	\arrow["{\Phi_{z_2}}", from=2-1, to=2-2]
	\arrow["{\Phi_{\Ad_{\Phi}\circ\Xi_1\bigl(k(z_1+w_1)\bigr)}}"', curve={height=24pt}, from=2-1, to=2-3]
	\arrow["{\Phi_{w_2}}", from=2-2, to=2-3]
\end{tikzcd}\]
where $k(z_2)\cong\Psi\bigl(k(z_1)\bigr)$ and $k(w_2)\cong\Psi\bigl(k(w_1)\bigr)$.
These observations are sufficient to imply that $\Psi$ actually sends any twisted skyscraper sheaf on $ A_1$ to a twisted skyscraper sheaf on $ A_2$, see the step iii) in the proof of \cite[Proposition 9.39]{Huy06}.
Therefore the functor $\Psi$ is induced by a morphism $\psi\colon  A_1 \to  A_2$ and a line bundle $\mathscr{N}\in\Pic( A_2)$,
see \cite[Corollary 5.3]{CS07} and \cite[Corollary 5.23]{Huy06} for more details.
Moreover, $\psi$ is an isomorphism since $\Psi$ is an equivalence.
\end{proof}

\item
We now verify that $\psi$ preserves the symplectic structures between $A_1$ and $A_2$.
By definition, $\psi$ is symplectic if and only if \[(\psi \times \psi)^*L_2\cong L_1\,.\]
According to the construction of $\psi$,  this is equivalent to the fact that the following diagram:
\begin{equation}
\label{diag:symplectic}
\begin{tikzcd}[sep = huge]
	{\bdc(A_1,q_1^*\alpha_1)} & {\bdc(A_2,q_2^*\alpha_2)} \\
	{\bdc(A_1,q_1^*\alpha_1)} & {\bdc(A_2,q_2^*\alpha_2)}\mathrlap{\,.}
	\arrow["{\Psi}", from=1-1, to=1-2]
	\arrow["{(-)\otimes\scP_{A_1,\widehat{z}_1}}"', from=1-1, to=2-1]
	\arrow["{(-)\otimes\scP_{A_2,\widehat{z}_2}}", from=1-2, to=2-2]
	\arrow["{\Psi}", from=2-1, to=2-2]
\end{tikzcd}
\end{equation}
is commutative for any $z_1\in A_1$ and $\psi(z_1)=z_2\in A_2$, where \(\widehat{z}_i \coloneqq \psi_{L_i}(z_i) = L_i|_{\{z_i\} \times A_i} \in \Pic^0(A_i)\).
To establish this commutativity, we use \eqref{diag:pointwise-line-bundle} and \eqref{diag:adjoint-between-two} to extend \eqref{diag:symplectic} to a larger diagram:
\begin{equation}
\label{diag:symplectic-extended}
\begin{tikzcd}[row sep=large,scale cd =0.85]
	{\bdc(X_1\times X_1,\alpha_1^{-1}\boxtimes\alpha^{\vphantom{1}}_1)} &&& {\bdc(X_2\times X_2,\alpha_2^{-1}\boxtimes\alpha^{\vphantom{2}}_2)} \\
	& {\bdc(A_1,q_1^*\alpha_1)} & {\bdc(A_2,q_2^*\alpha_2)} \\
	& {\bdc(A_1,q_1^*\alpha_1)} & {\bdc(A_2,q_2^*\alpha_2)} \\
	{\bdc(X_1\times X_1,\alpha_1^{-1}\boxtimes\alpha^{\vphantom{1}}_1)} &&& {\bdc(X_2\times X_2,\alpha_2^{-1}\boxtimes\alpha^{\vphantom{2}}_2)}\mathrlap{\,.}
	\arrow["{\bAd_{\Phi}}", from=1-1, to=1-4]
	\arrow["{\bad_{z_1}}"', from=1-1, to=4-1]
	\arrow["{\bXi_2^{-1}}"{description}, from=1-4, to=2-3]
	\arrow["{\bad_{z_2}}", from=1-4, to=4-4]
	\arrow["{\bXi_1}"{description}, from=2-2, to=1-1]
	\arrow["{\Psi}", from=2-2, to=2-3]
	\arrow["{(-)\otimes\scP_{A_1,\widehat{z}_1}}"', from=2-2, to=3-2]
	\arrow["{(-)\otimes\scP_{A_2,\widehat{z}_2}}", from=2-3, to=3-3]
	\arrow["{\Psi}", from=3-2, to=3-3]
	\arrow["{\bXi_1}"{description}, from=3-2, to=4-1]
	\arrow["{\bAd_{\Phi}}", from=4-1, to=4-4]
	\arrow["{\bXi_2^{-1}}"{description}, from=4-4, to=3-3]
\end{tikzcd}
\end{equation}
By \eqref{diag:adjoint-between-two}, $\psi(z_1)=z_2$ implies that
\[
\bAd_{\Phi}\circ\bXi_1\bigl(k(z_1)\bigr)\cong\bXi_2\circ\Psi\bigl(k(z_1)\bigr)\cong\bXi_2\bigl(k(z_2)\bigr)\,,\]
i.e., $\Phi\circ\Phi_{z_1}\cong\Phi_{z_2}\circ\Phi$, due to Proposition \ref{prop:point-to-translation}.
Therefore we have
\[
\bAd_{\Phi}\circ\bad_{z_1}\cong\bAd_{(\Phi\circ\Phi_{z_1})}\cong\bAd_{(\Phi_{z_2}\circ\Phi)}\cong\bad_{z_2}\circ\bAd_{\Phi}\,,
\]
and hence the outer square commutes.
Now the commutativity of the inner square follows from that of the outer areas.
More explicitly, we have
\begin{align*}
    (-)\otimes\scP_{A_2,\widehat{z}_2}\circ\Psi
    &\cong(-)\otimes\scP_{A_2,\widehat{z}_2}\circ\bXi_2^{-1}\circ\bAd_\Phi\circ\bXi_1 \\
    &\cong\bXi_2^{-1}\circ\bad_{z_2}\circ\bAd_{\Phi}\circ\bXi_1 \\
    &\cong\bXi_2^{-1}\circ\bAd_{\Phi}\circ\bad_{z_1}\circ\bXi_1 \\
    &\cong\bXi_2^{-1}\circ\bAd_{\Phi}\circ\bXi_1\circ(-)\otimes\scP_{A_1,\widehat{z}_1} \\
    &\cong\Psi\circ(-)\otimes\scP_{A_1,\widehat{z}_1}\,.
\end{align*}
This establishes the required commutativity, completing the proof that $\psi$ is a symplectic isomorphism.
\end{enumerate}
\hfill\ensuremath{\Box}

Considering the case of $(X_1,\alpha_1)=(X_2,\alpha_2)$, simply denoted by $(X,\alpha)$, we fix a choice of $a\in\rZ^2(X[n], \mathbb{G}_m)$.
By Theorem \ref{thm:main}, we obtain a group homomorphism
\begin{align}
\label{eq:auteq-to-sympaut}
\Aut(\bD^b(X,\alpha))\to\Aut(A_{(X,\alpha)})\,.
\end{align}
The next proposition generalizes \cite[Proposition 3.3]{Orl02}.
\begin{proposition}
    Let $(X,\alpha)$ be a twisted abelian variety with $\ord(\alpha)$ invertible in $k$, then the kernel of (\ref{eq:auteq-to-sympaut}) is isomorphic to the group $\ZZ\times A_{(X,\alpha)}(k)$, where the factor $\ZZ$ is generated by the shift functor.
\end{proposition}
\begin{proof}
    Suppose that the autoequivalence defined by $\scE\in\bD^b(X\times X,\alpha^{-1}\boxtimes\alpha)$ lies in the kernel of (\ref{eq:auteq-to-sympaut}).
    This means that for any $(y,\xi)\in X\times\widehat{X}$ we have $\Phi_{\scE}\circ\Phi_{(y,\xi)}\cong\Phi_{(y,\xi)}\circ\Phi_{\scE}$ by the construction of (\ref{eq:auteq-to-sympaut}).
    In particular, for any $x\in X$ we have 
    \[
    \Phi_{\scE}(k(x))\cong\Phi_{\scE}(k(x)\otimes\scP_{\xi})\cong\Phi_{\scE}(k(x))\otimes\scP_{\xi}\,,\quad\forall\xi\in\widehat{X}\,.
    \]
    Applying the forgetful functor $\bUp$ to these isomorphisms, the objects $\bUp\bigl(\Phi_{\scE}(k(x))\bigr)$ satisfy
    \[
    \bUp\bigl(\Phi_{\scE}(k(x))\bigr)\cong\bUp\bigl(\Phi_{\scE}(k(x))\bigr)\otimes\scP_\nu\,,\quad\forall\nu=[n]_{\widehat{X}}(\xi)\in\widehat{X}
    \]
    as objects in $\bD^b(X)$.
    Then according to \cite[Section 2.2]{dJO22} the support of $\bUp\bigl(\Phi_{\scE}(k(x))\bigr)$ is a finite set.
    Hence the same holds for $\Phi_{\scE}(k(x))$.
    Since $\Phi_{\scE}$ is an equivalence, the $\operatorname{Ext}$-groups constraints force each $\Phi_{\scE}(k(x))$ to be a skyscraper sheaf up to shift (cf.~\cite[proof of Lemma 2.4]{dJO22}) and this shift needs to be a constant due to semi-continuity.
    Hence $\Phi_{\scE}$ is isomorphic to $f_*(-)\otimes\scL[m]$ for some $f\colon X\to X$, $\scL\in\Pic(X)$ and $m\in\ZZ$ as in the proof of Proposition \ref{prop:orlov-twisted}.
    Finally, comparing the supports in the commutativity with all $\Phi_{(y,0)}$ shows that $f$ is a translation.
    The remaining line bundle is then invariant under all translations, hence lies in $\Pic^0(X)$.
    Combining this with the injectivity in Proposition \ref{prop:embedding-A-to-Auto} finishes the proof. 
\end{proof}

If we denote the image of (\ref{eq:auteq-to-sympaut}) by $\rU(A_{(X,\alpha)})$, then there is a short exact sequence
\begin{align}    
0\to\ZZ\times A_{(X,\alpha)}(k)\to\Aut(\bD^b(X,\alpha))\to\rU(A_{(X,\alpha)})\to1\,.
\end{align}
\begin{remark}
    Using the theory of twisted semi-homogeneous sheaves developed in \cite{Li25} or \cite{la26}, one can try to show that $\rU(A_{(X,\alpha)})$ is exactly the group of symplectic isomorphisms, assuming $\operatorname{char}(k)=0$.
\end{remark}

\subsection{Hodge-theoretic Realizations}
\label{subsec:cohom-real}

When $k = \mathbb{C}$, the symplectic abelian variety \(\scA_{(X,\alpha)}\) admits a Hodge-theoretic description, which we explain in this section.
Assume that the dimension of $X$ is $g$
and let $J$ denote the complex structure on $\rH_1(X,\mathbb{Z})$.
Using the exponential exact sequence
\[
\cdots \to \rH^2(X, \mathbb{Z}) \to \rH^2(X, \mathcal{O}_X) \to \rH^2(X, \mathcal{O}_X^*) \to \rH^3(X, \mathbb{Z}) \to \cdots\,,
\]
we can associate a \textbf{B}-field
\[
B_\alpha \in \rH^2(X, \mathbb{Q})
\]
so that its \((0, 2)\)-part \( B^{0,2}_\alpha \in \rH^{0,2}(X) = \rH^2(X, \mathcal{O}_X) \) gives rise to the class \( \alpha  \in \rH^2(X, \mathcal{O}_X^*)\).
Such an element in $\rH^2(X,\QQ)$ is called a \textbf{B}-field lift of the Brauer class $\alpha$.

Recall that in the construction of $A_{(X, \alpha)}$, we have fixed an element $e_{\alpha}\in\mathrm{Hom}(\wedge^2 X[n], \mu_n)$.
Now we shall take a \textbf{B}-field lift compatible with $e_{\alpha}$.
In~\cite{Ber72}, the construction of~\eqref{eq:lifting-Brauer} used the natural isomorphism
\[
\kappa\colon\rH^2(X,\mu_n)\cong\mathrm{Hom}(\wedge^2 X[n],\mu_n)\,.
\]
Then consider the commutative diagram
\[\begin{tikzcd}[column sep = huge]
	{\rH^2(X,\frac{1}{n}\ZZ)} & {\rH^2(X,\mu_n)} \\
	{\rH^2(X,\QQ)} & {\rH^2(X,\cO^*_X)}
	\arrow["{\mathrm{exp}(2\pi i\cdot-)}", two heads, from=1-1, to=1-2]
	\arrow[from=1-1, to=2-1]
	\arrow[from=1-2, to=2-2]
	\arrow["{\mathrm{exp}(2\pi i\cdot-)}", from=2-1, to=2-2]
\end{tikzcd}\]
Take a \textbf{B}-field lift $B_{\alpha}\in\rH^2(X,\frac{1}{n}\ZZ)\subset\rH^2(X,\QQ)$ such that $\mathrm{exp}(2\pi iB_{\alpha})=\kappa^{-1}(e_{\alpha})$.
We regard $B_{\alpha}$ as an element in $\Hom_{\ZZ}\bigl(\rH_1(X,\ZZ), \rH_1(\widehat{X},\frac{1}{n}\ZZ)\bigr)$ or its natural $\QQ$-linear extension.
Since $nB_{\alpha}(\rH_1(X,\mathbb{Z})) \subseteq \rH_1(\widehat{X},\mathbb{Z})$,
we define
\[-nB_{\alpha}\colon \rH_1(X,\mathbb{Q})/\rH_1(X,\mathbb{Z}) \longrightarrow \rH_1(\widehat{X},\mathbb{Q})/\rH_1(\widehat{X},\mathbb{Z})\,\]
on the quotient groups.
In particular, the map
\[
\phi_{B_\alpha}\colon \left(\frac{1}{n}\rH_1(X,\ZZ)\right)/\rH_1(X,\ZZ)\to\left(\frac{1}{n}\rH_1(\widehat{X},\ZZ)\right)/\rH_1(\widehat{X},\ZZ)
\]
is given by $-nB_{\alpha}$ restricted to $X[n]$.
\begin{lemma}
    We have $\phi_{B_{\alpha}}=\phi_{\alpha}\colon X[n]\to\widehat{X}[n]$, where $\phi_{\alpha}$ is the homomorphism~\eqref{eq:twist-hom} associated to $e_{\alpha}$ via the Weil pairing.
\end{lemma}
\begin{proof}
    Under the natural identification $\rH^2(X,\mu_n)\cong\mathrm{Hom}(\wedge^2\rH_1(X,\ZZ),\mu_n)$, $\mathrm{exp}(2\pi iB_{\alpha})$ gives the bilinear form
    \[
    (u,v)\mapsto\mathrm{exp}\bigl(2\pi i(B_{\alpha}u)^{\top}v\bigr)
    \]
    for any $u$, $v\in \mathrm{H}_1(X,\mathbb{Z})$.
    Then via $\rH_1(X,\ZZ)/\bigl(n\rH_1(X,\ZZ)\bigr)\cong X[n]$, $u\mapsto[\frac{u}{n}]$,
    \[
    e_{\alpha}\Bigl(\bigl[\frac{u}{n}\bigr], \bigl[\frac{v}{n}\bigr]\Bigr)=\mathrm{exp}\bigl(2\pi i(B_{\alpha}u)^{\top}v\bigr)
    \]
    via the isomorphism $\kappa$~\cite[\S 2]{Ber72}.

    Notice that we also have
    \[
    \mathrm{exp}\bigl(2\pi i(B_{\alpha}u)^{\top}v\bigr)=\Bigl\langle[\frac{u}{n}], [-B_{\alpha}v]\Bigr\rangle=\Bigl\langle\bigl[\frac{u}{n}\bigr], -nB_{\alpha}\bigl[\frac{v}{n}\bigr]\Bigr\rangle\,,
    \]
    see~\cite[\S 10.6]{Pol03}.
    Combining these two equations together implies $\phi_{B_{\alpha}}=\phi_{\alpha}$.
\end{proof}

Now we can give the Hodge-theoretic explanation of $A_{(X,\alpha)}$.
\begin{proposition}\label{prop:complex-structure-Hodge}
The abelian variety $A_{(X,\alpha)}$ is determined by $\rH_1(X,\mathbb{Z}) \oplus \rH_1(\widehat{X},\mathbb{Z})$ endowed with:
\begin{enumerate}[label={(\alph*)}]
    \item The complex structure:
    \[
    \cJ_{\alpha} \coloneq \begin{pmatrix}
        J & 0 \\
        B_\alpha J + J^{\top} B_\alpha & -J^\top
    \end{pmatrix}\,,
    \]
    where $J$ is the complex structure on $\rH_1(X,\mathbb{Z})$ determined by $X$ and $-J^\top$ is the natural dual complex structure on $\rH_1(\widehat{X},\mathbb{Z}) = \rH_1(X,\mathbb{Z})^*$.

    \item The natural symplectic structure induced by the dual pairing on $\rH_1(X,\mathbb{Z}) \oplus \rH_1(\widehat{X},\mathbb{Z})$.
\end{enumerate}
\end{proposition}

\begin{proof}
Consider the commutative diagram of lattices
\begin{equation}
\begin{tikzcd}[sep=large]
	{\Lambda^*} & {\Lambda\oplus\Lambda^*} & \Lambda \\
	{\Lambda^*} & {\Lambda\oplus\Lambda^*} & \Lambda\mathrlap{\,,}
	\arrow["i", from=1-1, to=1-2]
	\arrow["\id"', from=1-1, to=2-1]
	\arrow["q", from=1-2, to=1-3]
	\arrow["{\widetilde{\pi}_{\alpha}}", from=1-2, to=2-2]
	\arrow["{[n]_{\Lambda}}", from=1-3, to=2-3]
	\arrow["i"', from=2-1, to=2-2]
	\arrow["q"', from=2-2, to=2-3]
\end{tikzcd}
\end{equation}
where $\Lambda = \rH_1(X,\mathbb{Z})$ and $\Lambda^* = \rH_1(\widehat{X},\mathbb{Z})$ are the lattices under consideration,
and the maps are given as follows,

\hspace*{\fill}
\begin{minipage}{0.2\linewidth}
\begin{align*}
i \colon &\Lambda^* \to \Lambda \oplus\Lambda^* \\
 & \widehat{x}\mapsto(0,\widehat{x})\,,
\end{align*}
\end{minipage}
\hfill
\begin{minipage}{0.2\linewidth}
\begin{align*}
q\colon &\Lambda\oplus\Lambda^* \to \Lambda \\
& (x,\widehat{x})\mapsto x\,,
\end{align*}
\end{minipage}
\hfill
\begin{minipage}{0.4\linewidth}
\begin{align*}
\widetilde{\pi}_{\alpha}\colon &\Lambda\oplus\Lambda^*\to \Lambda\oplus\Lambda^* \\
&(x,\widehat{x}) \mapsto \bigl(nx,\widehat{x}+(nB_{\alpha})x\bigr)\,.
\end{align*}
\end{minipage}
\hspace*{\fill}

After base changing to $\mathbb{R}$ and taking the complex structures into consideration,
for any $(v,v^*)\in \Lambda_{\RR} \oplus \Lambda^*_{\RR}$, we have
\begin{align*}
    & \widetilde{\pi}_{\alpha,\RR}(\mathcal{J}_0(v,v^*))  - \mathcal{J}_{\alpha}(\widetilde{\pi}_{\alpha,\RR}(v,v^*)) \\
   = & \bigl(n Jv,-J^\top v^* + nB_{\alpha} J v\bigr) -  \mathcal{J}_{\alpha}\bigl(nv,v^*+nB_{\alpha}v\bigr) \\
   = & \bigl(n Jv,-J^\top v^* + nB_{\alpha} J v\bigr) - \bigl(nJv,(B_{\alpha}J+J^\top B_{\alpha})nv -J^\top (v^*+n B_\alpha v)\bigr) =0 \,,
\end{align*}
where $J \in \Hom(\Lambda_{\RR},\Lambda_{\RR})$ is the complex structure of $X$,
$-J^{\top}$ is the dual complex structure of $\widehat{X}$, and $\cJ_0 = \begin{pmatrix} J & 0 \\ 0 & -J^{\top}\end{pmatrix}$
along with $\cJ_{\alpha} =\begin{pmatrix} J & 0 \\ B_\alpha J + J^{\top} B_\alpha & -J^\top \end{pmatrix}$ are complex structures of $\Lambda_{\RR}\oplus\Lambda^*_{\RR}$.
Therefore we obtain the following commutative diagram
\begin{equation}
\begin{tikzcd}
	{\Lambda^*_{\RR}} && {\Lambda_{\RR}\oplus\Lambda^*_{\RR}} && \Lambda_{\RR} \\
	& {\Lambda^*_{\RR}} && {\Lambda_{\RR}\oplus\Lambda^*_{\RR}} && \Lambda_{\RR} \\
	{\Lambda^*_{\RR}} && {\Lambda_{\RR}\oplus\Lambda^*_{\RR}} && \Lambda_{\RR} \\
	& {\Lambda^*_{\RR}} && {\Lambda_{\RR}\oplus\Lambda^*_{\RR}} && \Lambda_{\RR} \mathrlap{\,.}
	\arrow["i_{\RR}"{description, pos=0.4}, from=1-1, to=1-3]
	\arrow["i_{\RR}"{description, pos=0.4}, from=2-2, to=2-4]
	\arrow["i_{\RR}"{description, pos=0.4}, from=3-1, to=3-3]
	\arrow["i_{\RR}"{description, pos=0.4}, from=4-2, to=4-4]
	\arrow["\id"{description}, from=1-1, to=2-2]
	\arrow["\id"{description}, from=3-1, to=4-2]
	\arrow["q_{\RR}"{description, pos=0.4}, from=1-3, to=1-5]
	\arrow["q_{\RR}"{description, pos=0.4}, from=2-4, to=2-6]
	\arrow["q_{\RR}"{description, pos=0.4}, from=3-3, to=3-5]
	\arrow["q_{\RR}"{description, pos=0.4}, from=4-4, to=4-6]
	\arrow["{\widetilde{\pi}_{\alpha,\RR}}"{description}, from=1-3, to=2-4]
	\arrow["{\widetilde{\pi}_{\alpha,\RR}}"{description}, from=3-3, to=4-4]
	\arrow["{-J^{\top}}"{description, pos=0.3}, from=1-1, to=3-1]
	\arrow["{-J^{\top}}"{description, pos=0.3}, from=2-2, to=4-2]
	\arrow["{\cJ_{0}}"{description, pos=0.3}, from=1-3, to=3-3]
	\arrow["{\cJ_{\alpha}}"{description, pos=0.3}, from=2-4, to=4-4]
	\arrow["J"{description, pos=0.3}, from=1-5, to=3-5]
	\arrow["J"{description, pos=0.3}, from=2-6, to=4-6]
	\arrow["{[n]_{\Lambda}}"{description}, from=1-5, to=2-6]
	\arrow["{[n]_{\Lambda}}"{description}, from=3-5, to=4-6]
\end{tikzcd}
\end{equation}

It is easy to see that the split complex structure $\cJ_0$ corresponds to the split abelian variety $X\times \widehat{X}$,
and let $\widetilde{  A}_{(X,\alpha)}$ be the abelian variety associated to $\bigl(\rH_1(X,\mathbb{Z}) \oplus \rH_1(\widehat{X},\mathbb{Z}),\cJ_{\alpha}\bigr)$,
the commutative diagram then induces a commutative diagram of abelian varieties
\begin{equation}
\label{diagram:covering-of-A}
\begin{tikzcd}[sep=large]
	0 & {\widehat{X}} & {X\times\widehat{X}} & X & 0 \\
	0 & {\widehat{X}} & {\widetilde{  A}_{(X,\alpha)}} & X & 0\mathrlap{\,.}
	\arrow[from=1-1, to=1-2]
	\arrow["i_0", from=1-2, to=1-3]
	\arrow["\id", from=1-2, to=2-2]
	\arrow["q_0", from=1-3, to=1-4]
	\arrow["{\widetilde{\pi}_{\alpha}}", from=1-3, to=2-3]
	\arrow[from=1-4, to=1-5]
	\arrow["{[n]_X}", from=1-4, to=2-4]
	\arrow[from=2-1, to=2-2]
	\arrow["{\widetilde{i}_{\alpha}}"', from=2-2, to=2-3]
	\arrow["{\widetilde{q}_{\alpha}}"', from=2-3, to=2-4]
	\arrow[from=2-4, to=2-5]
\end{tikzcd}
\end{equation}

We now check that $\widetilde{\pi}_{\alpha}\colon X\times \widehat{X} \to \widetilde{ A}_{(X,\alpha)}$ is given by taking the quotient by a graph of $X[n]$.
Let $j_0\colon X \to X\times\widehat{X}$, $x\mapsto(x,0)$ be a splitting of $0\to \widehat{X} \xrightarrow{i_0}X\times\widehat{X} \xrightarrow{q_0} X \to 0$.
For any $(y,\hat{y}) \in \ker\widetilde{\pi}_{\alpha}$,
one has $\widetilde{i}_{\alpha}(\hat{y}) + \widetilde{\pi}_{\alpha} \circ j_0(y) = 0$, then
\[0 = \widetilde{q}_{\alpha}\bigl(\,\widetilde{i}_{\alpha}(\hat{y}) + \widetilde{\pi}_{\alpha} j_0(y)\bigr) = 0 + [n]_X\circ q_0 \circ j_0 (y) = ny\,, \]
which means $y \in X[n]$.
Also note that if $(0,\hat{y})\in\ker\widetilde{\pi}_{\alpha}$, then $\widetilde{i}_{\alpha}(\hat{y}) = 0$, which implies $\hat{y} = 0$ as $\widetilde{i}_{\alpha}$ is injective.
Therefore $\ker\widetilde{\pi}_{\alpha}$ is the graph of a map, which is easily seen to be just $\phi_{B_{\alpha}}\colon X[n] \to \widehat{X}[n]$ according to the definition of $\widetilde{\pi}_{\alpha}$.

Then considering the following diagram of exact sequences arising from \eqref{eq:exact-seq-A} and \eqref{diagram:covering-of-A},
\[
\begin{tikzcd}[sep=large]
	0 & {\widehat{X}} & {X\times\widehat{X}} & X & 0 \\
	0 & {\widehat{X}} & {\bullet} & X & 0 \mathrlap{\,,}
	\arrow[from=1-1, to=1-2]
	\arrow["{i_0}", from=1-2, to=1-3]
	\arrow["\id", from=1-2, to=2-2]
	\arrow["{q_0}", from=1-3, to=1-4]
	\arrow[dashed, from=1-3, to=2-3]
	\arrow[from=1-4, to=1-5]
	\arrow["{{[n]_X}}", from=1-4, to=2-4]
	\arrow[from=2-1, to=2-2]
	\arrow[dashed, from=2-2, to=2-3]
	\arrow[dashed, from=2-3, to=2-4]
	\arrow[from=2-4, to=2-5]
\end{tikzcd}
\]
as we have shown that $\phi_{\alpha} = \phi_{B_{\alpha}}$,
and both the tuples
\[\bigl( A_{(X,\alpha)},\pi_{\alpha},i_{\alpha},q_{\alpha}\bigr) \quad\text{and}\quad \bigl(\widetilde{A}_{(X,\alpha)},\widetilde{\pi}_{\alpha},\widetilde{i}_{\alpha},\widetilde{q}_{\alpha}\bigr)\]
fit into the diagram as the dot object and the three dashed arrows,
then one can see that these two tuples are the same according to the universal properties.
Therefore $A_{(X,\alpha)}$ has two kinds of presentations,
the one given by the complex structure ensures that it is associated to $\cJ_{\alpha}$
and the other given by the quotient of $X \times \widehat{X}$ ensures that it inherits a symplectic structure descending from $X\times \widehat{X}$,
which is determined by the dual pairing on the lattices.
\end{proof}
One can see that the definition of $\cJ_{\alpha}$ depends only on $J$ and $\alpha$, up to conjugation by an integral matrix
\[
     \begin{pmatrix} \bI_{2g} & 0 \\ B  & \bI_{2g} \end{pmatrix}
\]
where $B\in\rH^2(X,\ZZ)\subset\mathrm{Hom}(\rH_1(X,\ZZ),\rH_1(\widehat{X},\ZZ))$.
Since $B$ is skew-symmetric when regarded as a matrix, the corresponding transform will not change the natural matrix representation of the dual pairing on $\rH_1(X\times\widehat{X}, \ZZ)$. 

\begin{remark}
\label{rmk:difference}
    \begin{enumerate}
    \item
    The choice of the matrix for the complex structure is lower triangular in order to be consistent with \cite[\S 2]{KO03},
    one can switch to an upper triangular matrix if one uses cohomology instead of homology, or starts with the dual abelian varieties.

    \item
    The sign of the matrix of the \textbf{B}-field actually has little effect,
    since one has the following equation of matrices
    \[
        \begin{pmatrix} -\bI_{2g} & 0 \\ 0  & \bI_{2g} \end{pmatrix}
        \begin{pmatrix} J & 0 \\ BJ + J^\top B & -J^\top \end{pmatrix}
        \begin{pmatrix} -\bI_{2g} & 0 \\ 0 & \bI_{2g} \end{pmatrix}
        = \begin{pmatrix} J & 0 \\ -BJ - J^\top B & -J^\top \end{pmatrix}\,,
    \]
    which shows that the complex structures corresponding to the \textbf{B}-fields $B$ and $-B$ considered as elements in $\Hom_{\mathbb{Q}}\bigl(\rH_1(X,\mathbb{Q}), \rH_1(\widehat{X},\mathbb{Q})\bigr)$ are conjugate.
    Therefore, there is a canonical way to adjust between the signs;
    however, according to the following equation
    \[
        \begin{pmatrix} -\bI_{2g} & 0 \\ 0  & \bI_{2g} \end{pmatrix}
        \begin{pmatrix} 0 & \bI_{2g} \\ \bI_{2g} & 0 \end{pmatrix}
        \begin{pmatrix} -\bI_{2g} & 0 \\ 0 & \bI_{2g} \end{pmatrix}
        = \begin{pmatrix} 0 & -\bI_{2g} \\ -\bI_{2g} & 0 \end{pmatrix}\,,
    \]
    the adjustment is not a symplectic isomorphism between $A_{(X,\alpha)}$ and $A_{(X,\alpha^{-1})}$,
    but changes the corresponding symplectic biextension with a multiple of $-1$.
    Together with the difference mentioned in Remark~\ref{rmk: diff-choice-symplectic-biextension}, these account for the subtlety in the matrix representation of the complex structure of $A_{(X,\alpha)}$,
    and our notation follows the one used by Kapustin--Orlov in \cite{KO03}.
    \end{enumerate}
\end{remark}

To end this section, as a consequence, we have proved Theorem \ref{thm:TDE} from Theorem \ref{thm:main} and Proposition \ref{prop:complex-structure-Hodge}.

\subsection{Application: dual of twisted abelian variety}

Let $k$ be an algebraically closed field of characteristic $0$. 
For abelian surfaces, a well-known result concerning Shioda's question asserts that two abelian surfaces $X$ and $X'$ are derived equivalent if and only if their associated Kummer surfaces are isomorphic, i.e.
\[
\mathrm{Km}(X) \cong \mathrm{Km}(X')
\quad\Longleftrightarrow\quad
\mathbf{D}^{\mathrm{b}}(X) \simeq \mathbf{D}^{\mathrm{b}}(X').
\]
(cf.~\cite{HosonoLianEtAl2003}). For twisted abelian surfaces $(X,\alpha)$, one may associate a unique twisted Kummer surface $(\mathrm{Km}(X), \widetilde{\alpha})$ via the canonical isomorphism
\[
\mathrm{Br}(X) \cong \mathrm{Br}(\mathrm{Km}(X)),
\]
(cf.~\cite[Proposition 2.2.1]{LZ25}). A natural question then arises:
\begin{quote}\itshape
Does an analogous statement hold in the twisted setting? That is, is $\mathbf{D}^{\mathrm{b}}(X,\alpha)$ determined by the isomorphism class of the associated twisted Kummer surface?
\end{quote}
A related question, whether $\mathbf{D}^{\mathrm{b}}(X,\alpha)$ is determined by the \emph{derived equivalence class} of the twisted Kummer surface, has been investigated in \cite{Stel07}.

As we shall see, the answer is negative. A typical counterexample already appears when $X'$ is the dual abelian surface $\widehat{X}$ of $X$: for any Brauer class $\alpha \in \operatorname{Br}(X)$, there exists a Brauer class $\beta \in \operatorname{Br}(\widehat{X})$ such that the associated twisted Kummer surfaces $(\operatorname{Km}(X), \widetilde{\alpha})$ and $(\operatorname{Km}(\widehat{X}), \widetilde{\beta})$ are isomorphic. Thus twisted Kummer surfaces alone do not detect the distinction between $(X,\alpha)$ and $(\widehat{X},\beta)$.
Nevertheless, we show that at the derived level, the categories $\mathbf{D}^{\mathrm{b}}(X,\alpha)$ and $\mathbf{D}^{\mathrm{b}}(\widehat{X},\beta)$ are not equivalent for $(X,\alpha)$ being very general. In fact, we prove a stronger statement in arbitrary dimension.

\begin{theorem}\label{thm:no-dual-equiv-general}
Let $X$ be a simple abelian variety of dimension $g \ge 2$ over an algebraically closed field of characteristic $0$ such that
\begin{itemize}
    \item $\mathrm{End}(X) \cong \mathbb{Z}$,
    \item $\mathrm{NS}(X) = \mathbb{Z}H$, where $H$ is an ample polarization of type $(1,\dots,1,d)$ with $d > 1$.
\end{itemize}
Then there exists $\alpha \in \mathrm{Br}(X)$ such that for any Brauer class $\beta \in \mathrm{Br}(\widehat{X})$, the twisted derived categories $\mathbf{D}^{\mathrm{b}}(X,\alpha)$ and $\mathbf{D}^{\mathrm{b}}(\widehat{X},\beta)$ are \emph{not} equivalent.
\end{theorem}
\begin{proof}
First note that our assumptions on $X$ imply that $X$ is not isomorphic to $\widehat{X}$, ruling out the trivial case.
Consider the Riemann form of $H$ on the Tate modules; see~\cite[\S 20]{Mum70}.
Then it gives compatible level structures in which the homomorphism $\phi_H\colon X[d]\to\widehat{X}[d]$ induced by $H$ is given by the block diagonal matrix
\[
\phi_H = \diag\Bigl(\underbrace{\bJ_2, \dots, \bJ_2}_{g-1 \text{ times}}, d\bJ_2\Bigr) \,,
\]
where $\bJ_2=\begin{pmatrix}0&1\\-1&0\end{pmatrix}$ is the standard $2\times 2$ symplectic matrix.
Define a Brauer class $\alpha\in\Br(X)[d]$ via the lift $\phi_\alpha\colon X[d]\to\widehat{X}[d]$ given by the block diagonal matrix
\[
\phi_\alpha = \diag\Bigl(\underbrace{0, \dots, 0}_{2g-3 \text{ times}}, \bJ_2, 0 \Bigr) \,.
\]
Here the identification between the matrix and the homomorphism is exactly via
$\widehat{X}[d]\cong\Hom(X[d],\mu_d)$ and the corresponding element in $\Hom(\wedge^2 X[d],\mu_d)$ is given by
\[
(x,y)\mapsto\langle x, \phi_\alpha(y)\rangle\,.
\]

Assume, for contradiction, that there exists $\beta$ such that $\bdc(X,\alpha)\simeq\bdc(\widehat{X},\beta)$.
By Theorem \ref{thm:main}, this yields a symplectic isomorphism $A_{(X,\alpha)}\cong A_{(\widehat{X},\beta)}$.
In particular, $X$ embeds as a Lagrangian subvariety of $A_{(X,\alpha)}$.

Let $i_X\colon X\hookrightarrow A_{(X,\alpha)}$ be such an embedding.
The assumptions on $X$ force the composition $X\hookrightarrow A_{(X,\alpha)}\twoheadrightarrow X$ to be an isogeny.
Indeed, since $X$ is simple, the composition is either an isogeny or the zero map. If it were the zero map, then $i_X(X)$ would be contained in $\widehat{X}\subset A_{(X,\alpha)}$, which is impossible since $X$ is not isomorphic to $\widehat{X}$.
Then $\mathrm{End}(X)\cong\mathbb{Z}$ implies that the kernel of this isogeny is $X[n]$ for some positive integer $n$.
Notice that $\alpha\neq 0$, so we have $n>1$ by~\cite[Proposition 3.2]{Pol96}.

Consequently the intersection $i_X(X)\cap\widehat{X}\subset A_{(X,\alpha)}$ is finite and isomorphic to $X[n]$.
According to a result of Polishchuk \cite[Theorem~1.2]{Pol96}, the existence of such a finite intersection forces the existence of a lift $\phi\colon X[m]\to\widehat{X}[m]$ (for some $m$ divisible by $d$) whose image is exactly $X[n]$.

Now we examine all lifts of $\alpha$ to level $m$.
Using the fixed level structures, any homomorphism $X[m]\to\widehat{X}[m]$ that lifts $\alpha$ can be written as
\[
\phi = \phi_\alpha^{(m)} + \phi_{\NS}^{(m)},
\]
where $\phi_\alpha^{(m)}$ is the natural lift of $\phi_\alpha$ to level $m$ (obtained by multiplying the $d$‑level matrix by $m/d$ in the appropriate block) and $\phi^{(m)}_{\NS}$ lies in the image of $\NS(X)/m\NS(X)$ under the map $\iota$ in sequence (\ref{eq:lifting-Brauer}).
Since $\NS(X)=\mathbb{Z}H$, every such $\phi_{\NS}^{(m)}$ is a multiple of the matrix representing $\phi_H$ at level $m$.
Explicitly, after choosing coordinates compatibly, the matrix of $\phi$ (acting on $(\mathbb{Z}/m\mathbb{Z})^{2g}$) has the following block structure:
\[
\diag\left(\underbrace{k\bJ_2, \dots, k\bJ_2}_{g-2 \text{ times}}, 
\begin{pmatrix}
        0 & k & 0 & 0 \\
        -k & 0 & m/d & 0 \\
        0 & -m/d & 0 & kd \\
        0 & 0 & -kd & 0
\end{pmatrix}
\right)
\]

The Smith normal form of the last $4 \times 4$ block is
\[
\diag\left(\gcd(k, m/d), \gcd(k, m/d), \frac{k^2d}{\gcd(k, m/d)}, \frac{k^2d}{\gcd(k,m/d)}\right)\,.
\]
Since $\gcd(m, \frac{k^2d}{\gcd(k, m/d)})=d\gcd(k, m/d)$, the image of $\phi$ is isomorphic to
\[
\bigl(\mathbb{Z}/\tfrac{m}{\gcd(k, m/d)}\mathbb{Z}\bigr)^{\!2}\;\oplus\;
\bigl(\mathbb{Z}/\tfrac{m}{d\gcd(k, m/d)}\mathbb{Z}\bigr)^{\!2}\;\oplus\;
\bigl(\mathbb{Z}/\tfrac{m}{\gcd(k,m)}\mathbb{Z}\bigr)^{\!2g-4}\,,
\]
for some integer $k$ depending on the lift. Because $d>1$, the image of $\phi$ cannot be isomorphic to $X[n]\cong(\mathbb{Z}/n\mathbb{Z})^{2g}$ for any $n$, contradicting the requirement obtained from the Lagrangian embedding. Thus no such $\beta$ exists, completing the proof. 
\end{proof}

\begin{remark}
\label{rmk:alternative-proof-surface}
For abelian surfaces ($g=2$), there is an alternative approach via the twisted derived Torelli theorem established in \cite[Theorem C]{CLZ26}, which describes derived equivalences of twisted abelian surfaces in terms of their twisted Mukai lattices. 
Specifically, two twisted abelian surfaces are derived equivalent if and only if there exists an admissible Hodge isometry between their twisted Mukai lattices. 
One could also use this criterion to give a different proof of Theorem \ref{thm:no-dual-equiv-general} in the surface case, by showing that no such isometry can exist between their twisted Mukai lattices.
\end{remark}

\section{Derived isogeny and principal isogeny}
\label{sec:derived-isogeny}

In this section, we characterize derived isogeny classes of abelian varieties in purely classical terms: two abelian varieties of dimension at least 2 are derived isogenous if and only if they are principally isogenous. The proof relies on a decomposition of principal isogenies into spectrally paired isogenies, each of which induces a twisted derived equivalence. As a byproduct, we also obtain a connection between derived isogenies of K3 surfaces and those of their associated Kuga--Satake varieties.

We work over an algebraically closed field \(k\) with \(\operatorname{char} k = p \neq 2\). All isogenies are assumed to be prime-to-\(p\), i.e., of degree coprime to \(p\) when \(p>0\).

\subsection{Symplectic Isomorphisms from Spectrally Paired Isogenies}

Let \( X \), \( Y \) be two abelian varieties over the field \( k \). Suppose that \( f \colon \widehat{X} \to Y \) is an isogeny.  A natural question is whether $f$ can induce a symplectic map between symplectic abelian varieties associated to $X$ and $Y$.

To study this problem more concretely,
we first introduce a construction that extends the isogeny $f$ to an isogeny between $X \times \widehat{X}$ and $Y \times \widehat{Y}$ as follows:
Suppose that $\ker f\subseteq\widehat{X}[n]$ for some $n$ that is invertible in \( k \), there exists an isogeny \( f' \colon Y \to \widehat{X} \) such that
\[
f \circ f' = [n]_Y \quad \text{and} \quad f' \circ f = [n]_{\widehat{X}}\,.
\]
Consider the extended isogeny
\[
\varphi_f \coloneqq
\begin{pmatrix}
0 & f \\
-n \widehat{f}' & 0
\end{pmatrix}
\colon X \times \widehat{X} \longrightarrow Y \times \widehat{Y}\,,
\]
a direct computation shows that
\begin{equation}
\label{eq:descending-X-to-Y}
\widehat{\varphi}_f \circ \psi_Y \circ \varphi_f = n^2 \psi_X\,,
\end{equation}
where \( \psi_X \) and \( \psi_Y \) are isomorphisms induced by the canonical symplectic biextensions of \( (X \times \widehat{X})^2 \) and \( (Y \times \widehat{Y})^2 \) respectively.

The question now is reduced to when $\varphi_f$ can be lifted to a symplectic isomorphism.
We introduce spectrally paired isogenies to settle this problem.

\begin{definition}
 A finite abelian group is called \emph{spectrally paired} if it is isomorphic to a group of the form \( \bigoplus_{i=1}^r (\mathbb{Z}/a_i\mathbb{Z})^2 \).  A prime-to-\(p\) isogeny \( f \colon \widehat{X} \to Y \) is said to be \emph{spectrally paired} if $\ker f$ is spectrally paired.
\end{definition}

The following result guarantees that there exists an anti-symmetric homomorphism associated to a spectrally paired isogeny.
\begin{lemma}
\label{lem:ker-to-im}
Assume $p\nmid n$. Let $G\subseteq \widehat{X}[n]$ be a spectrally paired subgroup.
Then there exists an anti-symmetric homomorphism $\phi\colon X[n]\to\widehat{X}[n]$ such that $\image\phi=G$.
\end{lemma}

\begin{proof}
    The finite abelian group $\widehat{X}[n]$ is homocyclic and hence homogeneous,
    therefore one can choose  a suitable isomorphism $\widehat{X}[n]\cong(\mathbb{Z}/n\mathbb{Z})^{2g}$,
    and identify $G$ with a diagonal subgroup
    \[G=\bigoplus^{r}_{i=1}(\mathbb{Z}/a_i\mathbb{Z})^2 \subseteq(\mathbb{Z}/n\mathbb{Z})^{2g}\,,\]
    (see \cite[Theorem 2]{zbMATH00061414} and \cite[\S 15 Property (E)]{zbMATH05080399}).

    Consider the anti-symmetric block diagonal matrix
    \[
    A = \diag\biggl(\frac{n}{a_1}\bJ_2,\cdots,\frac{n}{a_r}\bJ_2,0,\cdots,0\biggr)\,,
    \]
    this matrix $A$ defines an element in $\Hom(\wedge^2 X[n],\mu_n)$ via the Weil pairing as before.
    It is not hard to see that $\image\phi =G$. The anti-symmetry of $\phi$ follows from that of the matrix $A$.
\end{proof}

The key result is

\begin{theorem}
\label{thm:spectral-induce-symplectic-iso}
Let \(f \colon \widehat{X} \to Y\) be a spectrally paired prime-to-\(p\) isogeny. Then there exist natural twisted abelian varieties \((X,\alpha)\) and \((Y,\beta)\) with a symplectic isomorphism
\[
\psi \colon A_{(X,\alpha)} \longrightarrow A_{(Y,\beta)}
\]
fitting into the commutative diagram:
\[
\begin{tikzcd}[column sep=large]
	X \times \widehat{X} \ar[r, "\varphi_f"] \ar[d, "\pi_X"'] &
	Y \times \widehat{Y} \ar[d, "\pi_Y"] \\
	A_{(X,\alpha)} \ar[r, "n\psi"'] &
	A_{(Y,\beta)}\mathrlap{\,.}
\end{tikzcd}
\]
Moreover, the intersection \(\psi(\widehat{X}) \cap \widehat{Y}\) is finite with order invertible in \(k\).
\end{theorem}

\begin{proof}
By Lemma \ref{lem:ker-to-im}, there exists an anti-symmetric homomorphism \( \phi_X \)  with \( \image \phi_X = \ker f \).   Let $\alpha$ be the image of $\phi_X$  in $\rH^2(X,\mathbb{G}_m)$.
Note that the graph  of $\phi_X $  is contained in \( \ker \varphi_f \), we obtain a factorization
\[
\varphi_f\colon X \times \widehat{X} \xlongrightarrow{\pi_X}A_{(X,\alpha)} \xlongrightarrow{\pi'} Y \times \widehat{Y}\,.
\]

The line bundle \( L_X^{\otimes n} \) descends to a symplectic biextension \( L \) on \( A_{(X,\alpha)}^2 \) via $\pi_X$, inducing an isomorphism \( \psi_L \colon A_{(X,\alpha)} \to \widehat{A}_{(X,\alpha)} \).
These fit in the diagram:
\begin{equation}
\label{diag:decompostion-X-A-Y}
\begin{tikzcd}[sep=large]
	{X \times \widehat{X}} & A_{(X,\alpha)} & {Y \times \widehat{Y}} \\
	{\widehat{X} \times X} & {\widehat{A}_{(X,\alpha)}} & {\widehat{Y} \times Y} \mathrlap{\,.}
	\arrow["{\pi_X}", from=1-1, to=1-2]
	\arrow["{n^2 \psi_X}"', from=1-1, to=2-1]
	\arrow["{\pi'}", from=1-2, to=1-3]
	\arrow["{n \psi_L}"', from=1-2, to=2-2]
	\arrow["{\psi_Y}", from=1-3, to=2-3]
	\arrow["{\widehat{\pi}_X}"', from=2-2, to=2-1]
	\arrow["{\widehat{\pi}'}"', from=2-3, to=2-2]
\end{tikzcd}
\end{equation}

To construct \((Y,\beta)\), note that \(\ker \pi' \subseteq A_{(X,\alpha)}[n] = \ker(n\psi_L)\), so there exists an isogeny \(\pi'_Y \colon Y \times \widehat{Y} \to A_{(X,\alpha)}\) with
\[
\pi' \circ \pi'_Y = [n]_{Y \times \widehat{Y}}\,, \quad
\pi'_Y \circ \pi' = [n]_{A_{(X,\alpha)}}\,.
\]
From the right square of \eqref{diag:decompostion-X-A-Y}, we derive:
\[
n \widehat{\pi}'_Y \circ \psi_L \circ \pi'_Y = n^2 \psi_Y\,,
\]
and consequently:
\begin{equation}
\label{eq:descending-Y-to-A}
(\pi'_Y \times \pi'_Y)^* L \cong L_Y^{\otimes n}\,.
\end{equation}

According to the construction of $A_{(X,\alpha)}$, it is clear that $\widehat{X}$ is a Lagrangian subvariety of $A_{(X,\alpha)}$ via the embedding $\pi_X|_{\widehat{X}}$.
We now show that $\widehat{Y}$ can also be viewed as one.

\begin{lemma}
  The restriction \(\pi'_Y|_{\widehat{Y}} \colon \widehat{Y} \to A_{(X,\alpha)}\) is injective, realizing \(\widehat{Y}\) as a Lagrangian subvariety.
\end{lemma}
\begin{proof}
We first establish that
\begin{equation}
\label{eq:commutative-embedding}
\pi_X|_{X}=\pi'_Y|_{\widehat{Y}}\circ(-\widehat{f}')\,.
\end{equation}
Indeed, note that
\[
\pi'\circ\pi_X|_X=(0,-n\widehat{f}')=\pi'\circ\pi'_Y|_{\widehat{Y}}\circ(-\widehat{f}')\,,
\]
in other words, $\pi_X|_{X}$ and $\pi'_Y|_{\widehat{Y}}\circ(-\widehat{f}')$ are equal after composing with $\pi'$.
This implies that they must be equal before composition since $\pi'$ is an isogeny.

Let $\#(-)$ denote the order of a finite group scheme.
Note that \eqref{eq:commutative-embedding} means that $\ker(-\widehat{f}')\subset\ker(\pi_X|_X)$.
We now compute the order of these two group schemes.
First, note that the order of $\ker(\pi_X|_X)$ is $\#(\ker\phi_X)$,
and then the order of $\ker(-\widehat{f}')$ is given as
\[
\#(\ker\widehat{f}')=\#(\ker f')=\frac{\#(\ker([n]_{\widehat{X}}))}{\#(\ker f)}=\frac{\#(\ker([n]_X))}{\#(\image\phi_X)}=\#(\ker\phi_X)\,.
\]
Therefore, it follows that $\pi'_Y|_{\widehat{Y}}$ must be injective.

The isotropic condition for $\widehat{Y}$ follows directly by restricting the identity $\widehat{\pi}'_Y \circ \psi_L \circ \pi'_Y = n \psi_Y\,$ to $\widehat{Y}$, since $\widehat{Y}$ is trivially isotropic with respect to $\psi_Y$.
\end{proof}

Together with \eqref{eq:descending-Y-to-A}, these imply that $\ker\pi'_Y$ is the graph of some anti-symmetric homomorphism $\phi_Y\colon Y[n]\to\widehat{Y}[n]$ (\cf\cite{Pol96}).
Let $\beta$ be the image of $\phi_Y$ in $\rH^2(Y,\GG_m)$.
A result of Polishchuk \cite[Theorem 1.2]{Pol96} then ensures that there exists a symplectic isomorphism $\psi\colon A_{(X,\alpha)}\to A_{(Y,\beta)}$ and its composition with $\pi'_Y$ is the quotient map
\[
\pi_Y \colon Y\times\widehat{Y}\longrightarrow A_{(Y,\beta)}\,.
\]
One can check that $n\psi\circ\pi_X=\pi_Y\circ\varphi_f$ directly.

For the last statement, note that
\[
\pi'(\widehat{X}\cap\psi^{-1}(\widehat{Y}))\subset\bigl(\pi'\circ\pi_X(\widehat{X})\bigr)\cap\bigl(\pi'\circ\pi'_Y(\widehat{Y})\bigr)=\{0\}\,.
\]
Hence $\psi(\widehat{X})\cap\widehat{Y}$ is finite of order invertible in $k$ since $\pi'$ is an isogeny whose degree is invertible in $k$.
\end{proof}

As a consequence, we obtain the following corollary.
\begin{corollary}\label{cor:spiso-tde}
    If $X$ and $\widehat{Y}$ are prime-to-$p$ spectrally paired isogenous, then $X$ and $Y$ are twisted derived equivalent. In particular, $X$ and $\widehat{Y}$ are derived isogenous.
\end{corollary}

\begin{remark}
    We note that the sufficiency of such isogenies for twisted derived equivalences was also independently obtained by Lane \cite[Theorem 3.12]{la26} using a different approach via semi-homogeneous twisted vector bundles.
\end{remark}

\subsection{Decomposition of Principal Isogenies}

Based on Theorem \ref{thm:spectral-induce-symplectic-iso}, we can figure out which kind of isogenies induce derived isogenies.
The following elementary result is crucial.

\begin{theorem}
\label{thm:isogeny-decomposition}
If the dimension $g \ge 2$, every principal isogeny $f\colon X\to Y$
decomposes as a composition of spectrally paired isogenies, after possibly composing with $X\xlongrightarrow{[N]_X} X$ for some positive integer $N$.  Moreover, if $f$ is prime-to-$p$, then all the spectrally paired isogenies that occur in the decomposition can be chosen to be prime-to-$p$.
\end{theorem}

\begin{proof}
Let us first show the existence of the decomposition.
The kernel of $f$ is finite \'etale and is contained in $X[n]\cong(\mathbb{Z}/n\mathbb{Z})^{2g}$ for some $n$ invertible in $k$.
Viewing $\ker f$ as a finite abelian group, we may write \[\ker f=\bigoplus^{2g}_{i=1}\mathbb{Z}/a_i\mathbb{Z}\] with $a_i=\prod\limits_{j=1}^{i} d_j$ for $i = 1,\ldots, 2g$,
where $d_i \in \mathbb{Z}_{>0}$ and $\deg(f)=\prod\limits_{i=1}^{2g} a_i$ is a square.
Set $\ker (f\circ[N]_X)=\bigoplus_{i=1}^{2g} \ZZ/Na_i\ZZ$,
then it suffices to show that there is a filtration of subgroups
\[0 = K_0 \subset K_1\subset K_2 \ldots  K_{l-1} \subset K_{l}= \ker (f\circ[N]_X)\]
for some integers $N>0$ and $\ell \ge 0$, such that each sub-quotient $K_i/K_{i-1}$ is spectrally paired for $i = 1, \dots, l$.

If we identify $\ker f$ as the cokernel of the diagonal map
\[\diag(a_1,\ldots, a_{2g})\colon \mathbb{Z}^{2g}\longrightarrow \mathbb{Z}^{2g}\,,\]
then it is equivalent to show that $\diag(a_1,\ldots, a_{2g})\in \GL_{2g}(\ZZ)$ can be decomposed as a product of scalar matrices in $\GL_{2g}(\QQ)$ and matrices in $\GL_{2g}(\ZZ)$ whose cokernels are spectrally paired.
Here by the notation $\GL$ we mean matrices with non-zero determinant.

\begin{enumerate}[align=left,leftmargin=0pt,labelindent=0pt,listparindent=\parindent,labelwidth=0pt,itemindent=!,label={\bfseries(\Roman*)}]
\item For $g=2$, note that there is the following decomposition:
\[
\diag(a_1,a_2,a_3,a_4)
=
\begin{pmatrix}
    d_1 \bI_2 & 0 \\
    0 & d_1 \bI_2
\end{pmatrix}
\begin{pmatrix}
    \bI_2 & 0 \\
    0 & d_2d_3 \bI_2
\end{pmatrix}
\begin{pmatrix}
    1 & 0 & 0 & 0 \\
    0 & d_2 & 0 & 0 \\
    0 & 0 & 1 & 0 \\
    0 & 0 & 0 & d_4
\end{pmatrix}\,.
\]

To prove the assertion,  we need to further decompose the third matrix when $d_2d_4$ is a square. Assuming $\gcd(d_2,d_4)=d$, the condition that $d_2d_4$ is a square implies that $d_2=da^2$ and $d_4=db^2$ for some $a$, $b \in \mathbb{Z}_{>0}$.
The matrix decomposes as
\[
\begin{pmatrix}
    1 & 0 & 0 & 0 \\
    0 & da^2 & 0 & 0 \\
    0 & 0 & 1 & 0 \\
    0 & 0 & 0 & db^2
\end{pmatrix}
=
\begin{pmatrix}
    1 & 0 & 0 & 0 \\
    0 & d & 0 & 0 \\
    0 & 0 & 1 & 0 \\
    0 & 0 & 0 &d
\end{pmatrix}
\begin{pmatrix}
    1 & 0 & 0 & 0 \\
    0 & a^2 & 0 & 0 \\
    0 & 0 & 1 & 0 \\
    0 & 0 & 0 & 1
\end{pmatrix}
\begin{pmatrix}
    1 & 0 & 0 & 0 \\
    0 & 1 & 0 & 0 \\
    0 & 0 & 1 & 0 \\
    0 & 0 & 0 & b^2
\end{pmatrix}.
\]
The last two factors decompose via
\begin{align*}
   \begin{pmatrix}
        1 & & & \\ & 1 & & \\ & & 1 & \\ & & & x^2
    \end{pmatrix}
    = \frac{1}{x}\begin{pmatrix}
        1 & & & \\ & 1 & & \\ & & x & \\ & & & x
    \end{pmatrix} \begin{pmatrix}
        1 & & & \\ & x & & \\ & & 1 & \\ & & & x
    \end{pmatrix} \begin{pmatrix}
        x & & & \\ & 1 & & \\ & & 1 & \\ & & & x
    \end{pmatrix}
    \end{align*}
which completes the proof in this case.

\item
For $g > 2$, we separate the last three coordinates using the decomposition:
\begin{equation}
\label{eq:high-genus-decomp}
\begin{aligned}
\diag(a_1, \ldots, a_{2g}) = \frac{1}{d_{2g}} &
\begin{pmatrix}
\diag(a_1d_{2g}, a_2d_{2g} , \ldots, a_{2g-3}d_{2g}, a_{2g-2}) & 0 \\
0 & \bI_2
\end{pmatrix}\\
& \begin{pmatrix}
\bI_{2g-2} & 0 \\
0 & \diag(a_{2g}, a_{2g})
\end{pmatrix}
\begin{pmatrix}
\bI_{2g-3} & 0 & 0 & 0\\
0 & d_{2g} & 0 & 0 \\
0 & 0 & 1 & 0 \\
0 & 0 & 0 & d_{2g}
\end{pmatrix}\,.
\end{aligned}
\end{equation}
The remaining matrix $\diag(a_1d_{2g}, a_2d_{2g} , \ldots, a_{2g-3}d_{2g},\allowbreak a_{2g-2})$
has the determinant which is a square.
After putting this matrix into Smith normal form, we apply the decomposition~\eqref{eq:high-genus-decomp} until we reach a remaining $4\times 4$ matrix.
A scalar denominator $N$ needed in the smaller block is handled in the full matrix by adding the factor $\diag(\bI_{2g-2}, N\bI_2)$, whose cokernel is spectrally paired.
This completes the induction.

\end{enumerate}

For the last assertion, note that if $f$ is prime-to-$p$, then all the $d_i$ are prime to $p$, and our construction preserves this property throughout the decomposition. Therefore, all the spectrally paired isogenies in the decomposition are also prime-to-$p$.
\end{proof}

\subsection{Proof of Theorem \ref{thm:main-2} and Theorem \ref{thm3}}
\label{subsec:Proof-main-2}

Suppose that $g\geq2$.
A derived isogeny is a sequence of derived equivalences of twisted abelian varieties,
according to Theorem \ref{thm:main} and the latter assertion of Theorem \ref{Thm:Polishchuk},
it induces a sequence of principal isogenies and thus again principal after composition.

Conversely, by Theorem \ref{thm:isogeny-decomposition}, a prime-to-$p$ principal isogeny is a composition of prime-to-$p$ spectrally paired isogenies up to scalar multiplications. By Corollary \ref{cor:spiso-tde}, we see that $X$ and $Y$ are derived isogenous.

\subsection{Proof of Corollary \ref{cor:KS-isogeous}}

We work over $k = \mathbb{C}$; the general case follows by descent.  
Let $X_1$, $X_2$ be two complex projective K3 surfaces that are derived isogenous.  
By \cite[Theorem 0.1]{Huy19} there exists a rational Hodge isometry  
\[
\rT(X_1)_{\mathbb{Q}} \;\cong\; \rT(X_2)_{\mathbb{Q}}\,,
\]  
where $\rT(X_i) \subset \rH^2(X_i,\mathbb{Z})$ is the transcendental lattice.  

Choose integral overlattices $\rT'(X_i) \subset \rT(X_i)$ (of finite index) such that $\rT'(X_1) \cong \rT'(X_2)$ as integral Hodge structures, and set $d_i \coloneq [\rT(X_i):\rT'(X_i)]$.  
Let $\mathrm{KS}(L)$ denote the Kuga--Satake abelian variety associated to an integral weight‑$2$ Hodge lattice $L$. By definition, we have 
\[\mathrm{KS}(X_i)=\mathrm{KS}(\rT(X_i))\,.\]
The isomorphism $\rT'(X_1) \cong \rT'(X_2)$ gives an isomorphism of abelian varieties
\[
\mathrm{KS}\bigl(\rT'(X_1)\bigr) \;\cong\; \mathrm{KS}\bigl(\rT'(X_2)\bigr)\,.
\]

For each $i$ the inclusion $\rT'(X_i) \hookrightarrow \rT(X_i)$ induces a natural isogeny
\[
\varphi_i \colon \mathrm{KS}\bigl(\rT'(X_i)\bigr) \longrightarrow \mathrm{KS}\bigl(\rT(X_i)\bigr)\,.
\]
For simplicity, we denote $V=\rT(X_i)$ and $W=\rT'(X_i)$.
Let $r=\operatorname{rk} V$.
After choosing bases, the inclusion $W\subset V$ is represented by an integral nonsingular matrix $A$ with $\det A = d_i$.
As $\mathbb{Z}$-modules, the even Clifford algebra is isomorphic to the even exterior algebra. Hence the kernel of $\varphi_i$ is
\[
\ker\!\Bigl(\mathrm{Cl}^+(W_{\mathbb{R}})/\mathrm{Cl}^+(W)\to\mathrm{Cl}^+(V_{\mathbb{R}})/\mathrm{Cl}^+(V)\Bigr)
\;\cong\;
\ker\!\Bigl(\bigwedge^{\rm even}W_{\mathbb{R}}/\bigwedge^{\rm even}W\to\bigwedge^{\rm even}V_{\mathbb{R}}/\bigwedge^{\rm even}V\Bigr)\,.
\]
Since the kernel on each even exterior power $\bigwedge^{2j}$ has order $(\det A)^{\binom{r-1}{2j-1}}$, the total kernel has order
\[
\prod_{j=1}^{\lfloor r/2\rfloor} (\det A)^{\binom{r-1}{2j-1}} = (\det A)^{2^{r-2}} \,.
\]
Consequently
\[
\deg \varphi_i = (\det A)^{2^{r-2}} = d_i^{2^{r-2}} \,.
\]
Thus $\deg \varphi_i$ is a square, and $\varphi_i$ is a principal isogeny whenever $r \ge 3$ (equivalently $2^{r-2}$ is even).
Composing the isogenies $\varphi_i$ with the isomorphism $\mathrm{KS}(\rT'(X_1)) \cong \mathrm{KS}(\rT'(X_2))$ yields a principal isogeny between $\mathrm{KS}(X_1)$ and $\mathrm{KS}(X_2)$ if the dimension of $\mathrm{KS}(X_i)$ is greater than $1$. Hence Theorem \ref{thm:main-2} implies that the Kuga--Satake varieties are derived isogenous in this case. \qed

\begin{remark}\label{rmk:KSexists}    
Over an arbitrary algebraically closed field $k$, an associated Kuga--Satake variety  $\mathrm{KS}(S)$ can be obtained via the descent theory if its dimension is greater than $2$ (cf.~\cite[Lemma 1.7.1]{An96}).
The result of Corollary \ref{cor:KS-isogeous} also holds for projective K3 surfaces over an algebraically closed field $k$ with $\operatorname{char}(k)=0$, if we restrict to Kuga--Satake varieties of dimension greater than $2$. This dimension condition is equivalent to requiring that the Picard number of the K3 surface is at most $18$.
\end{remark}

\printbibliography

\end{document}